\documentclass[12pt]{amsart}
\usepackage{txfonts}
\usepackage{amscd,amssymb,mathabx}
\usepackage{graphicx,calrsfs}
\usepackage[colorlinks,plainpages,backref,urlcolor=blue,hyperindex,breaklinks]{hyperref}
\usepackage{tikz}
\usetikzlibrary{cd}
\usetikzlibrary{calc}
\usepackage{mdwlist}
\usepackage[shortlabels]{enumitem}

\newtheorem{theorem}{Theorem}[section]
\newtheorem{corollary}[theorem]{Corollary}
\newtheorem{lemma}[theorem]{Lemma}
\newtheorem{proposition}[theorem]{Proposition}

\theoremstyle{definition}

\newtheorem{example}[theorem]{Example}
\newtheorem{remark}[theorem]{Remark}

\newcommand{\N}{\mathbb{N}}
\newcommand{\Z}{\mathbb{Z}}
\newcommand{\Q}{\mathbb{Q}}
\newcommand{\R}{\mathbb{R}}
\newcommand{\C}{\mathbb{C}}

\newcommand{\T}{\mathbb{T}}
\newcommand{\TT}{\mathbb{T}}

\newcommand{\RP}{\mathbb{RP}}

\renewcommand{\L}{\mathbf{L}}

\renewcommand{\k}{\Bbbk}
\newcommand{\RR}{\mathcal{R}}
\newcommand{\V}{\mathcal{V}}
\newcommand{\VV}{\mathcal{V}}
\newcommand{\A}{{\mathcal{A}}}
\newcommand{\ZZ}{{\mathcal{Z}}}

\newcommand{\WW}{\mathcal{W}}
\newcommand{\YY}{{\mathcal{Y}}}

\newcommand{\B}{{\mathfrak{B}}}
\renewcommand{\AA}{{\mathfrak{A}}}
\newcommand{\II}{{\mathfrak{I}}}
\newcommand{\h}{{\mathfrak{h}}}

\DeclareMathOperator{\rank}{rank}
\DeclareMathOperator{\gr}{gr}
\DeclareMathOperator{\im}{im}
\DeclareMathOperator{\coker}{coker}

\DeclareMathOperator{\id}{id}
\DeclareMathOperator{\pr}{pr}
\DeclareMathOperator{\ab}{{ab}}
\DeclareMathOperator{\abf}{{abf}}

\DeclareMathOperator{\Sym}{Sym}

\DeclareMathOperator{\Hom}{{Hom}}

\DeclareMathOperator{\Hilb}{{Hilb}}

\DeclareMathOperator{\ann}{{ann}}

\DeclareMathOperator{\Lie}{Lie}
\DeclareMathOperator{\Spec}{{Spec_m}}
\DeclareMathOperator{\supp}{{supp}}

\DeclareMathOperator{\Aut}{Aut}
\DeclareMathOperator{\Inn}{Inn}
\DeclareMathOperator{\Out}{Out}

\DeclareMathOperator{\Tors}{Tors}
\DeclareMathOperator{\TC}{TC}

\DeclareMathOperator{\Fitt}{Fitt}

\DeclareMathOperator{\rat}{\scalebox{0.5}{$\Q$}}
\DeclareMathOperator{\dash}{\!-\!}

\newcommand{\p}{{\mathfrak p}}
\newcommand{\m}{{\mathfrak m}}

\DeclareMathOperator{\ii}{i}

\newcommand{\surj}{\twoheadrightarrow}
\newcommand{\inj}{\hookrightarrow}
\newcommand{\isom}{\xrightarrow{
   \,\smash{\raisebox{-0.3ex}{\ensuremath{\scriptstyle\simeq}}}\,}}

\newcommand{\abs}[1]{\left| #1 \right|}

\newcommand{\ssqrt}{\!\sqrt}
\newcommand{\vs}{\vskip 2pt}

\def\dot{\mathchar"013A}  
\newcommand{\hdot}{{\raise1pt\hbox to0.35em{\Huge $\dot$}}} 
\newcommand{\bwedge}{\mbox{\normalsize $\bigwedge$}}
\newcommand{\bbwedge}{\mbox{\small $\bigwedge$}}
\definecolor{lime}{HTML}{A6CE39}
\DeclareRobustCommand{\orcidicon}{
	\begin{tikzpicture}
	\draw[lime, fill=lime] (0,0) 
	circle [radius=0.16] 
	node[white] {{\fontfamily{qag}\selectfont \tiny ID}};
	\draw[white, fill=white] (-0.0625,0.095) 
	circle [radius=0.007];
	\end{tikzpicture}
	\hspace{-2mm}
}
\foreach \x in {A, ..., Z}{\expandafter\xdef\csname orcid\x\endcsname{\noexpand\href{%
https://orcid.org/\csname orcidauthor\x\endcsname}{\noexpand\orcidicon}}}

\makeatletter
\@namedef{subjclassname@2020}{%
  \textup{2020} Mathematics Subject Classification}
\makeatother

\begin{document}

\title[Alexander invariants and jump loci in group extensions]{%
Alexander invariants and cohomology jump loci in \\ group extensions} 

\author[Alexander~I.~Suciu]{Alexander~I.~Suciu$^1$\orcidA{}}
\address{Department of Mathematics,
Northeastern University,
Boston, MA 02115, USA}
\email{\href{mailto:a.suciu@northeastern.edu}{a.suciu@northeastern.edu}}
\urladdr{\href{https://web.northeastern.edu/suciu/}%
{web.northeastern.edu/suciu/}}
\thanks{$^1$Supported in part by Simons Foundation Collaboration 
Grants for Mathematicians \#354156 and \#693825}



\begin{abstract}
We study the integral, rational, and modular Alexander invariants, 
as well as the cohomology jump loci of groups arising as extensions 
with trivial algebraic monodromy.  Our focus is on extensions of the form 
$1\to K\to G\to Q\to 1$, where $Q$ is an abelian group acting trivially 
on $H_1(K;\Z)$, with suitable modifications in the rational 
and mod-$p$ settings.  We find a tight relationship between 
the Alexander invariants, the characteristic varieties, and the 
resonance varieties of the groups $K$ and $G$.  This  leads to 
an inequality between the respective Chen ranks, which becomes 
an equality in degrees greater than $1$ for split extensions.
\end{abstract}
\maketitle

\section{Introduction}
\label{sect:intro}

\subsection{Overview}
\label{intro:overview}

The Alexander invariants and the characteristic varieties of 
spaces and groups have their origin in the study of the Alexander 
polynomials of knots and links.  The basic topological idea in defining 
these invariants is to take the homology of the universal abelian cover 
of a cellular complex $X$ and view it as a module over the group ring 
of $H_1(X;\Z)$. One then studies the support loci of these modules, 
or, alternatively, the characteristic varieties, which are the jump 
loci for homology with coefficients in rank $1$ local systems 
on $X$. 

From a group-theoretical point of view, one looks at the 
fundamental group, $G=\pi_1(X)$, and its derived series, 
$G\triangleright G'\triangleright G'' \triangleright \cdots $. Since all the 
succesive quotients of this series are abelian, and since 
 $G_{\ab}=G/G'$ naturally acts on the group $B(G)=G'/G''$, 
one may view $B(G)$ as a module over the group ring $\Z[G_{\ab}]$, 
filtered by the powers of the augmentation ideal. The group $G$ is 
also filtered by its lower central series, 
$G\triangleright [G,G]\triangleright [G,[G,G]]\triangleright \cdots$, 
whose succesive quotients are again abelian, and whose 
associated graded object, $\gr(G)$, inherits 
the structure of a graded Lie algebra. As noted by Massey 
in \cite{Ms-80}, when $G_{\ab}$ is finitely generated, the 
ranks of the associated graded module of $B(G)$ determine the 
Chen ranks of $G$, that is, the ranks of its Chen Lie algebra,  
$\gr(G/G'')$.

The Alexander invariant, $B(G)$, 
determines the characteristic varieties, $\VV_k(G)$, which 
live inside the character group, $\Hom(G,\C^*)$. 
The infinitesimal version of this theory starts from the low-degrees 
cup-product map in group cohomology, 
$\cup_G\colon H^1(G;\Z)\wedge H^1(G;\Z) \to H^2(G;\Z)$, 
and builds from it the infinitesimal Alexander invariant $\B(G)$ 
and the resonance varieties $\RR_k(G)$, which live inside the 
affine space $H^1(G;\C)$. As shown in \cite{DPS-duke}, 
under a formality assumption, there is a strong relationship 
between the modules $B(G)$ and $\B(G)$, on the one hand, 
and the characteristic and resonance varieties (collectively known 
as the cohomology jump loci) of $G$, on the other hand.

We revisit here these classical topics from three main directions. 
First, we broaden the study of the Alexander invariants to include 
their rational and modular versions, based on the maximal torsion-free 
abelian cover and the mod-$p$ congruence covers of a space, 
or, alternatively, on the rational and mod-$p$ variants  of the 
derived series. Second, we analyze the behavior of the 
Alexander-type invariants in short exact sequences of the form 
$1\to K\to G\to Q\to 1$, where the quotient group $Q$ is abelian and the 
sequence stays exact upon abelianization, with suitable modifications in 
the rational and modular settings. Thirdly, we establish a tight relationship 
between the cohomology jump loci and the Chen ranks of the groups 
$G$ and $K$ in such sequences. We proceed now to give a more 
detailed account of the main results of this paper. 

\subsection{Lower central series and derived series}
\label{intro:derived}
Among all the descending series of subgroups associated to a group $G$, 
the most prominent are the lower central series, $\{\gamma_n(G)\}_{n\ge 1}$, 
and the derived series, $\{G^{(r)}\}_{r\ge 0}$. 
Following \cite{BL,Co04,CH05,CH08,CH10,Ha05,St,St83}, 
we also consider the rational and mod-$p$ (where $p$ is a prime) 
versions of these series. All these series start at $G$, and obey 
the following recursion formulas:
\begin{align}
\label{lcs1}
\gamma_{n+1}(G) &=[ G,\gamma_n(G)]
&G^{(r)}&=[G^{(r-1)},G^{(r-1)}]\, ,
\\
\label{lcs2}
\gamma^{\rat}_{n+1}(G) &=\ssqrt{[ G,\gamma^{\rat}_n(G)]}
&G^{(r)}_{\rat}&=\sqrt{ \big[G^{(r-1)}_{\rat},G^{(r-1)}_{\rat}\big]}\, ,
\\
\label{lcs3}
\gamma^{p}_{n+1}(G)&=(\gamma^{p}_{n}(G))^p \big[G,\gamma^{p}_{n}(G)\big]
&G^{(r)}_p&=(G^{(r-1)}_p)^p \big[G^{(r-1)}_p,G^{(r-1)}_p\big]\, .
\end{align}

Here, for a subset $S\subseteq G$, we let $\ssqrt{S}$ 
denote the set of elements $g\in G$ such that $g^m\in S$ for 
some $m>0$, and for a subgroup $H\le G$, we let 
$H^p=\langle g^p \mid g\in H\rangle $ denote 
the subgroup generated by the $p$-powers of elements in $H$ 
(see also \cite{HHR, Leedham-McKay, LR}). Moreover, 
the notation $HK$ in \eqref{lcs3} is shorthand for the subgroup 
of $G$ generated by $H$ and $K$. The successive quotients 
of all these series are abelian groups, which are torsion-free 
in case \eqref{lcs2} and elementary $p$-groups in case \eqref{lcs3}. In particular, 
$G/G'=G_{\ab}$ is the abelianization of $G$, while $G/G'_{\rat}=G_{\abf}$ 
is its maximal torsion-free abelian quotient and $G/G'_p=H_1(G;\Z_p)$ is the 
maximal elementary $p$-abelian quotient of $G$.

The direct sum of the lower central series quotients, $\gr(G)  
= \bigoplus_{n\ge 1} \gamma_n(G)/\gamma_{n+1}(G)$, 
acquires the structure of a graded Lie algebra, with addition 
induced from the group multiplication and Lie bracket 
induced from the group commutator. The associated graded 
Lie algebra $\gr(G)$ is generated by its degree $1$ piece, 
$\gr_1(G)=G_{\ab}$. Thus, if the first Betti number $b_1(G)=\dim H_1(G,\Q)$ 
is finite, the LCS ranks $\phi_n(G)= \dim_{\Q} \gr_n(G)\otimes \Q$ 
are also finite. The graded Lie algebras $\gr^{\rat}(G)$ and $\gr^{p}(G)$ 
and their graded ranks are defined in like manner. Moreover, if $b_1(G)<\infty$, 
then $\phi^{\rat}_n(G)=\phi_n(G)$, and if  $b^p_1(G)=\dim_{\Z_p} H_1(G,\Z_p)$ is finite, 
then so are the mod-$p$ LCS ranks, $\phi_n^p(G)$.

Replacing in this construction the group $G$ by its maximal 
metabelian quotient, $G/G''$, leads to the Chen Lie algebra, 
$\gr(G/G'')$, studied in \cite{Chen51,Markl-Papadima,Ms-80,
PS-imrn04,SW-forum}. We also consider the Lie algebras 
$\gr^{\rat}(G/G''_{\rat})$ and $\gr^p(G/G''_{p})$. The graded 
ranks of all these Lie algebras, $\theta_n(G)=\theta^{\rat}_n(G)$ 
and $\theta^p_n(G)$,  are bounded above by the corresponding 
LCS ranks.

\subsection{Alexander invariants}
\label{intro:alexinv}
We are primarily  interested here in the Alexan\-der-type invariants 
associated with the three kinds of derived series from \eqref{lcs1}--\eqref{lcs3}, 
and how these modules relate to other group invariants. As mentioned 
in \S\ref{intro:overview}, the classical Alexander invariant of $G$ is 
the abelian group 
\begin{equation}
\label{eq:alexinv-intro}
B(G)= G'/G'',
\end{equation} 
viewed as a module over the group-ring $\Z[G_{\ab}]$.  
We introduce and study here two variations thereof: the {\em rational 
Alexander invariant}\/ is the quotient $B_{\rat}(G)= G'_{\rat}/G''_{\rat}$, 
viewed as a module over $\Z[G_{\abf}]$, while the {\em mod-$p$ 
Alexander invariant}\/ is the quotient $B_{p}(G)= G'_{p}/G''_{p}$, 
viewed as a module over $\Z_p[H_1(G;\Z_p)]$. In the process, 
we define analogues of the Crowell exact sequence, relating 
the the Alexander invariant, the Alexander module, and the 
augmentation ideal of the group ring. Important to our analysis 
are the naturality properties of these constructions. For instance, 
every homomorphism $\alpha\colon G\to H$ gives rise 
to a morphism of modules, $B(\alpha)\colon B(G)\to B(H)$, which 
covers the ring map $\tilde\alpha_{\ab}\colon \Z[G_{\ab}]\to\Z[H_{\ab}]$. 
In turn, $B(\alpha)$ factors through a $\Z[G_{\ab}]$-linear map 
from $B(G)$ to the module $B(H)_{\alpha}$ obtained from $B(H)$ 
by restriction of scalars along $\tilde\alpha_{\ab}$.

In \cite{Ms-80}, Massey uncovered a very fruitful connection between 
the Alexander invariant of a group and the lower central series 
of its maximal metabelian quotient. Due to the importance of this 
connection to the further development of the subject, we provide 
in \S\ref{sect:ext} a complete account of the result 
(part \eqref{mm1} of the theorem below), and establish 
rational and modular analogues of Massey's correspondence. 
We summarize these results, as follows.

\begin{theorem}
\label{intro:massey}
For a group $G$, the following hold, for all $n\ge 0$. 
\begin{enumerate}[itemsep=3pt]
\item \label{mm1}
$I^n B(G) = \gamma_{n+2} (G/G'')$, where 
$I$ is the augmentation ideal of $\Z[G_{\ab}]$.
\item \label{mm2}
$I^n (B_{\rat}(G) \otimes \Q) 
= \gamma^{\rat}_{n+2} (G/G''_{\rat}) \otimes \Q$, 
where $I$ is the augmentation ideal of $\Q[G_{\abf}]$.
\item \label{mm3}
$I^n B_p(G) = \gamma_{n+2} (G/G''_p)$, where 
$I$ is the augmentation ideal of $\Z_p[H_1(G;\Z_p)]$.
\end{enumerate} 
\end{theorem}

This theorem proves to be very useful, in a variety of contexts. 
For instance, when $G$ is finitely generated, it allows us to express the 
generating series for the three flavors of Chen ranks in terms of the 
Hilbert series for the corresponding Alexander invariants.

We also relate in this work the integral and rational versions of the Alexander 
invariant. To start with, we show in Proposition \ref{prop:bq-tf} that the inclusion 
$G'\inj G'_{\rat}$ induces a functorial morphism, 
$\kappa\colon B(G)\to B_{\rat}(G)$. When 
$\Tors(G_{\ab})$ is finite, the map $\kappa\otimes \Q\colon 
B(G)\otimes \Q \to B_{\rat}(G)\otimes \Q$ is surjective, 
but in general it is not an isomorphism. Nevertheless, 
if $b_1(G)<\infty$, we prove in Theorem \ref{thm:hat-kappa}, 
that $\kappa\otimes \Q$ induces isomorphisms 
on $I$-adic completions and associated graded modules. 
This theorem partly overlaps with a result from \cite{DHP14}, 
where the group $G'_{\rat}$ is called the Johnson kernel of $G$, 
due to the role it plays in D.~Johnson's study of the 
Torelli group of a surface and of the Johnson homomorphism.

\subsection{Group extensions}
\label{intro:ext}
Our main focus in this paper is on how the  algebraic and 
geometric invariants of groups mentioned above 
behave under group extensions. Given a short 
exact sequence of groups, 
\begin{equation}
\label{intro:exact} 
\begin{tikzcd}[column sep=16pt]
1\ar[r] & K\ar[r, "\iota"] & G \ar[r, "\pi"] & Q\ar[r] & 1 ,
\end{tikzcd}
\end{equation}
we relate---under suitable assumptions---the Alexander invariants, the characteristic 
varieties, the associated graded Lie algebras, and the Chen Lie algebras 
of the group $G$ to those of its subgroup $K$. The assumptions we make are 
tailored to the three versions (integral, rational, and modular) under consideration, 
and are basically of two types. 

One type of constraint is on the group $Q$: we require it to be abelian, 
specializing to torsion-free in the $\Q$-version and elementary 
$p$-abelian in the mod-$p$ version. The other type of constraint 
in on the exactness of the sequence obtained from the given one  
by applying the functors sending $G$ to $G_{\ab}$, $G_{\abf}$, and 
$G_{\ab}\otimes \Z_p$, respectively; when that happens, we say \eqref{intro:exact}  
is an $\ab$-exact, $\abf$-exact, or $p$-exact sequence, respectively. 
When the given sequence splits, these conditions amount to the triviality 
of the action of $Q$ on $K_{\ab}$, $K_{\abf}$ (or $K_{\ab}\otimes \Q$ 
if $K_{\abf}$ is finitely generated), and $K_{\ab}\otimes \Z_p$, respectively. 
In the split case, our study relies on work of \cite{BG,FR,GP,Paris} 
and our recent results from \cite{Su-lcs} regarding the several types of 
lower central series and associated graded Lie algebras of split extensions.  

Examples of (split) $\ab$-exact extensions include the right angled Artin 
groups $G$ associated to finite, connected graphs, with subgroup $K$ the 
corresponding Bestvina--Brady group 
(Examples \ref{ex:raag-bb}, \ref{ex:raag-cv}, \ref{ex:raag-res}) 
or certain Artin kernels (Example \ref{ex:artin-ker}); 
the fundamental groups of complements of arrangements of transverse 
planes through the origin of $\R^4$ 
(Examples \ref{ex:planes}, \ref{ex:high-depth-cv}, \ref{ex:high-depth-res}); 
other types of link groups (Example \ref{ex:en}); 
and the pure braid groups $P_n$ (Example \ref{ex:braids}).

\subsection{Alexander invariants in group extensions}
\label{intro:alex-inv-ext}
We are now in a position to summarize our main results connecting 
the aforementioned algebraic invariants of a group $G$ and a normal 
subgroup $K\triangleleft G$. 
The first result deals with the integral case, and is proved in Theorem 
\ref{thm:alex-abex} and Corollary \ref{cor:alex-abex-split}, 
where more detailed statements can be found.

\begin{theorem}
\label{intro:abex}
Suppose $1\to K\xrightarrow{\iota} G\to Q\to 1$ is an 
$\ab$-exact sequence of groups, and $Q$ is abelian. Then,
\begin{enumerate}[itemsep=2pt]
\item \label{b-z-i}
The induced map on Alexander invariants, 
$B(\iota)\colon B(K) \to B(G)$, factors through a 
$\Z[K_{\ab}]$-linear isomorphism, $B(K) \to B(G)_{\iota}$.
\item  \label{gab-theta}
If $G_{\ab}$ is finitely generated, then 
$\theta_n(K)\le \theta_n(G)$ for all $n\ge 1$.
\item \label{g-z-i}
If the sequence is split exact, then $\iota$ induces 
isomorphisms of graded Lie algebras, $\gr_{\ge 2}(K) \isom \gr_{\ge 2}(G)$ and 
$\gr_{\ge 2}(K/K'')  \isom  \gr_{\ge 2}(G/G'')$. Consequently, if $b_1(G)<\infty$, 
then $\phi_n(K)=\phi_n(G)$ and $\theta_n(K)=\theta_n(G)$ for all $n\ge 2$.
\end{enumerate}
\end{theorem}

When $G=G_{\Gamma}$ is the right-angled Artin group associated to a finite 
simple graph $\Gamma$ and $K=N_{\Gamma}$ is the corresponding 
Bestvina--Brady group from \cite{BB}, the above theorem recovers 
in a slightly stronger form several results from \cite{PS-jlms07}. 
The next result deals with the rational case, and is proved in Theorem 
\ref{thm:alex-lcs-ngq} and Corollary \ref{cor:alex-abfex-split}.

\begin{theorem}
\label{intro:abex-q}
Suppose \eqref{intro:exact} is an $\abf$-exact sequence and 
$Q$ is torsion-free abelian. Then, 
\begin{enumerate}[itemsep=2pt]
\item \label{b-z-ii}
The map $\iota$ induces a $\Z[K_{\abf}]$-linear isomorphism, 
$B_{\rat}(K) \to B_{\rat}(G)_{\iota}$.
\item  \label{gabf-theta}
If $G_{\abf}$ is finitely generated, then 
$\theta_n(K)\le \theta_n(G)$ for all $n\ge 1$.
\item \label{g-z-ii}
If the sequence is split exact, then $\iota$ induces 
isomorphisms of graded Lie algebras, 
$\gr^{\rat}_{\ge 2}(K) \isom \gr^{\rat}_{\ge 2}(G)$ and 
$\gr^{\rat}_{\ge 2}(K/K'')  \isom  \gr^{\rat}_{\ge 2}(G/G'')$. 
Consequently, if $b_1(G)<\infty$, then $\phi_n(K)=\phi_n(G)$ 
and $\theta_n(K)=\theta_n(G)$ for all $n\ge 2$.
\end{enumerate}
\end{theorem}

One instance where this theorem applies is the case when 
$K$ is the fundamental group of a connected CW-complex $X$ 
with $b_1(X)<\infty$ and $G$ is the fundamental group of 
the mapping torus of a map $f\colon X\to X$ inducing the identity 
on $H_1(X;\Q)$; see Corollary \ref{cor:circle-fibration}.

The last result of this type deals with the modular case, 
and is proved in Theorem \ref{thm:alex-lcs-ngp} and 
Corollary \ref{cor:alex-lcs-ngp-split}.

\begin{theorem}
\label{intro:abex-p}
Suppose \eqref{intro:exact} is a $p$-exact sequence and 
$Q$ is an elementary abelian $p$-group. Then  
\begin{enumerate}[itemsep=2pt]
\item \label{b-z}
The map  $\iota$  induces  a $\Z_p[K_{\ab}\otimes \Z_p]$-linear 
isomorphism, $B_{p}(K) \to B_{p}(G)$.
\item  \label{gp-theta}
If $b_1^p(G)<\infty$, then $\theta^p_n(K)\le \theta^p_n(G)$ for all $n\ge 1$.
\item \label{g-z}
If the sequence is split exact, then $\iota$ induces 
isomorphisms of graded Lie algebras, 
$\gr^{p}_{\ge 2}(K) \isom \gr^{p}_{\ge 2}(G)$ and 
$\gr^{p}_{\ge 2}(K/K'')  \isom  \gr^{p}_{\ge 2}(G/G'')$. 
Moreover, if $b_1^p(G)<\infty$, then 
$\phi^{p}_n(K)=\phi^{p}_n(G)$ and $\theta^{p}_n(K)=\theta^{p}_n(G)$ 
for all $n\ge 2$.
\end{enumerate}
\end{theorem}

\subsection{Characteristic varieties}
\label{intro:cvar}
For a finitely generated group $G$, the set of complex-valued 
characters, $\T_G= \Hom(G,\C^{*})$, is a complex algebraic variety, 
with coordinate ring $\C[G_{\ab}]$, and also an abelian group, with 
multiplication given by $(\rho\cdot \rho' )(g)\coloneqq \rho(g)\rho(g')$ 
for characters $\rho,\rho'\colon G\to \C^*$. In fact, since the group 
operations are regular maps, 
$\T_G$ is a complex algebraic group; as such, it is isomorphic to 
$(\C^*)^{r}\times \Tors(G_{\ab})$, where $r=b_1(G)$. This group 
may be thought of as the moduli space of rank $1$ local systems, $\C_{\rho}$, 
on a connected CW-complex $X$ with fundamental group $G$. Taking 
homology with coefficients in such local systems carves out subvarieties 
$\VV_k(G)\subset \T_G$ where the $\C$-vector space $H_1(X,\C_{\rho})$ has 
dimension at least~$k$.

The geometry of these {\em characteristic varieties}\/ is intimately related to the homological 
and finiteness properties of normal subgroups $K\triangleleft G$ with abelian 
quotient $Q=G/K$ and of regular, abelian covers of spaces $X$ with $\pi_1(X)=G$.  
For instance, the stratification of the character group by the varieties 
$\VV_k(G)$ determines the first Betti number of any finite abelian 
cover $Y$ as above, see e.g.~\cite{Li,MS-imrn}.
The characteristic varieties also carry precise information 
regarding the homological and geometric finiteness properties 
of infinite abelian covers, see e.g.~\cite{DPS-imrn, DF, PS-plms10, Su-imrn}, 
and provide powerful obstructions to formality and quasi-projectivity 
of spaces and groups, see e.g.~\cite{DPS-imrn,DPS-duke,PS-formal}.

The characteristic varieties of a group $G$ are controlled by its Alexander 
invariant in a manner that is crucial to our analysis. 
More precisely, $\VV_k(G)$ coincides (at least away from the 
identity character $1$), with the support of the $k$-th exterior power 
of the $\C[G_{\ab}]$-module $B(G)\otimes \C$. Furthermore, letting  
$\WW_k(G)$ be the intersection of $\V_k(G)$ with the 
identity component of the character group, we have that 
\begin{equation}
\label{eq:cvar-intro}
\WW_k(G)=\supp \!\big(\bwedge^k B_{\rat}(G)\otimes \C\big),
\end{equation}
at least away from~$1$. Although results along these lines have been 
known for a long time (see e.g.~\cite{DPS-imrn,Hi97,Li,MS-imrn}), 
there does not appear to be a complete proof in the literature, at least 
not in the generality posited here; therefore, we supply full details in 
Theorems \ref{thm:cvb} and \ref{thm:cvbq}.

\subsection{Characteristic varieties in group extensions}
\label{intro:cv-ext}
We are now in a position to summarize our main results connecting 
the characteristic varieties of groups $G$ and $K$ as above. 

\begin{theorem}
\label{intro:cv-extensions}
Let $\begin{tikzcd}[column sep=14pt]
\!\!1\ar[r] & K\ar[r, "\iota"]
& G \ar[r] & Q\ar[r] & 1\!\!
\end{tikzcd}$ 
be an exact sequence of finitely generated groups. 

\begin{enumerate}
\item \label{int-cv1}
If the sequence is $\ab$-exact and $Q$ is abelian, 
then the map $\iota^*\colon \TT_{G} \to\TT_{K}$ 
restricts to maps $\iota^* \colon \VV_k(G)\to \VV_k(K)$ 
for all $k\ge 1$; furthermore, 
$\iota^* \colon \VV_1(G)\to \VV_1(K)$ is a surjection.

\item \label{int-cv2}
If the sequence is $\abf$-exact and $Q$ is torsion-free abelian, 
then the map $\iota^*\colon \TT^0_{G} \surj \TT^0_{K}$ 
restricts to maps $\iota^* \colon \WW_k(G)\to \WW_k(K)$ 
for all $k\ge 1$; furthermore, 
$\iota^* \colon \WW_1(G)\to \WW_1(K)$ is a surjection.
\end{enumerate}
\end{theorem}

This result is proved in Theorem \ref{thm:cv-abf}.  As a corollary, 
we show the following: If $\VV_1(G)$ is finite (in the first case),  
or $\WW_1(G)$ is finite (in the second case), then the Chen 
ranks $\theta_n(K)$ vanish for all sufficiently large $n$.

In upcoming work \cite{Su-mfmono}, we will apply 
Theorem \ref{intro:cv-extensions} to the study of 
Milnor fibrations of complex hyperplane arrangements. 
Given an arrangement $\A=\{H_1,\dots ,H_n\}$ in $\C^{d}$, 
where each hyperplane $H_i$ is the kernel of a linear form 
$\alpha_i\colon \C^{d} \to \C$, the product of these forms, $f$, 
is a polynomial of degree $n$ in $d$ variables. 
Letting $M=\{f\ne 0\}$ be the complement of the arrangement in $\C^d$, 
the polynomial map $f\colon \C^d\to \C$ restricts to a map $f\colon M\to \C^*$. 
The latter map was shown by Milnor to be a smooth fibration, with fiber $F=\{f=1\}$ 
and monodromy $h\colon F\to F$ given by $h(z)=e^{2\pi \ii/n} z$. 
Setting $G=\pi_1(M)$ and $K=\pi_1(F)$, we may apply parts \eqref{int-cv1} 
or \eqref{int-cv2} of the theorem, under the hypothesis that $h_*$ 
acts trivially on $H_1(K;\Z)$ or $H_1(K;\Q)$, respectively.  Examples from 
\cite{Su-revroum, Su-mfmono} show that the maps $\iota^* \colon \V_k(G)\to\V_k(K)$ 
and $\iota^* \colon \WW_k(G)\to\WW_k(K)$ may fail to be surjective 
for $k>1$, even in this very special context. This phenomenon 
leads to subtle invariants that can distinguish the homotopy types of 
the Milnor fibers of certain arrangements whose complements are 
homotopy equivalent.

\subsection{Holonomy, formality, and resonance}
\label{intro:holo-res}
There are two other important Lie algebras associated to a finitely 
generated group $G$.  The first one is the holonomy Lie algebra, 
$\h(G)$, which was defined by Chen \cite{Chen77} as 
the quotient of the free Lie algebra on $G_{\abf}$ 
by the Lie ideal generated by the image 
of the dual of the cup-product map $\cup_G$. 
This is a quadratic Lie algebra which maps surjectively 
to $\gr(G)$.  The graded ranks of its second derived quotient, 
$\bar\theta_n(G)$---known as the holonomy Chen ranks---are 
bounded below by the usual Chen ranks. 
Following \cite{PS-imrn04}, we use the holonomy Lie algebra  
to construct an infinitesimal version of the Alexander invariant, 
\begin{equation}
\label{eq:infalex-intro}
\B(G)=\h(G)'/\h(G)'',
\end{equation}
which is a graded module over the symmetric 
algebra $\Sym(G_{\abf})$, whose graded ranks coincide with the 
holonomy Chen ranks, after a shift of $2$. 

From a rational homotopy point of view, most important is the Malcev Lie 
algebra, $\m(G)$, defined by Quillen in \cite{Qu} as the (complete, 
filtered) Lie algebra of primitive elements in the $I$-adic completion 
of $\Q[G]$.  The associated graded Lie algebra with respect to this 
filtration, $\gr(\m(G))$, is isomorphic to $\gr(G) \otimes \Q$. 
The group $G$ is said to be graded formal if the 
canonical surjection $\h(G)\otimes \Q \surj \gr(G) \otimes \Q$
is an isomorphism; it is $1$-formal if, in addition, $\m(G)$ 
is isomorphic, as a filtered Lie algebra, to the completion of 
$\gr(G)\otimes \Q$ with respect to the bracket-length filtration.  

The {\em resonance varieties}\/ of $G$ are infinitesimal analogues 
of the characteristic varieties, defined purely in terms of 
cohomological data. More precisely, let 
$H^{\hdot}=H^{\hdot}(G;\C)$ be the cohomology algebra of $G$. 
For each $k\ge 1$, the depth $k$ resonance variety $\RR_k(G)$ 
consists of all elements $a\in H^1$ 
for which there exist $u_1,\dots, u_k \in H^1$ such that $au_i=0$ 
in $H^2$ and $\{a,u_1,\dots, u_k\}$ are linearly independent. 
These sets are homogeneous algebraic subvarieties of the 
affine space $H^1=\C^r$ which are controlled by the 
infinitesimal Alexander invariant; more exactly, 
\begin{equation}
\label{eq:resvar-intro}
\RR_k(G)=\supp \!\big(\bwedge^k \B(G)\otimes \C\big),
\end{equation}
at least away from~$0$. This was proved in \cite{DP-ann, DPS-serre} 
in the case when $G$ is finitely generated; we give in 
Theorem \ref{thm:res-supp} a different proof---%
valid for all groups with $b_1(G)<\infty$---%
based on the BGG correspondence and an infinitesimal 
version of the Crowell exact sequence.
 
The resonance varieties capture deep 
information about qualitative properties and numerical 
invariants of a finitely generated group $G$. When 
compared to the characteristic varieties via the 
Tangent Cone theorem of \cite{DPS-duke}, 
they obstruct $G$ from being $1$-formal, or being realizable 
as the fundamental group of a quasi-projective or K\"{a}hler manifold. 
In favorable situations, they do allow for the computation of the 
Chen ranks $\theta_n(G)$ in terms of the dimensions of the 
irreducible components of $\RR_1(G)$; see 
\cite{AFPRW1, AFRS, CSc-adv, Su-conm}. 
Finally, they have applications to the study of homological finiteness 
properties in the Johnson filtration of mapping class groups and 
automorphism groups of free groups, and to Green’s conjecture 
on free resolutions of canonical curves; see 
\cite{ AFPRW2, DHP14, DP-ann,  PS-jtop,  PS-crelle}.

\subsection{Resonance varieties in extensions}
\label{intro:ext-res}

Our final result relates the resonance varieties of a group $G$ to 
those of a normal subgroup $K\triangleleft G$, under suitable 
assumptions on the quotient $Q=G/K$ and its action on the 
first homology of $K$, as well as on the formality properties 
of $G$ and $K$. In Theorems \ref{thm:res-split-abf} 
and \ref{thm:res-abf}, we establish the following result, 
which is one of the main results of this paper.

\begin{theorem}
\label{intro:res}
Let $\begin{tikzcd}[column sep=14pt]
\!\!1\ar[r] & K\ar[r, "\iota"]
& G \ar[r] & Q\ar[r] & 1\!\!
\end{tikzcd}$ be an exact sequence of finitely generated groups. 
Suppose that either one of the following set of conditions is satisfied.
\begin{enumerate}[itemsep=1pt]
\item \label{res1}
The sequence is split exact, $G$ is graded formal, $Q$ is abelian, 
and $Q$ acts trivially on $H_1(K;\Q)$.
\item \label{res2}
The sequence if $\ab$-exact, $G$ and $K$ are $1$-formal, and  
$Q$ is abelian. 
\item \label{res3}
The sequence if $\abf$-exact, $G$ and $K$ are $1$-formal, and 
$Q$ is torsion-free abelian. 
\end{enumerate}
Then the homomorphism $\iota^*\colon H^1(G, \C) \surj H^1(K, \C)$ 
restricts to maps $\iota^* \colon \RR_k(G)\surj \RR_k(K)$ for all $k\ge 1$; 
furthermore, the map $\iota^* \colon \RR_1(G)\surj \RR_1(K)$ is a 
surjection.
\end{theorem}

In particular, if any one of the above conditions is satisfied and if 
$\RR_1(G)\subseteq 0$, then we infer from the above theorem 
that $\RR_1(K)\subseteq 0$, which in turn implies that the 
$\Q$-vector space $\B(K)\otimes \Q$ is finite-dimensional. 
Under the $1$-formality assumptions from either \eqref{res2} or 
\eqref{res3}, we conclude that $\theta_n(G)$ and $\theta_n(K)$ 
vanish for sufficiently large $n$.

When $G=G_{\Gamma}$ and $K=N_{\Gamma}$ are the right-angled 
Artin group and the Bestvina--Brady group associated to a connected 
graph $\Gamma$, the theorem recovers one of the main results 
from \cite{PS-jlms07} and extends it further.
In \cite{Su-mfmono} we will apply Theorem \ref{intro:res} in 
the case when $G$ is a complex arrangement group, $K$ 
is the group of its Milnor fiber, and the monodromy of the 
Milnor fibration acts trivially on $H_1(K;\Q)$. 

\subsection{Organization}
\label{intro:organize}
The paper is organized in four parts of roughly equal length. 

Part \ref{part:alex} deals with the Alexander invariants, the derived series, 
and the lower central series of a group. We concentrate 
in \S\ref{sect:alexinv} on the classical derived series and the 
associated Alexander invariant. In 
\S\ref{sect:derived-alex-q} we turn to the rational version of these objects, 
while in \S\ref{sect:derived-alex-p} we study the mod-$p$ version. 
We conclude in \S\ref{sect:grg} with a quick review of the corresponding 
lower central series, associated graded Lie algebras, and Chen Lie algebras.

Part \ref{part:ext}  explores group extensions of the form $1\to K\to G\to Q\to 1$, 
focussing on the relationship between the Alexander invariants and the lower 
central series of those groups.
We start in \S\ref{sect:ext} with  Massey's correspondence between the filtration 
of the Alexander invariant of $K$ by powers of the augmentation ideal of $Q$ 
and the lower central series of the maximal metabelian quotient of $G$. 
We continue in \S\ref{sect:split} with an overview of the lower central 
series and associated graded algebras of split extensions of groups.  
In \S\ref{sect:ab-exact} we explore ab-exact sequences, 
and establish our results on the way the Alexander invariants and 
the Chen ranks behave under such extensions. 
We prove analogous results for abf-exact sequences 
in \S\ref{sect:abf-exact} and for $p$-exact sequences in 
\S\ref{sect:p-exact}.

Parts \ref{part:cjl} and \ref{part:resonance} contain a detailed study 
of the cohomology jump loci of a finitely generated group, and the way 
these loci behave under the aforementioned types of group extensions. 
We start in \S\ref{sect:cvs} with the 
characteristic varieties, and continue in \S\ref{sect:alex-vars} with 
the Alexander varieties, focussing on the relationship between the two. 
In \S\ref{sect:cv-ext} we establish our structural results relating the 
characteristic varieties of a group $G$ to those of a normal subgroup $K$ 
under appropriate assumptions. After discussing in \S\ref{sect:formal} 
the Malcev and holonomy Lie algebras of a group and the resulting notion 
of $1$-formality, we provide in \S\ref{sect:inf-alex} and \S\ref{sect:res} 
detailed information on the infinitesimal Alexander invariant and the 
resonance varieties. Finally, we establish in \S\ref{sect:res-ext} our 
results on the way resonance behaves in ab- and abf-exact group extensions, 
under suitable formality hypothesis. 

\part{Alexander invariants}
\label{part:alex}

\section{Derived series and the Alexander invariant}
\label{sect:alexinv}

We start with a review of the derived series and the Alexander invariant 
of a group, and discuss some of their basic properties.

\subsection{Derived series}
\label{subsec:derived}

Let $G$ be a group. If $H$ and $K$ are subgroups of $G$, then 
$[H,K]$ denotes the subgroup of $G$ generated by all elements 
of the form $[a,b]=aba^{-1}b^{-1}$ with $a \in H$ and $b \in K$. 
If both $H$ and $K$ are normal subgroups, then their commutator 
$[H,K]$ is again a normal subgroup.  Moreover, if 
$\alpha \colon G\to H$ is a homomorphism, then 
\begin{equation}
\label{eq:comm}
\alpha([H,K])\subseteq [\alpha(H),\alpha(K)]\, .
\end{equation}

The {\em derived series}\/ of $G$, denoted $\{G^{(r)}\}_{r\ge 0}$, is 
defined inductively by \eqref{lcs1}; 
in particular, $G^{(1)}=G'$ is the derived subgroup and  $G^{(2)}=G''$. 
Using \eqref{eq:comm} and 
induction on $r$, it is readily seen that the terms of the derived series are 
fully invariant subgroups; that is, if $\alpha \colon G\to H$ is a group 
homomorphism, then $\alpha(G^{(r)})\subseteq H^{(r)}$, for all $r$. 
Consequently, the derived series is a normal series, i.e., $G^{(r)}\triangleleft G$, 
for all $r$. Moreover, since  $G^{(r-1)}/G^{(r)}=( G^{(r-1)} )_{\ab}$, 
all the successive quotients of the series are abelian groups. 

A group $G$ is said to be {\em solvable}\/ if its derived series of $G$ terminates 
in finitely many steps; that is, $G^{(\ell)}=\{1\}$ for some integer $\ell\ge 0$.  
The smallest such integer, $\ell(G)$, is then called the derived length of $G$. 
Clearly, $\ell(G)\le 1$ if and only if $G$ is abelian, while 
$\ell(G)\le 2$  if and only if $G$ is metabelian.  The maximal  solvable quotient of 
$G$ of length $r$ is $G/G^{(r)}$; in particular, the maximal metabelian quotient 
is $G/G''$. 

\subsection{Alexander invariant and Alexander module}
\label{subsec:alexinv}
Among the successive quotients of the derives series of a group $G$, the second one 
plays a special role.  The {\em Alexander invariant}\/ of $G$ is the abelian group 
\begin{equation}
\label{eq:gprime}
B(G)\coloneqq   G'/G''\, , 
\end{equation}
viewed as a module over the group-ring $\Z[G_{\ab}]$; alternatively, 
$B(G)=G'_{\ab}=H_1(G'_{\ab};\Z)$.  Addition in $B(G)$ 
is induced from multiplication in $G$, to wit, $(xG'')+(yG'')=xy G''$ for all 
$x,y\in G'$, while scalar multiplication is induced from conjugation in the 
maximal metabelian quotient, $G/G''$, via the exact sequence 
\begin{equation}
\label{eq:gprimeprime}
\begin{tikzcd}
1\ar[r]& G'/G'' \ar[r]& G/G'' \ar[r]& G/G' \ar[r]& 1\, .
\end{tikzcd} 
\end{equation}
That is, $gG'\cdot xG'' = gxg^{-1}G''$ for all $g\in G$ and $x\in G'$, 
with the action of $G/G'=G_{\ab}$ 
extended $\Z$-linearly to the whole of $\Z[G_{\ab}]$.

The augmentation map, $\varepsilon\colon \Z[G]\to \Z$, is the linear 
extension to group rings of the trivial homomorphism, $G\to \{1\}$;  
let $I(G)=\ker(\varepsilon)$ be the augmentation ideal. Closely related 
to the Alexander invariant is the {\em Alexander module}\/ of $G$, 
\begin{equation}
\label{eq:alex-mod}
A(G)=\Z[G_{\ab}]\otimes_{\Z[G]} I(G) \, ,
\end{equation}
with $\Z[G_{\ab}]$-module structure coming from multiplication 
on the left factor. 

In order to better understand the $\Z[G_{\ab}]$-modules $A(G)$ and $B(G)$, 
we will look at them next from a topological point of view, 
following the approach of Massey from \cite{Ms-80}.

\subsection{Topological interpretation}
\label{subsec:alexinv-top}

Let $X$ be a connected CW-complex with $\pi_1(X,x_0)=G$. 
(We may assume $X$ has a single $0$-cell, which we then take as the 
basepoint $x_0$.)  Lifting the cell structure of $X$ to the maximal abelian cover, 
$q\colon X^{\ab}\to X$, we obtain an augmented chain complex of free 
$\Z[G_{\ab}]$-modules, 
\begin{equation}
\label{eq:abcover-cc}
\begin{tikzcd}[column sep=18pt]
\cdots \ar[r] & 
C_{2}(X^{\ab};\Z) \ar[r, "\partial^{\ab}_{2}"] &[8pt] 
 C_{1}(X^{\ab};\Z) \ar[r, "\partial^{\ab}_{1}"] &[8pt] 
  C_{0}(X^{\ab};\Z) \ar[r, "\varepsilon"] & \Z  \ar[r] &0, 
\end{tikzcd}
\end{equation}
where $ C_{0}(X^{\ab};\Z) =\Z[G_{\ab}]$ and 
$\varepsilon$ is the augmentation map.
Since $\pi_1(X^{\ab})=G'$, the Alexander invariant $B(G)=(G')_{\ab}$ 
is isomorphic to $H_1(X^{\ab};\Z)$, the first homology group of the chain 
complex \eqref{eq:abcover-cc}, with module structure induced 
by the action of $G_{\ab}$ by deck transformations. In other 
words, $B(G)=H_1(X;\Z[G_{\ab}])$.

\begin{example}
\label{ex:free}
Let $X=\bigvee^n S^1$ be a wedge of $n$ circles. Identify $\pi_1(X)$ 
with the free group $F_n=\langle x_1,\dots ,x_n\rangle$ and $(F_n)_{\ab}$ 
with $\Z^n$. The chain complex $(C_i((T^n)^{\ab};\Z),\partial_i^{\ab})$ 
of the universal (abelian) cover of the $n$-torus $T^n=K(\Z^n,1)$ 
may be viewed as the Koszul complex on elements 
$t_1-1, \dots, t_n-1$ over the ring $\Z[\Z^n]=\Z[t_1^{\pm 1}, \dots , t_n^{\pm 1}]$. 
The Alexander invariant $B(F_n)$, then, equals the $\Z[\Z^n]$-module 
$\coker (\partial_3^{\ab})$; in particular, $B(F_2)=\Z[\Z^2]$, while $B(F_3)=
\coker\big( \!\begin{pmatrix}1-t_3 & t_2-1 &1-t_1\end{pmatrix}\!\colon 
\Z[\Z^3]\to \Z[\Z^3]^3 \big)$.
\end{example}

Consider now the fiber $F=q^{-1}(x_0)$, and fix a basepoint $\tilde{x}_0\in  F$. 
The homology long exact sequence for the pair $(X^{\ab},F)$ yields an 
exact sequence of $\Z[G_{\ab}]$-modules,
\begin{equation}
\label{eq:homology-pair}
\begin{tikzcd}[column sep=18pt]
0\ar[r]& H_1(X^{\ab}; \Z) \ar[r]&H_1(X^{\ab},F; \Z) 
\ar[r]& H_0(F; \Z) \ar[r]& H_0(X^{\ab}; \Z)  \ar[r] & 0 \, .
\end{tikzcd}
\end{equation}
As noted previously, the first term is the Alexander invariant $B(G)$. 
The second term is the Alexander module $A(G)$; indeed, sending an  
element $g-1\in I(G)$ to the path in $X^{\ab}$ from 
$\tilde{x}_0$ to $g\tilde{x}_0$ obtained by lifting the loop $g$ at 
$\tilde{x}_0$ induces an isomorphism 
$A(G)\isom H_1(X^{\ab},F; \Z)$. Finally, 
the homomorphism $H_0(F;\Z)\to H_0(X^{\ab};\Z)$ 
may be identified with the augmentation map, 
$\varepsilon \colon \Z[G_{\ab}]\to \Z$. Therefore, 
\begin{equation}
\label{eq:alexmod}
A(G)=\coker(\partial_2^{\ab})\, ,
\end{equation}
and the sequence \eqref{eq:homology-pair} yields 
an exact sequence of $\Z[G_{\ab}]$-modules,
\begin{equation}
\label{eq:crowell}
\begin{tikzcd}[column sep=18pt]
0\ar[r]& B(G)\ar[r]&A(G)\ar[r]&I(G_{\ab})\ar[r]& 0 \, ,
\end{tikzcd}
\end{equation}
known as the {\em Crowell exact sequence}\/ of the group, cf.~\cite{Cr65,Ms-80}.
When $G_{\ab}$ is finitely generated, the ring $\Z[G_{\ab}]$ is Noetherian. 
In this case, the Alexander module $A(G)$ is finitely generated, and hence 
the presentation \eqref{eq:alexmod} may be reduced to a 
finite presentation. Thus, by \eqref{eq:crowell}, the Alexander 
invariant $B(G)$ may also be finitely presented. 

If $G$ admits a finite presentation, say, 
$G=\langle x_1,\dots , x_m\mid 
r_1, \dots , r_{\ell}\rangle$, the $\Z[G_{\ab}]$-linear map
 $\partial_2^{\ab}\colon \Z[G_{\ab}]^{\ell} \to \Z[G_{\ab}]^{m}$ 
from \eqref{eq:abcover-cc} may be identified with the classical 
Alexander matrix, whose entries are the abelianized Fox derivatives 
of the relators,  $\ab(\partial r_i/\partial x_j)$; hence, the module $A(G)$ 
is presented by the Alexander matrix. When $G_{\ab}$ is torsion-free, 
a method for finding a presentation for $B(G)$ is outlined in 
\cite{Ms-80}; an explicit presentation is not known 
even in the case when $G$ is a link group, but 
there is an algorithm for producing such a presentation in 
the case when $G$ is an arrangement group, see \cite{CS-tams99}.

\subsection{Functoriality properties}
\label{subsec:alexinv-func}
The assignments $G\leadsto G_{\ab}$ and $G\leadsto B(G)$ are functorial 
and compatible with one another, in a sense that we now make precise. 
For more background information on some of this material, we refer 
to \cite[Ch.~III]{Brown}.

First consider a ring map, $\varphi\colon R\to S$.  We say that a 
map $\psi\colon M\to N$ from an $R$-module $M$ to an $S$-module $N$ 
{\em covers}\/ $\varphi$ (or, for short, that $\psi$ is a $\varphi$-morphism) 
if $\psi(rm) = \varphi(r) \psi(m)$ for all $r\in R$ and $m\in M$. 
Such a map $\psi$ can be viewed as the composite
\begin{equation}
\label{eq:cover-factor}
\begin{tikzcd}[column sep=20pt]
M \ar[r] & N_{\varphi} \ar[r] &N \, ,
\end{tikzcd}
\end{equation}
where $N_{\varphi}$ is the $R$-module obtained from $N$ by 
restriction of scalars via $\varphi$, the first arrow is the set map 
$\psi$ viewed as an $R$-linear map, and the second arrow is the 
identity map of $N$, thought of as covering the ring map $\varphi$.  

Now let $\alpha\colon G\to H$ be a group homomorphism. 
Then $\alpha$ extends linearly to a ring map,  
$\tilde\alpha \colon \Z [G]\to \Z [H]$. The assignment 
$G\leadsto \Z[G]$, $\alpha\leadsto \tilde\alpha$ is 
functorial, and takes injections to injections and 
surjections to surjections. 

The map $\alpha$ also restricts 
to homomorphisms $ G'\to H'$ and $G''\to H''$, and thus induces 
homomorphisms $G/G'\to H/H'$ and $G'/G''\to H'/H''$, which we will denote 
by $\alpha_{\ab}\colon G_{\ab}\to H_{\ab}$ and $B(\alpha) \colon 
B(G) \to B(H)$, respectively.  If $\beta\colon H\to K$ is another homomorphism, 
then clearly $\beta_{\ab}\circ \alpha_{\ab}=(\beta\circ \alpha)_{\ab}$ and 
$B(\beta)\circ B(\alpha)=B(\beta\circ \alpha)$. If $\alpha$ is surjective, 
then $B(\alpha)$ is also surjective, but if $\alpha$ is injective, 
$B(\alpha)$ need not be injective.  

\begin{example} 
\label{ex:trefoil}
Let $G=\langle x_1,x_2\mid x_1x_2x_1=x_2x_1x_2\rangle$, so that $G_{\ab}=\Z$. 
We then have $G'=F_2$ and $B(G')=\Z[\Z^2]$, whereas $B(G)=\Z[t^{\pm 1}]/(t^2-t+1)$. 
Thus, if $\alpha\colon G'\inj G$ is the inclusion, the map $B(\alpha)$ 
is not injective.
\end{example}

Given a homomorphism $\alpha\colon G\to H$, 
let $\tilde\alpha_{\ab} \colon \Z [G_{\ab}]\to \Z [H_{\ab}]$ be the 
linear extension of $\alpha_{\ab}$ to group rings. 
The map $B(\alpha) \colon B(G) \to B(H)$ can then be interpreted as 
a map of modules covering $\tilde\alpha_{\ab}$. 
Alternatively, let $B(H)_{\alpha}$ be the $\Z [G_{\ab}]$-module obtained 
from $B(H)$ by restriction of scalars via $\tilde\alpha_{\ab}$. 
The map $B(\alpha)$ can then be viewed as the composite 
$B(G)\to  B(H)_{\alpha} \to B(H)$, where the first arrow is 
a $\Z [G_{\ab}]$-linear map and the second arrow is the identity 
map of $B(H)$, viewed as covering the ring map $\tilde\alpha_{\ab}$.  

Here is a topological interpretation. Let $f\colon X\to Y$ 
be a continuous maps between connected CW-complexes; without loss 
of essential generality, we may assume $f$ is cellular and 
basepoint-preserving. Let $f_{\sharp}\colon \pi_1(X, x_0)\to \pi_1(Y,y_0)$ 
be the induced homomorphism on fundamental groups, and let 
$f^{\ab}\colon X^{\ab} \to Y^{\ab}$ be the lift to universal abelian 
covers. It is readily seen that the morphism 
$B(f_{\sharp})\colon B(\pi_1(X))\to B(\pi_1(Y))$ 
coincides with the induced homomorphism in first homology, 
$f_*\colon H_1(X^{\ab};\Z) \to H_1(Y^{\ab};\Z)$.

Likewise, there is an induced morphism on Alexander modules, 
$A(\alpha) \colon A(G) \to A(H)$, which covers $\tilde\alpha_{\ab}$ 
and admits a similar topological interpretation. The 
restriction of $A(\alpha)$ to $B(G)$ coincides with $B(\alpha)$, 
and induces the map $\tilde\alpha_{\ab}$ on augmentation ideals, 
thereby showing that the Crowell exact sequence \eqref{eq:crowell} 
is natural with respect to group homomorphisms. 

\section{The rational derived series and the rational Alexander invariant}
\label{sect:derived-alex-q}

In this section we discuss the rational versions of the derived series 
and of the Alexander invariant.

\subsection{The rational derived series}
\label{subsec:derived-q}
For a subset $S$ of a group $G$, we let 
$\ssqrt{S}=\{g\in G\mid \text{$g^m \in S$ for some $m\in \N$} \}$ 
be its {\em isolator}. 
Clearly, $S\subseteq \ssqrt{S}$ and 
$\ssqrt{\!\smash[b]{\!\sqrt{S}}}= \ssqrt{S}$; moreover, if 
$\alpha\colon G\to H$ is a homomorphism, then 
$\alpha(\!\sqrt{S})\subseteq \ssqrt{\alpha(S)}$. 
The isolator of a subgroup need not be a subgroup; for instance, 
$\ssqrt{\{1\}}=\Tors(G)$, the set of torsion elements 
in $G$, which is not a subgroup in general. On the 
other hand, if $G$ is nilpotent, then the isolator of 
any subgroup of $G$ is again a subgroup. The theory 
of isolators was initiated by P.~Hall \cite{Hall} in the 
context of $p$-groups; we refer to \cite[\S 2.3]{LR} for 
more on this subject.

The following notion was introduced by Harvey in \cite{Ha05}, and 
further studied by Cochran and Harvey in \cite{Co04,CH05}. 
The {\em rational derived series}\/ of a group $G$, denoted 
$\{G^{(r)}_{\rat}\}_{r\ge 0}$, is defined inductively 
by setting $G^{(0)}_{\rat}=G$ and 
\begin{equation}
\label{eq:gq-series}
G^{(r)}_{\rat}=\sqrt{ \big[G^{(r-1)}_{\rat},G^{(r-1)}_{\rat}\big]}.
\end{equation}
Using induction on $r$ and the fact that homomorphisms preserve 
commutators and isolators, it is readily seen that the terms of this 
series are fully invariant subgroups. 
In particular, the rational derived series is a normal series, a property also 
noted in \cite[Lemma 3.2]{Ha05}. Furthermore, the successive quotients, 
$G^{(r)}_{\rat}/G^{(r+1)}_{\rat}$, are torsion-free abelian groups; 
in fact, as shown in  \cite[Lemma 3.5]{Ha05}, 
\begin{equation}
\label{eq:gq-quotients}
G^{(r)}_{\rat}/G^{(r+1)}_{\rat} \cong \big(G^{(r)}_{\rat}\big)_{\abf} \, .
\end{equation}

In particular, $G/G'_{\rat}$ is equal to $G_{\abf}=G_{\ab}/\Tors(G_{\ab})$, 
the maximal torsion-free abelian quotient of $G$, showing that 
$G'_{\rat}$ is the kernel of the projection map $\abf\colon G\surj G_{\abf}$. 
Since $G'=\ker(\ab\colon G\surj G_{\ab})$, we obtain a short exact sequence, 
\begin{equation}
\begin{tikzcd}[column sep=18pt]
1\ar[r]& G' \ar[r]& G'_{\rat} \ar[r]& \Tors(G_{\ab}) \ar[r]& 1 .
\end{tikzcd}
\end{equation}
In particular, if $G_{\ab}$ is torsion-free, then $G'=G'_{\rat}$.  

\begin{remark}
\label{rem:johnson}
In \cite{DHP14,DP-ann}, the group $G'_{\rat}=\ker (\abf\colon G\surj G_{\abf})$ 
is called the {\em Johnson kernel}\/ of $G$, and is denoted by $K_G$. The 
motivation for this terminology is that, in the case when $G=\mathcal{T}_g$ 
is the Torelli group of a surface of genus $g\ge 3$, the subgroup 
$K_G\triangleleft \mathcal{T}_g$ is the subgroup generated 
by Dehn twists along separating simple closed curves---a subgroup 
introduced and studied by D.~Johnson in the 1980s; see \cite{CEP} 
for more on this.
\end{remark}

It is also known that the quotient groups $G/G^{(r+1)}_{\rat}$ are 
poly-torsion-free-abelian (PTFA) groups, see \cite[Corollary 3.6]{Ha05}.
Cleary, $G^{(r)}\subseteq G^{(r)}_{\rat}$ for all $r$, but the inclusions 
are strict in general. 
Nevertheless, if all the quotients $G^{(r)}/G^{(r+1)}$ are torsion-free---%
which is the case when $G$ is a finitely generated free group, or a knot group---%
then $G^{(r)}=G^{(r)}_{\rat}$ for all $r$, cf.~ \cite[Corollary 3.7]{Ha05}.  

\subsection{The rational Alexander invariant}
\label{subsec:alexinv-q}
By analogy with the classical definition, we define the 
{\em rational Alexander invariant}\/ of a group $G$ to be the quotient 
\begin{equation}
\label{eq:alex-q}
B_{\rat}(G)\coloneqq G'_{\rat}/G''_{\rat}\, ,  
\end{equation}
viewed as a module over $\Z[G_{\abf}]$, where recall $G'_{\rat}=\ssqrt{G'}$
and $G''_{\rat}=\ssqrt{\left[\ssqrt{G'},\ssqrt{G'}\right]}\triangleleft G'_{\rat}$. 
The module structure on $B_{\rat}(G)$ is induced 
by conjugation in the maximal torsion-free metabelian quotient, $G/G''_{\rat}$, 
via the exact sequence 
\begin{equation}
\label{eq:gprime-rat}
\begin{tikzcd}
1\ar[r]& G'_{\rat}/G''_{\rat} \ar[r]& G/G''_{\rat} \ar[r]& G/G' _{\rat}\ar[r]& 1 .
\end{tikzcd} 
\end{equation}
That is, $gG'_{\rat}\cdot xG''_{\rat} = gxg^{-1}G''_{\rat}$ for all 
$g\in G$ and $x\in G'_{\rat}$, with the action of $G/G' _{\rat}=G_{\abf}$ 
extended $\Z$-linearly to the whole of $\Z[G_{\abf}]$. 

Given a group homomorphism $\alpha\colon G\to H$, the maps 
$\alpha'\colon G'_{\rat}\to H'_{\rat}$ and $\alpha''\colon G''_{\rat}\to H''_{\rat}$
induce a morphism $B_{\rat}(\alpha)\colon B_{\rat}(G)\to B_{\rat}(H)$ 
between rational Alexander invariants. 
It is readily seen that the assignments $G\leadsto G_{\abf}$ 
and $G\leadsto B_{\rat}(G)$ are functorial and compatible 
with one another, in a sense similar to the one described 
in \S\ref{subsec:alexinv}. 

Pursuing our analogy with the classical situation, we also define the 
{\em rational Alexander module}\/ of $G$ to be the  $\Z[G_{\abf}]$-module 
$A_{\rat}(G)\coloneqq \Z[G_{\abf}]\otimes_{\Z[G]} I(G)$. 

\subsection{Topological interpretation}
\label{subsec:alexinv-top-q}
Let $X$ be a connected CW-complex with $\pi_1(X)=G$, 
and let $q_0\colon X^{\abf}\to X$ be the maximal torsion-free 
abelian cover of $X$.  We then have a commuting 
diagram of regular covers,
\begin{equation}
\label{eq:ab-abf-covers}
\begin{tikzcd}[column sep=10pt, row sep=10pt, ]
X^{\ab} \ar[dr, "s"] \ar[dd, "q"]\\
& X^{\abf} \ar[dl, "q_0"] \\
X
\end{tikzcd} 
\end{equation}
where $s\colon X^{\ab}\to X^{\abf}$ is an abelian cover 
with deck group $\Tors(G_{\ab})$. Clearly, if $G_{\ab}$ is finitely 
generated, then $s$ is a finite cover. Note that $H_{\hdot}(X^{\abf};\Z)$ is 
a module over $\Z[G_{\abf}]$, with module structure induced by the 
action of $G_{\abf}$ on $X^{\abf}$  by deck transformations. 

\begin{lemma}
\label{lem:bq}
With notation as above, 
\begin{enumerate}
\item \label{b1}
$B_{\rat}(G)\cong H_1(X^{\abf};\Z)/\Z\dash\Tors$, 
as $\Z[G_{\abf}]$-modules.
\item \label{b2}
$B_{\rat}(G)\otimes \Q\cong H_1(X^{\abf};\Q)$, 
as $\Q[G_{\abf}]$-modules.
\item  \label{b3}
If $G_{\ab}$ is torsion-free, then $B_{\rat}(G)\cong B(G)/\Z\dash\Tors$, 
as $\Z[G_{\ab}]$-modules.
\end{enumerate}
\end{lemma}

\begin{proof}
From definition \eqref{eq:alex-q} and formula \eqref{eq:gq-quotients}, 
we have that $B_{\rat}(G)=\big(G'_{\rat}\big)_{\abf}$. Since $G'_{\rat}=\pi_1(X^{\abf})$, 
the first claim follows. Tensoring both sides with $\Q$ yields the second claim.  
The third claim follows at once from the first one.
\end{proof}

\begin{remark}
\label{rem:bbq-map}
In view of Lemma \ref{lem:bq}, part \eqref{b1} and the discussion from 
Remark \ref{rem:johnson}, the rational Alexander invariant $B_{\rat}(G)$ 
may be viewed as the torsion-free abelianization of the Johnson kernel $K_G$. 
In a related vein, we considered in \cite{PS-crelle} the $\Q[G_{\abf}]$-module 
$\tilde{B}(G)\coloneqq H_1(G_{\rat};\Q)$, which was called there 
the reduced Alexander invariant of $G$. In view of 
Lemma \ref{lem:bq}, part \eqref{b2}, this module is 
isomorphic to $B_{\rat}(G)\otimes \Q$.
\end{remark}

Let $(C_{\hdot}(X^{\ab};\Z), \partial^{\ab})$ be the $\Z[G_{\ab}]$-equivariant 
chain complex of $X^{\ab}$ from \eqref{eq:abcover-cc}. 
We denote by $\nu\colon G_{\ab}\surj G_{\abf}$ the projection map, 
and we let $\tilde\nu\colon \Z[G_{\ab}]\surj \Z[G_{\abf}]$ be its 
linear extension to group rings. The $\Z[G_{\abf}]$-equivariant 
chain complex of $X^{\abf}$, denoted $(C_{\hdot}(X^{\abf};\Z), \partial^{\abf})$, 
may be obtained from \eqref{eq:abcover-cc} by extension of scalars, via 
the ring map $\tilde\nu$; in particular, 
$\partial^{\abf}=\partial^{\ab} \otimes_{\Z[G_{\ab}]} \Z[G_{\abf}]$.

Let also $F_0=q^{-1}_0(x_0)$ be the fiber of the cover $q_0\colon X^{\abf}\to X$ 
over a basepoint $x_0\in X$, and let $I_0(G)$ be the kernel of the augmentation 
map $\varepsilon \colon \Z[G_{\abf}] \to \Z$. Proceeding as in \S\ref{subsec:alexinv-top}, 
we obtain the following lemma, which summarizes the properties of the rational 
Alexander module, and its relationship to the rational Alexander invariant.

\begin{lemma}
\label{lem:alexmod-q}
With notation as above,
\begin{enumerate}[itemsep=2pt]
\item \label{amq1}
$A_{\rat}(G)\otimes \Q= H_1(X^{\abf},F_0;\Q)$. 
\item \label{amq2}
$A_{\rat}(G) \otimes \Q$ is the cokernel of the map 
$\partial_2^{\abf}\otimes \Q\colon C_2(X^{\abf};\Q)\to C_1(X^{\abf};\Q)$.
\item \label{amq3}
The homology exact sequence of the pair $(X^{\abf},F_0)$ 
yields a short exact sequence of $\Q[G_{\abf}]$-modules,
\begin{equation}
\label{eq:crowell-rat}
\begin{tikzcd}[column sep=18pt]
0\ar[r]& B_{\rat}(G)\otimes \Q \ar[r]&A_{\rat}(G)\otimes \Q
\ar[r]&I_{0}(G)\otimes \Q \ar[r]& 0 \, .
\end{tikzcd}
\end{equation}
\end{enumerate}
\end{lemma}
Sequence \eqref{eq:crowell-rat} may be viewed as the rational analogue 
of Crowell's exact sequence \eqref{eq:crowell}, and enjoys a 
similar naturality property with respect to group homomorphisms.

\subsection{Relating the Alexander invariants}
\label{subsec:relate-alex}
We conclude this section with a comparison between the two kinds of 
Alexander invariants defined so far: $B(G)$, viewed as a $\Z[G_{\ab}]$-module, 
and $B_{\rat}(G)$, viewed as a $\Z[G_{\abf}]$-module. The comparison is done 
via the natural projection $\nu\colon G_{\ab}\surj G_{\abf}$ and its extension 
to a ring map, $\tilde\nu\colon \Z[G_{\ab}]\surj \Z[G_{\abf}]$.

\begin{proposition}
\label{prop:bq-tf}
For a group $G$, the following hold.
\begin{enumerate}[itemsep=2pt]
\item \label{bbq1}
The inclusion $G'\inj G'_{\rat}$ induces a functorial $\tilde\nu$-morphism, 
$\kappa\colon B(G)\to B_{\rat}(G)$.
\item \label{bbq2}
Suppose $\Tors(G_{\ab})$ is finite. 
Then the map $\kappa\otimes \Q\colon 
B(G)\otimes \Q \to B_{\rat}(G)\otimes \Q$ is surjective.
\item \label{bbq3}
Suppose $G_{\ab}$ is torsion-free. Then the map 
$\kappa\otimes \Q\colon 
B(G)\otimes \Q \to B_{\rat}(G)\otimes \Q$ is 
an isomorphism.
\end{enumerate}
\end{proposition}

\begin{proof}
\eqref{bbq1}
The inclusion $G'\inj G'_{\rat}$ restricts to a map $G''\inj G''_{\rat}$, 
and thus induces a  group homomorphism, $G'/G''\to G'_{\rat}/G''_{\rat}$, 
$g G'' \mapsto g G''_{\rat}$, which is functorial in $G$. This homomorphism 
can be viewed as a map $\kappa\coloneqq\kappa^G\colon B(G)\to B_{\rat}(G)$
which covers the ring map $\tilde\nu$ 
and satisfies $\kappa^H\circ B(\alpha)=B_{\rat}(\alpha)\circ \kappa^G$, 
for all homomorphisms $\alpha\colon G\to H$. We may also 
view $\kappa$ as the composite $B(G)\to  B_{\rat}(G)_{\tilde\nu} \to B_{\rat}(G)$,  
where the first arrow is a $\Z [G_{\ab}]$-linear map and the second 
arrow is the identity map of $B_{\rat}(G)$, viewed as covering the  
map $\tilde\nu$. 

Alternatively, recall from \eqref{eq:ab-abf-covers} that 
the cover $q\colon X^{\ab}\to X$ factors 
through a cover $s\colon X^{\ab} \to  X^{\abf}$, so that 
$q_0\circ s=q$. The composite  
\begin{equation}
\label{eq:ab-abf}
\begin{tikzcd}[column sep=26pt]
H_1(X^{\ab};\Z) \ar["s_*"]{r} & H_1(X^{\abf};\Z)
\ar[twoheadrightarrow]{r} & H_1(X^{\abf};\Z)/\Z\dash\Tors 
\end{tikzcd} 
\end{equation}
then coincides with the morphism $\kappa$ defined above.  

\eqref{bbq2}
From the last description, it follows that the homomorphism
$s_*\otimes \Q\colon H_1(X^{\ab};\Q) \to H_1(X^{\abf};\Q)$ coincides with 
$\kappa\otimes \Q$. Now, if $\Tors(G_{\ab})$ is finite, then 
$s\colon X^{\ab}\to X^{\abf}$ is a finite cover, and so the transfer 
map, $\tau\colon H_1(X^{\abf};\Q)\to H_1(X^{\ab};\Q)$, 
provides a splitting for $s_*$. Therefore, $\kappa\otimes \Q$
is surjective.

\eqref{bbq3}
If $G_{\ab}$ is torsion-free, then $\ker s_*=0$, and so 
$\kappa\otimes \Q=s_*\otimes \Q$ is an isomorphism.
\end{proof}

If $G$ is a finitely generated free group, or a knot group, then, by the discussion 
in \S\ref{subsec:derived-q}, the map $\kappa\colon B(G)\to B_{\rat}(G)$ is a 
$\tilde\nu$-isomorphism. On the other hand, as Example 
\ref{ex:dihedral-alex} below shows, the condition that $G_{\ab}$ be 
torsion-free is necessary for part \eqref{bbq3} of Proposition \ref{prop:bq-tf} 
to hold. Moreover, as Example \ref{ex:alex-torsion} shows, the map  
$\kappa\colon B(G)\to B_{\rat}(G)$ itself need not be an 
isomorphism, even when $G_{\ab}$ is torsion-free.

\begin{example}
\label{ex:dihedral-alex}
Let $G=\Z_2*\Z_2=\langle x_1, x_2\mid x_1^2, x_2^2\rangle $. 
Then $G'=\Z=\langle [x_1,x_2]\rangle$ and $G''=\{1\}$, whence 
$B(G)=\Z$, whereas 
$G'_{\rat}=G''_{\rat}=G$, whence $B_{\rat}(G)=0$.
\end{example}

\begin{example}
\label{ex:alex-torsion}
Let $G=\langle x_1, x_2\mid [x_1,x_2]^n\rangle$.  Then 
$B(G)=\Z[x_1^{\pm 1},x_2^{\pm 1}]/(n)\cong \Z_n[x_1^{\pm 1},x_2^{\pm 1}]$ 
is non-zero if $n\ge 2$, whereas $B_{\rat}(G)=0$.
\end{example}

\section{The derived $p$-series and the mod-$p$ Alexander invariant}
\label{sect:derived-alex-p}

We now review the mod-$p$ version of the derived series and introduce 
the corresponding mod-$p$ Alexander invariant.

\subsection{The derived $p$-series }
\label{subsec:derived-p}
Fix a prime $p$. 
Following Stallings \cite{St83}, Cochran and  Harvey \cite{CH08,CH10}, 
and Lackenby \cite{Lk10}, we define the {\em derived $p$-series}\/ of 
$G$, denoted $\big\{G^{(r)}_p\big\}_{r\ge 0}$, by
\begin{equation}
\label{eq:derived-p}
G^{(0)}_p=G, \quad G^{(r)}_p=
\left\langle \big(G^{(r-1)}_p\big)^p ,\big[G^{(r-1)}_p,G^{(r-1)}_p\big]\right\rangle .
\end{equation} 
Using induction on $r$, it is readily seen that the terms of this series 
are fully invariant subgroups. 
Moreover, each subgroup $G^{(r)}_p$ is a normal subgroup of $G$ 
of index a power of $p$, see \cite{St83}, and 
$G^{(r-1)}_p /G^{(r)}_p \cong H_1(G^{(r-1)}_p ;\Z_p)$, 
see \cite{Lk10}.  In particular, $G/G'_p=H_1(G;\Z_p)$ is the 
maximal elementary $p$-abelian quotient of $G$.

\begin{example}
\label{ex:p-derived-abel}
Suppose $G$ is abelian, Then clearly $G^{(r)}_p=G^{p^{r}}$. 
In particular, if $G$ is elementary $p$-abelian, then $G^{(r)}_p=\{1\}$ 
for all $r\ge 1$.
\end{example}

\begin{example}
\label{ex:agemo}
Suppose $G$ is a $p$-group. Then $G^p =\mho(G)$ is 
the agemo subgroup of $G$, while $G_p'=\Phi(G)$ is the Frattini 
subgroup (see \cite{HHR, Leedham-McKay} for a more general context).
\end{example}

The derived $p$-series can be characterized as the 
fastest descending normal (and even subnormal) 
series for which the successive quotients 
are $\Z_p$-vector spaces, cf.~\cite{St83}. 
The next result captures some of the salient 
features of this series.

\begin{lemma}[\cite{CH08}]
\label{lem:ch-p}
For a group $G$, a prime $p$, and an integer $r\ge 1$, the following hold.
\begin{enumerate}
\item \label{ch1}
$G^{(r)}_p = \ker \big(G^{(r-1)}_p \surj
\big(G^{(r-1)}_p\big)_{\ab} \otimes \Z_p\big)$.\\[-4pt]
\item \label{ch2}
$G^{(r-1)}_p /G^{(r)}_p \cong H_1\big(G; \Z_p[G/G^{(r-1)}_p]\big)$, as right 
$\Z_p[G/G^{(r-1)}_p]$-modules.
\item \label{ch3}
If $G$ is finitely generated, then $G /G^{(r)}_p$ is a finite $p$-group, 
with all elements having order dividing $p^r$.
\end{enumerate}
\end{lemma}

\subsection{The mod-$p$ Alexander invariant}
\label{subsec:alexinv-p}
The first two $p$-derived subgroups of $G$ are 
given by $G'_p=\langle G^p, G'\rangle$ 
and $G''_p=\langle G^{p^2}, (G')^p, [G^p,G^p] , [G^p,G'], G''\rangle$; 
moreover, $G''_{p}\triangleleft G'_{p}$.
We define the {\em  mod-$p$ Alexander invariant}\/ of 
$G$ to be the quotient of these two subgroups, 
\begin{equation}
\label{eq:alex-p}
B_{p}(G)\coloneqq G'_{p}/G''_{p}\, ,
\end{equation}
and view it as a module over the group-ring 
$\Lambda_p\coloneqq \Z_p[H_1(G;\Z_p)]$. 
The module structure is induced by conjugation in the maximal 
metabelian $p$-quotient, $G/G''_{p}$, via the exact sequence 
$1\to G'_{p}/G''_{p} \to G/G''_{p}\to G/G'_{p}\to 1$. 
By Lemma \ref{lem:ch-p}, we have that 
\begin{equation}
\label{eq:alex-p-bis}
B_{p}(G)=(G'_p)_{\ab}\otimes \Z_p= H_1(G'_{p};\Z_p)\cong H_1(G;\Lambda_p)
\end{equation}
as a $\Lambda_p$-module. 
Let $b^p_i(G)\coloneqq\dim_{\Z_p} H_i(G;\Z_p)$
be the $i$-th mod-$p$ Betti number of $G$. 
If $G$ is finitely generated, then the $\Z_p$-vector 
space $B_{p}(G)$ is finite-dimensional, of dimension  
equal to $b^p_1(G'_p)$.

We also define the {\em mod-$p$ Alexander module}\/ 
of $G$ as $A_p(G)=\Lambda_p \otimes_{\Z[G]} I(G)$, 
with $\Lambda_p$-module structure given by multiplication 
on the left factor. Given a group homomorphism $\alpha\colon G\to H$, 
we let $B_{p}(\alpha)\colon B_{p}(G)\to B_{p}(H)$ be the morphism 
between mod-$p$ Alexander invariants induced by the restriction 
$\alpha'\colon G'_p\to H'_p$. It is readily seen that the assignments 
$G\leadsto H_1(G;\Z_p)$ and $G\leadsto B_{p}(G)$ are functorial 
and compatible with one another, in a manner similar to the one 
described in \S\ref{subsec:alexinv}. 

\subsection{Topological interpretation}
\label{subsec:alexp-top}
The mod-$p$ Alexander invariant admits the following 
interpretation in terms of covering spaces. Let $X$ be 
a connected CW-complex with $\pi_1(X)=G$, and let 
$q_p\colon X^{(p)}\to X$ be the $p$-congruence cover of $X$;  
that is, the regular $H_1(X;\Z_p)$-cover classified 
by the composite
\begin{equation}
\label{eq:modp-cong}
\begin{tikzcd}[column sep=26pt]
\pi_1(X) \ar[twoheadrightarrow, "\ab"]{r} & H_1(X;\Z) 
\ar[twoheadrightarrow, "\nu_p"]{r} & H_1(X;\Z_p) ,
\end{tikzcd} 
\end{equation}
where $\nu_p$ is the coefficient homomorphism 
induced by the projection $\Z\surj \Z_p$. We then have 
a commuting diagram of regular covers,
\begin{equation}
\label{eq:p-covers}
\begin{tikzcd}[column sep=10pt, row sep=10pt, ]
X^{\ab}\phantom{ \, .}  \ar[dr, "s_p"] \ar[dd, "q"]\\
& X^{(p)} \ar[dl, "q_p"] \\
X \, .
\end{tikzcd} 
\end{equation}

By construction, $\pi_1(X^{(p)})=G'_p$; therefore, 
$B_p(G)\cong H_1(X^{(p)};\Z_p)$, with $\Lambda_p$-module 
structure induced by the action of $H_1(X;\Z_p)$ by deck 
transformations.  Proceeding as in \S\ref{subsec:alexinv-top}, 
we may identify the mod-$p$ Alexander invariant of $G$ with  
the cokernel of the boundary map 
$\partial_2^{p}\colon C_2(X^{(p)};\Z_p)\to C_1(X^{(p)};\Z_p)$.
Moreover, $A_{p}(G)=H_1(X^{(p)},F_p; \Z_p)$, where 
$F_p=q^{-1}_p(x_0)$. The $\Z_p$-homology exact sequence 
of the pair $(X^{(p)},F_p)$ now yields a natural short 
exact sequence of $\Lambda_p$-modules,
\begin{equation}
\label{eq:crowell-p}
\begin{tikzcd}[column sep=18pt]
0\ar[r]& B_{p}(G) \ar[r]&A_{p}(G)
\ar[r]&I_{p}(G) \ar[r]& 0 \, ,
\end{tikzcd}
\end{equation}
where $I_p(G)=\ker (\varepsilon\colon \Lambda_p\to \Z_p)$ 
is the augmentation ideal. 
This sequence is natural, and may be viewed as the 
mod-$p$ analogue of Crowell's exact sequence \eqref{eq:crowell}.

\subsection{A comparison map}
\label{subsec:nup}
The next lemma provides a functorial comparison map between the reduction 
mod-$p$ of the usual Alexander invariant and its mod-$p$ version. 
Let $\tilde\nu_p\colon \Z_p[H_1(G;\Z)]\surj \Z_p[H_1(G;\Z_p)]$ be 
the linear extension of the coefficient homomorphism $\nu_p\colon 
H_1(G;\Z)\surj H_1(G;\Z_p)$ to group rings.  

\begin{lemma}
\label{lem:bp-tf}
The inclusion $G'\inj G'_{p}$ induces a functorial $\tilde\nu_p$-morphism, 
$\kappa_p\colon B(G) \otimes \Z_p \to  B_{p}(G)$. 
\end{lemma}

\begin{proof}
The inclusion $G'\inj G'_{p}$ restricts to a map $G''\inj G''_{p}$, 
and thus induces a group homomorphism, $G'/G''\to G'_{p}/G''_{p}$. This map 
factors through a homomorphism, $G'/G'' \otimes \Z_p \to G'_{p}/G''_{p}$, 
which can can be viewed as a $\tilde\nu_p$-morphism, 
$\kappa_p\coloneqq\kappa^G_p\colon B(G)\otimes \Z_p\to B_{p}(G)$. 
Clearly, $\kappa^H_p\circ B(\alpha)=B_{p}(\alpha)\circ \kappa^G_p$,  
for all homomorphisms $\alpha\colon G\to H$.
\end{proof}

We may interpret the map $\kappa_p$ as the homomorphism 
$(s_p)_*\colon H_1(X^{\ab}; \Z_p) \to H_1(X^{(p)}; \Z_p)$
induced in first homology by the cover $s_p$ from diagram 
\eqref{eq:p-covers}. As illustrated in the two examples below, the 
map $\kappa_p$ is neither injective nor surjective, in general.

\begin{example}
\label{ex:zp-alex-free}
Let $X=\bigvee^n S^1$ be a wedge of $n\ge 2$ circles. Then 
$X^{(p)}\simeq \bigvee^m S^1$, where $m=p^n(n-1)+1$. Identifying 
$F_n=\pi_1(X)$, we find that $B_p(F_n)=H_1(X^{(p)},\Z_p)=\Z_p^m$.
On the other hand, we infer from Example \ref{ex:z-alex}  that  
$B(F_2)\otimes \Z_p=\Z_p[\Z^2]$; hence, the map 
$\kappa_p\colon B(F_2)\otimes \Z_p\to B_p(F_2)$ is 
not injective. 
\end{example}

\begin{example}
\label{ex:z-alex}
Suppose $G$ is abelian. Then $B(G)=B_{\rat}(G)=0$, yet 
$B_p(G)=G^p/G^{p^2}$, which is non-trivial 
in general; for instance, $B_p(\Z^n)= \Z^n_p$. In fact,  
the mod-$p$ Alexander invariant may be able to distinguish 
groups for which the other two kinds of Alexander invariants 
coincide; for example, $B_p(\Z_p\oplus \Z_p)=0$, yet 
$B_p(\Z_{p^2})=\Z_p$.
\end{example}

\section{Associated graded Lie algebras}
\label{sect:grg} 

In this section we consider three types of graded Lie 
algebras associated to a group $G$---%
the usual one, its  rational version, and its mod-$p$ version---%
and review some of their salient features.

\subsection{Lower central series and associated graded Lie algebra}
\label{subsec:lcs}

The {\em lower central series}\/ (LCS) of a group $G$, 
denoted $\gamma(G)=\{\gamma_n(G)\}_{n\ge 1}$, was 
introduced by P.~Hall in \cite{Hall}. Defined inductively 
by formula \eqref{lcs1}, 
this an N-series, in the sense of Lazard \cite{Lazard}; 
that is to say, $[\gamma_m(G),\gamma_{n}(G)]\subseteq 
\gamma_{m+n}(G)$, for all $m, n \ge 1$, see for instance 
\cite{Leedham-McKay,MKS, Passman}. 
In particular, the LCS is a central series, i.e., 
$[G,\gamma_n (G)]\subseteq \gamma_{n+1} (G)$ for all $n$.  
Moreover, its terms are fully invariant subgroups of $G$. 
By definition, the LCS terminates in finitely many 
steps if the group is nilpotent, and the intersection of the 
terms of the series is trivial if and only if $G$ is residually nilpotent. 

Note that $G^{(n-1)}\subseteq \gamma_{n}(G)$, with 
equality for $n=1$ and $2$. Consequently, every nilpotent 
group is solvable. As another consequence of this observation, 
the Alexander invariant $B(G)=G'/G''$ surjects onto 
the quotient group $\gamma_2(G)/\gamma_3(G)$. 

The {\em associated graded Lie algebra}\/ of $G$ is the direct 
sum of the successive quotients of the lower central series. 
The addition in $\gr(G)$ is induced from the group multiplication, 
while the Lie bracket is induced from the group commutator; 
furthermore, this bracket is compatible with the grading. 
By construction, the Lie algebra $\gr(G)$ is generated by 
its degree $1$ piece, $\gr_1(G)=G_{\ab}$; 
thus, if $G_{\ab}$ is finitely generated, then so are the 
LCS quotients of $G$. Likewise, the $\Q$-Lie algebra $\gr(G)\otimes \Q$ 
is generated in degree $1$ by the $\Q$-vector space 
$G_{\ab}\otimes \Q=H_1(G;\Q)$. Assume now that the 
first Betti number, $b_1(G)=\dim_{\Q} H_1(G;\Q)$, is finite; 
we may then define the {\em LCS ranks}\/ of $G$ as 
\begin{equation}
\label{eq:lcs-ranks}
\phi_n(G)\coloneqq \dim_{\Q} \gr_n(G) \otimes \Q \, .
\end{equation}

If $\alpha\colon G\to H$ is a group homomorphism,
then $\alpha(\gamma_n(G))\subseteq \gamma_n(H)$, 
and thus $\alpha$ induces a map 
$\gr(\alpha)\colon \gr(G)\to \gr(H)$.  It is readily 
seen that this map preserves Lie brackets and that 
the assignment $\alpha \leadsto \gr(\alpha)$ is functorial. 

The associated graded Lie algebra of a group $G$ may be approximated 
by the associated graded Lie algebras of its solvable quotients. 
For each $r\ge 2$, the quotient map, $G\surj G/G^{(r)}$, induces a surjective 
morphism, $\gr_n(G) \surj \gr_n(G/G^{(r)})$, which is an isomorphism for 
$n\le 2^{r}-1$, see \cite{SW-jpaa}. In the case when $r=2$, originally studied by 
K.-T. Chen in \cite{Chen51}, the corresponding Lie algebra, $\gr(G/G'')$, 
is called the {\em Chen Lie algebra}\/ of $G$. Assuming $b_1(G)<\infty$, we   
may define the {\em Chen ranks}\/ of $G$ as 
\begin{equation}
\label{eq:chen-ranks}
\theta_n(G)\coloneqq \phi_n(G/G'') =\dim_{\Q} \gr_n(G/G'') \otimes \Q \, .
\end{equation}
In view of the above discussion, 
we have that $\theta_n(G)\le \phi_n(G)$, with equality for $n\le 3$. We refer 
to \cite{PS-imrn04, SW-jpaa, SW-forum} for detailed treatments 
of this subject.

\subsection{The rational lower central series}
\label{subsec:lcs-stallings}

The rational version of the lower central series of a group $G$, 
denoted $\gamma^{\rat}(G)=\big\{\gamma^{\rat}_n (G)\big\}_{n\ge 1}$, 
was introduced  by Stallings in \cite{St}, and further studied in 
\cite{BL,CH05,Mass,Passman,Su-lcs}. The series in defined inductively by  
\begin{equation}
\label{eq:gamma-q-filtration}
\gamma^{\rat}_1 (G)=G\: \text{ and }\:
\gamma^{\rat}_{n+1} (G)=\ssqrt{ \left[G,\gamma^{\rat}_{n}(G)\right]}\, .
\end{equation}

The terms of this series are fully invariant subgroups of $G$. Moreover, 
as shown in \cite{Su-lcs}, we have that $\gamma^{\rat}_{n} (G)=\ssqrt{\gamma_{n}(G)}$, 
for all $n\ge 1$. This implies that $\gamma^{\rat}(G)$ is an N-series  
(see \cite{Mass,Passman}), and thus, a central series. In fact, the rational  
LCS is the most rapidly descending central series whose 
successive quotients are torsion-free abelian groups.
This series terminates in finitely many steps if the group is 
a torsion-free nilpotent. Moreover, the intersection of the 
terms of the rational LCS is trivial if and only $G$ is residually 
torsion-free nilpotent (see \cite{BL}). 
Also note that we have inclusions $G^{(n-1)}_{\rat}\subseteq \gamma^{\rat}_n(G)$, 
with equality for $n=1$ and $2$.  Consequently, the 
rational Alexander invariant, $B_{\rat}(G)=G'_{\rat}/G''_{\rat}$, 
maps surjectively onto the group $\gamma_2^{\rat}(G)/\gamma^{\rat}_3(G)$. 

The direct sum of the successive quotients of $\gamma^{\rat}(G)$,  
\begin{equation}
\label{eq:grg-q}
\gr^{\rat}(G)= \bigoplus\nolimits_{n\ge 1} 
\gamma^{\rat}_n (G)/ \gamma^{\rat}_{n+1} (G),
\end{equation}
with Lie bracket induced from the group commutator, 
constitutes the {\em rational associated graded Lie algebra}\/ of $G$. 
If $\alpha\colon G\to H$ is a group homomorphism,
then $\alpha$ induces a functorial morphism of graded Lie algebras, 
$\gr^{\rat}(\alpha)\colon \gr^{\rat}(G)\to \gr^{\rat}(H)$. 

Clearly $\gamma_n(G)\le \gamma^{\rat}_n(G)$ for all $n$. 
We thus we have an induced map between associated graded Lie 
algebras, $\Phi=\Phi^G\colon \gr(G)\to \gr^{\rat}(G)$, which is functorial 
with respect to group homomorphisms. In degree $1$, this map 
is simply the projection $G_{\ab}\surj G_{\abf}$, with kernel $\Tors(G_{\ab})$.
In \cite[Proposition 7.2]{BL}, Bass and Lubotzky prove the following result.

\begin{proposition}[\cite{BL}]
\label{prop:bl-grlie}
For a group $G$, the following hold.
\begin{enumerate}[itemsep=2pt]
\item \label{bl-gr1}
The map $\Phi\colon \gr(G)\to \gr^{\rat}(G)$ has torsion kernel 
and cokernel in each degree.

\item \label{bl-gr2}
The map $\Phi\otimes \Q\colon  
\gr(G)\otimes \Q\to \gr^{\rat}(G)\otimes \Q$ 
is an isomorphism of graded Lie algebras.
\end{enumerate}
\end{proposition}

Recall from \S\ref{subsec:lcs} that the Lie algebra $\gr(G)$ 
is generated by its degree $1$ piece, $\gr_1(G)=G_{\ab}$. It  
follows from Proposition \ref{prop:bl-grlie}, part \eqref{bl-gr2}  
that the $\Q$-Lie algebra $\gr^{\rat}(G)\otimes \Q$ 
is also generated in degree $1$, this time by the $\Q$-vector 
space $G_{\ab}\otimes \Q=H_1(G;\Q)$.  

Assume now that $b_1(G)<\infty$. In this case, we may define the 
rational LCS ranks by 
$\phi^{\rat}_n(G)\coloneqq \dim_{\Q} \gr^{\rat}_n(G) \otimes \Q$ 
and the rational Chen ranks by 
$\theta^{\,\rat}_n(G)\coloneqq \phi^{\rat}_n(G/G''_{\rat})$. 
With this setup, Proposition \ref{prop:bl-grlie}, part \eqref{bl-gr2} 
has the following immediate corollary. 

\begin{corollary}
\label{cor:lcs-chen-q}
Suppose $b_1(G)<\infty$. Then 
$\phi^{\rat}_n(G)=\phi_n(G)$ and $\theta^{\,\rat}_n(G)=\theta_n(G)$ 
for all $n\ge 1$.
\end{corollary}

\subsection{The mod $p$ lower central series}
\label{subsec:lcs-p}

Now fix a prime $p$.  The {\em mod-$p$ lower central series}, 
denoted $\{\gamma^{p}_n(G)\}_{n\ge 1}$, was introduced   
by Stallings in \cite{St}, and further studied in many works, including 
\cite{BG,CH05,Lk10,Paris,Su-lcs}. The series is defined inductively, as follows:  
\begin{equation}
\label{eq:gamma-p-filtration}
\gamma^{p}_1 (G)=G \:\text{ and }\: 
\gamma^{p}_{n+1} (G)=\left\langle(\gamma^{p}_{n}(G))^p, 
[G,\gamma^{p}_{n} (G)] \right\rangle .
\end{equation} 

This is a $p$-torsion series, meaning that $(\gamma^{p}_{n}(G))^p \subseteq 
\gamma^{p}_{n+1} (G)$ for all $n\ge 1$; it is also an N-series, cf.~\cite{Paris}.  
Therefore, the successive quotients, 
$\gr_n^{p}(G)\coloneqq \gamma^{p}_{n}(G)/\gamma^{p}_{n+1}(G)$,  
are $p$-torsion abelian groups, and thus can be viewed as $\Z_p$-vector spaces;  
in particular, $G/\gamma^{p}_2(G)=H_1(G;\Z_p)$. 
The mod-$p$ LCS is the fastest-descending central series 
whose successive quotients are $\Z_p$-vector spaces.  
Clearly, the terms of the series are fully invariant subgroups. 

\begin{remark}
\label{rem:Zassenhaus}
A sequence $K=\{K_n\}_{n\ge 1}$ of subgroups of $G$ is said to be an 
{\em $N_p$-series}  if $K$ is an $N$-series and 
$K_n^p\subseteq K_{pn}$ for all $n\ge 1$. An $N_p$-series is 
both an $N$-series and a $p$-torsion series, but the converse does 
not hold.  The Stallings series $\gamma^{p}(G)$ is not in general an 
$N_p$-series. Rather, the canonical $N_p$-series associated to $G$, 
denoted $\gamma^{[p]}(G)$, is the one introduced by H.~Zassenhaus 
in 1939, whose terms are given inductively by $\gamma^{[p]}_1(G)=G$ and 
\begin{equation}
\label{eq:np-series}
\gamma^{[p]}_{n}(G)=\prod_{\substack{mp^{j}\ge n\\m\ge 1, j\ge 0}} 
\big(\gamma^{[p]}_{m}(G) \big)^{p^j} ,
\end{equation}
see for instance \cite[p.~17]{Leedham-McKay}. 
As noted in \cite{Cooper}, the two filtrations, 
$\gamma^{p}(G)$ and $\gamma^{[p]}(G)$, are cofinal, though they may 
differ significantly at a granular level. For instance, if $G$ is abelian, 
then $\gamma^{p}_n(G)=G^{p^{n-1}}$, whereas 
$\gamma^{[p]}_n(G)=G^{p^{j}}$, where $j=\lceil \log_p(n) \rceil$. 
\end{remark}

As observed in \cite{Lk10}, each term $G^{(r)}_p$ of the derived $p$-series 
contains $\gamma^p_n(G)$ as a normal subgroup, for all $n$ sufficiently large.   
The reason for this is that the $p$-lower central series of any finite $p$-group (in 
particular, the quotient $G/G^{(r)}_p$), terminates at $1$.  
Moreover, $G'_{p}=\gamma^{p}_2(G)$ and 
$G''_{p}\subseteq \gamma^{p}_3(G)$. Consequently, we have a 
surjective homomorphism from the mod-$p$ Alexander invariant, 
$B_p(G)=G'_p/G''_p$, to $\gr_2^p(G)=\gamma_2^{p}(G)/\gamma^{p}_3(G)$. 

\begin{example}
\label{ex:p-abel}
If $G$ is abelian, then $\gamma^{p}_n(G)=G^{(n-1)}_p=G^{p^{n-1}}$. 
For instance, if $G=\Z=\langle a\rangle$, then 
$\gamma^p_n(G)=\langle a^{p^{n-1}}\rangle \cong \Z$ for all $n\ge 1$.
On the other hand, if $G=\Z_p^s$ is an elementary $p$-abelian group, then 
$\gamma^p_n(G)=\{1\}$ for all $n\ge 2$. 
\end{example}

The associated graded Lie algebra, 
$\gr^{p}(G)\coloneqq \bigoplus_{n\ge 1} \gr_n^{p}(G)$, 
with addition and Lie bracket defined as before, 
can be viewed as a Lie algebra over the field $\Z_p$. 
The assignment $G\leadsto \gr^{p}(G)$ is functorial. 
The Lie algebra $\gr^{p}(G)$ 
supports additional operations, $\gr^{p}_n(G) \to \gr^{p}_{n+1}(G)$, 
which are induced by the $p$-th power map, $g\mapsto g^p$. 
As observed in \cite[\S{12}]{BL}, this Lie algebra is 
generated through Lie brackets and 
power operations by its degree $1$ piece, 
$\gr^p_1(G)=H_1(G;\Z_p)$.

Assume now that $b^p_1(G)=\dim_{\Z_p} H_1(G;\Z_p)$ is finite. 
By the observation we just made, 
all homogeneous pieces of $\gr^{p}(G)$ are also finite-dimensional. 
Hence, we may define the mod-$p$ LCS ranks of $G$ by setting 
$\phi^{p}_n(G)\coloneqq \dim_{\Z_p} \gr^{p}_n(G)$ 
and the mod-$p$ Chen ranks as 
$\theta^{p}_n(G)\coloneqq \dim_{\Z_p} \gr^{p}_n(G/G''_{p})$. 
Clearly, $\phi^{p}_n(G)\ge \theta^{p}_n(G)$ for all $n\ge 1$. 

Finally, suppose $G$ is finitely generated; a previous remark   
then gives $\dim_{\Z_p} B_p(G)\ge \phi_2^p(G)$. 
If $b^p_2(G)$ is also finite, work of 
Shalen and Wagreich \cite{SW92} (see also Lackenby \cite{Lk09})
shows that 
$\phi_2^p(G)\ge \binom{b_1^{p}(G)}{2} +b_1^p(G)-b_2^p(G)$. 
Putting these inequalities together yields a lower bound 
on the dimension of the $\Z_p$-vector space $B_p(G)$,  
solely in terms of the first two mod-$p$ Betti numbers of $G$. 

\begin{example}
\label{ex:heis-alex}
If $G$ is the fundamental group of a closed $3$-manifold 
(assumed to be orientable if $p$ is odd), then Poincar\'e duality 
gives $\dim_{\Z_p} B_p(G)\ge \binom{b_1^{p}(G)}{2}$. For 
$G=\Z^3$ this is an equality, but in general 
the inequality is strict. For instance, if $G$ is the 
Heisenberg nilmanifold group of $3\times 3$ upper diagonal 
integral matrices with $1$'s down the diagonal, then  $b_1^p(G)=2$, 
yet $B_p(G)=\Z_p\oplus \Z_p$.
\end{example}

\part{Group extensions}
\label{part:ext}
\section{Massey's correspondence}
\label{sect:ext}

We now switch our focus to group extensions. After a brief discussion of the 
monodromy action of an extension $1\to K\to G\to Q\to 1$, we present a detailed 
overview of Massey's correspondence between the filtration of the Alexander 
invariant of $K$ by powers of the augmentation ideal of $Q$ 
and the lower central series of the maximal metabelian quotient of $G$.  
Finally, we extend this correspondence to the rational and mod-$p$ settings.

\subsection{Monodromy action}
\label{subsec:mono}
We start with a quick review of group extensions; for more on this topic, we refer 
to \cite[Ch.~IV]{Brown} and \cite[Ch.~VI]{HS}. Consider a short exact sequence 
of groups, 
\begin{equation}
\label{eq:abc-exact}
\begin{tikzcd}[column sep=20pt]
1\ar[r] &  K \ar[r, "\iota"]
& G \ar[r, "\pi"] & Q \ar[r]  & 1.
\end{tikzcd}
\end{equation}
Let $\Aut(K)$ be the group of {\em right}\/ automorphisms of $K$, 
with group operation $\alpha\cdot \beta = \beta\circ \alpha$. 
Choosing a set-theoretic section of the projection map $\pi$, i.e., a function 
$\sigma\colon Q\to G$ such that $\pi\circ \sigma=\id_Q$, defines a function 
$\varphi \colon Q\to \Aut(K)$ by setting $\varphi(x)(a)=\sigma(x) a \sigma(x)^{-1}$ 
for $x\in Q$ and $a\in K$.  

In general, this function (which does depend on the choice 
of section) is not a homomorphism.  Nevertheless, if we 
post-compose it with the projection from $\Aut(K)$ to the outer 
automorphism group $\Out(K)\coloneqq \Aut(K)/\Inn(K)$, the resulting 
map, $\tilde\varphi\colon Q\to \Out(K)$, \emph{is}\/ a group homomorphism, 
which does not depend on the choice of section used to define it. The map 
sending an automorphism of $K$ to the induced automorphism of $K_{\ab}$ 
factors through a homomorphism, $\upsilon\colon \Out(K)\to\Aut(K_{\ab})$. 
We call the composite $\upsilon\circ \tilde\varphi\colon Q\to \Aut(K_{\ab})$ 
the monodromy representation of the extension \eqref{eq:abc-exact}.

Assume now that the exact sequence \eqref{eq:abc-exact} splits, i.e., 
there is a homomorphism 
$\sigma\colon Q\to G$ such that $\pi\circ \sigma=\id_Q$.  In this case, the 
corresponding function, $\varphi \colon Q\to \Aut(K)$, is a well-defined 
homomorphism.  This approach realizes the group $G$ as a split 
extension, $G=K\rtimes_{\varphi} Q$. That is, $G$ equals $K\times Q$ 
as a set, and its group operation can be expressed as 
$(x_1,y_1)\cdot (x_2,y_2)=(x_1\varphi(y_1)(x_2) , y_1y_2)$. 
In what follows, we will identify the group $Q$ with its image under the 
splitting, $\sigma(Q)\le G$, and thus view $Q$ as a subgroup of $G$; 
likewise, we will identify $K$ with $\iota(K)$, and thus view it as a normal 
subgroup of $G$.  With these identifications, 
the action of $Q$ on $K$ is simply the restriction of the conjugation 
action in $G$; that is, $\varphi(y)(x)=yxy^{-1}$. 

\subsection{Massey's correspondence}
\label{subsec:massey}
Suppose now that the normal subgroup $K\triangleleft G$ from \eqref{eq:abc-exact} 
is abelian. In that case, Massey established in \cite{Ms-80} a simple, 
yet very fruitful connection between $B(K)$, the Alexander invariant of $K$,  
and the lower central series of $G/G''$, the maximal metabelian quotient of $G$. 
Since the original reference contains only sketches of proofs, we provide complete 
details, which will also serve as a blueprint for the rational and modular 
extensions of this correspondence that will be given below.

Since the group $K$ is assumed to be abelian, the monodromy of the extension, 
$\varphi=\upsilon\circ\tilde\varphi \colon Q\to \Aut(K)$, is a well-defined 
homomorphism, which puts the structure of a $\Z[Q]$-module on $K$. 
For a group $G$ and a coefficient ring $\k$, we denote by $I_{\k}(G)$ 
the augmentation ideal of $\k[G]$, that is, the kernel of the ring map 
$\varepsilon\colon \k[G]\to \k$ defined by 
$\varepsilon \big(\sum n_g g\big)=\sum n_g$.

\begin{lemma}[\cite{Ms-80}]
\label{lem:massey}
Let $1\to K\to G\to Q\to 1$ be an extension of groups, and assume 
$K$ is abelian. Define a filtration $\{K_n\}_{n\ge 0}$ on $K$ inductively, 
by setting $K_0=K$ and $K_{n+1}=[G,K_n]$, and let $I=I_{\Z}(Q)$. 
Then $K_n =I^n  K$ for all $n\ge 0$.
\end{lemma}

\begin{proof}
We write the group operation in $K$ multiplicatively when viewing it 
as a subgroup of $G$, and additively when viewing it as a $\Z[Q]$-module. 
Fix a set-section $\sigma\colon Q\to G$ of the projection $\pi\colon G\to Q$, 
and let $\varphi \colon Q\to \Aut(K)$ be the corresponding monodromy. 
An element $h\in Q$ then acts on $K$ by sending an element $k\in K$ to 
$h\cdot k= \varphi(h)(k) = gkg^{-1}\in K$, where $g=\sigma(h)\in G$. 

The claim is now proved by induction on $n$, with the base case $n=0$ 
being obvious. So assume $K_n =I^n K$. Consider a commutator 
$gkg^{-1}k^{-1}\in K_{n+1}=[G,K_n]$ with $g\in G$ and $k\in K_n$. 
In view of the above observations, such a 
commutator corresponds in one-to-one fashion to the element 
$(h -1) k \in I  K_n$, where $h=\pi(g)$. By the induction hypothesis, 
$I  K_n=I^{n+1} K$; thus, $K_{n+1}=I^{n+1}K$, and we are done.
\end{proof}

\begin{theorem}[\cite{Ms-80}]
\label{thm:massey-alexinv}
Let $G$ be a group, and let $I=I_{\Z}(G_{\ab})$. 
Then $I^n B(G) = \gamma_{n+2} (G/G'')$, for all $n\ge 0$. 
\end{theorem}

\begin{proof}
Consider the extension $1\to G'/G'' \to G/G''\to G/G'\to 1$ 
from \eqref{eq:gprimeprime}, and recall that the Alexander invariant 
$B(G)$ is the (abelian) group $G'/G''$, viewed as a 
module over $\Z[G_{\ab}]$. Let $\{(G'/G'')_n\}_{n\ge 0}$ 
be the filtration of the subgroup $K=G'/G''$ defined in Lemma \ref{lem:massey}; 
the lemma then gives $(G'/G'')_n= I^n B(G)$ for all $n\ge 0$. 

It remains to show that $(G'/G'')_n= \gamma_{n+2}(G/G'')$ for all $n\ge 0$. 
(Note that $\gamma_{n+2}(G/G'')$ is a subgroup of $\gamma_{2}(G/G'')=G'/G''$, 
and thus is an abelian group.) We prove this claim by induction, with the 
base case $n=0$ clearly holding. For the induction step, we have that 
$(G'/G'')_{n+1}=[G,(G'/G'')_{n}]=[G, \gamma_{n+2}(G/G'')]=\gamma_{n+3}(G/G'')$, 
and the proof is complete.
\end{proof}

\subsection{Completion and associated graded of $B(G)$}
\label{subsec:gr-bg}
Recall that the ring $R=\Z[G_{\ab}]$ admits a filtration by powers 
of the augmentation ideal $I=I_{\Z}(G_{\ab})$. Let 
$\widehat{R}=\varprojlim R/I^n$ be the completion 
of $R$ with respect to this filtration, and let 
$\gr(R)=\gr(\widehat{R})=\bigoplus_{n\ge 0} I^n/I^{n+1}$ 
be the associated graded object. Both $\widehat{R}$ and $\gr(R)$ 
acquire in a natural way a ring structure, which is compatible 
with the filtration, respectively, the grading. 

Let $\widehat{B}=\varprojlim B/I^nB$ be the $I$-adic completion of the 
Alexander invariant $B=B(G)$, viewed as a module over $\widehat{R}$, 
and  let $\gr(\widehat{B})=\gr(B)$ be the associated graded object, 
viewed as a (graded) module over $\gr(R)$.  As such, $\gr(B)$ is 
generated by its degree $0$ piece, $B/IB$. It follows 
from Theorem \ref{thm:massey-alexinv} that the 
$\gr(R)$-generators of $\gr(B)$ correspond to a generating 
set for $\gr_2(G)$; moreover, the theorem has the 
following immediate corollary.

\begin{corollary}[\cite{Ms-80}]
\label{cor:alex-chen}
$\gr_n(B(G))\cong \gr_{n+2} (G/G'')$, for all $n\ge 0$. 
\end{corollary}

Now suppose that $b_1(G)<\infty$. Then $\gr(B(G)\otimes \Q)$ is 
a finitely generated graded module over the graded ring $\gr(\Q[G_{\ab}])$. 
Let $\theta_n(G)=\dim_{\Q} \gr_n(G/G'')\otimes \Q$ be the Chen ranks 
of $G$, starting with $\theta_1(G)=b_1(G)$.  
As a consequence of Corollary \ref{cor:alex-chen}, 
the Hilbert series of the rationalization of $\gr(B(G))$ determines 
the Chen ranks of $G$, as follows,
\begin{equation}
\label{eq:hilb-b-chen}
\Hilb (\gr(B(G) \otimes \Q),t) = \sum_{n\ge 0} \theta_{n+2} (G) t^n .
\end{equation}

Let $\alpha\colon G\to H$ be a group homomorphism. 
Recall from \S\ref{subsec:alexinv-func} that $\alpha$ 
induces a map of modules, $B(\alpha)\colon B(G)\to B(H)$, 
which covers the ring map $\tilde\alpha_{\ab}\colon R \to S$, 
where $S=\Z[H_{\ab}]$.  Clearly, the map $B(\alpha)$ preserves 
$I$-adic filtrations, and thus induces a morphism 
$\widehat{B(\alpha)}\colon \widehat{B(G)}\to \widehat{B(H)}$ 
which covers the ring map 
$\hat{\tilde\alpha}_{\ab}\colon \widehat{R} \to \widehat{S}$. 
Passing to associated graded objects, we obtain the morphism 
$\gr(B(\alpha))\colon \gr(B(G))\to \gr(B(H))$, which covers the ring map 
$\gr(\tilde\alpha_{\ab})\colon \gr (R)\to \gr(S)$.  For future reference, 
we record a fact regarding this last map.

\begin{lemma}
\label{lem:inj}
Suppose $H_{\ab}$ is finitely generated and the map 
$\alpha_{\ab}\colon G_{\ab} \to H_{\ab}$ is injective.  
Then the map $\gr(\tilde\alpha_{\ab})\colon \gr (R)\to \gr(S)$ 
is also injective.
\end{lemma}

\begin{proof}
Our assumptions imply that $G_{\ab}$ is also finitely generated. 
Letting $r$ and $s$ denote the minimum number of generators 
of $G_{\ab}$ and $H_{\ab}$, respectively, the map 
$\alpha_{\ab}$ lifts to an injective $\Z$-linear map, 
$\Z^r\to \Z^s$, given by multiplication with a matrix $M$. 

The  ring $\gr(R)$ can be described as the quotient of the 
polynomial ring $\Z[x_1,\dots, x_r]$ by a monomial ideal 
determined by $\Tors(G_{\ab})$, and likewise for $\gr(S)$. 
The ring map $\gr(\tilde\alpha_{\ab})\colon \gr (R)\to \gr(S)$ 
lifts to a map between polynomial rings, 
$\mu\colon \Z[x_1,\dots, x_r] \to \Z[y_1,\dots ,y_s]$, 
which may be identified with the linear change of variables defined by $M$. 
Clearly, the map $\mu$ is injective, and thus the map 
$\gr(\tilde\alpha_{\ab})$ is also injective. 
\end{proof}

\subsection{A rational Massey correspondence}
\label{subsec:massey-rat}
The next two results are rational analogues of Massey's correspondence. 
Both the statements and the proofs are similar to the integral case, though 
they do require some modifications, which we record below.

\begin{lemma}
\label{lem:massey-q}
Let $1\to K\to G\to Q\to 1$ be an extension of groups, and assume 
$K$ is abelian. Define a filtration $\{K^{\rat}_n\}_{n\ge 0}$ on $K$ inductively, 
by setting $K^{\rat}_0=K$ and $K^{\rat}_{n+1}=\ssqrt{[G,K^{\rat}_n]}$.  
Letting $I=I_{\Q}(Q)$, we have 
$K^{\rat}_n\otimes \Q =I^n (K \otimes \Q)$, for all $n\ge 0$.
\end{lemma}

\begin{proof}
Recall that the action of $h\in Q$ on $k\in K$ is given by 
$h\cdot k= gkg^{-1}$, where $g=\sigma(h)\in G$. 
By definition, an element $x\in K$ belongs to $\ssqrt{[G,K]}$ 
if there is an integer $m>0$ such that $x^m$ can be 
written as a product of commutators of the form 
$gkg^{-1}k^{-1}\in  [G,K]$. Viewing now the $\Q$-vector space 
$\ssqrt{[G,K]} \otimes \Q$ as a module over $\Q[Q]$, the element 
$x$ corresponds to a sum of elements of the form 
$\frac{1}{m} (h -1)\cdot k \in I (K\otimes \Q)$. 
This shows that 
$\ssqrt{[G,K]} \otimes \Q=I  (K\otimes \Q)$, 
proving the claim for $n=1$. The general case 
follows by induction on $n$, as in the proof of 
Lemma \ref{lem:massey}.
\end{proof}

\begin{theorem}
\label{thm:massey-al-rat}
Let $G$ be a group and let $I=I_{\Q}(G_{\abf})$. 
Then $I^n (B_{\rat}(G) \otimes \Q) = 
\gamma^{\rat}_{n+2} (G/G''_{\rat})\otimes \Q$, 
for all $n\ge 0$. 
\end{theorem}

\begin{proof}
Consider the extension $1\to G'_{\rat}/G''_{\rat} \to G/G''_{\rat}\to G/G'_{\rat}\to 1$ 
from \eqref{eq:gprime-rat}, and recall that the rational Alexander invariant 
$B_{\rat}(G)$ is equal to $G'_{\rat}/G''_{\rat}$, viewed as a module 
over $\Z[G_{\abf}]$.  Let $\{(G'_{\rat}/G''_{\rat})^{\rat}_n\}_{n\ge 0}$ 
be the filtration on $G'_{\rat}/G''_{\rat}$ defined in Lemma \ref{lem:massey-q}; 
the lemma then gives 
$(G'_{\rat}/G''_{\rat})^{\rat}_n \otimes \Q= I^n ( B_{\rat}(G)\otimes \Q)$ 
for all $n\ge 0$. 

To complete the proof, we show by induction that 
$(G'_{\rat}/G''_{\rat})^{\rat}_n= 
\gamma^{\rat}_{n+2}(G/G''_{\rat})$ for all $n\ge 0$. 
(Note that $\gamma^{\rat}_{n+2}(G/G''_{\rat})$ is a subgroup of 
$\gamma^{\rat}_{2}(G/G''_{\rat})=G'_{\rat}/G''_{\rat}$, 
and thus is a torsion-free abelian group.) The base case $n=0$ is   
immediate. For the induction step, we have that 
\[
(G'_{\rat}/G''_{\rat})^{\rat}_{n+1}
=\ssqrt{\left[G,(G'_{\rat}/G''_{\rat})^{\rat}_{n}\right]} 
=\ssqrt{\left[G, \gamma^{\rat}_{n+2}(G/G''_{\rat})\right]}
=\gamma^{\rat}_{n+3}(G/G''_{\rat})\, ,
\] 
and we are done.
\end{proof}

\begin{corollary}
\label{cor:alex-rat-chen}
$\gr_n(B_{\rat}(G) \otimes \Q)\cong \gr_{n+2} (G/G''_{\rat}) \otimes \Q$, 
for all $n\ge 0$. 
\end{corollary}

\begin{proof}
Follows from Theorem \ref{thm:massey-al-rat} by taking associated graded 
groups with respect to the $I$-adic filtration on $B_{\rat}(G) \otimes \Q$ and 
the $\gamma^{\rat} (G/G''_{\rat})$ filtration on $G/G''_{\rat}$, respectively.
\end{proof}

\begin{corollary}
\label{cor:chen-alrat}
If $b_1(G)<\infty$, then, for all $n\ge 2$, 
\[
\theta_n(G)=\theta^{\rat}_n(G)=
\dim_{\Q} \gr_{n-2}(B(G) \otimes \Q)=
\dim_{\Q} \gr_{n-2}(B_{\rat}(G) \otimes \Q)  .
\] 
\end{corollary}

\begin{proof}
The first equality follows from Corollary \ref{cor:lcs-chen-q}. 
By formula \eqref{eq:hilb-b-chen} we have $\theta_n(G)=\dim_{\Q}
\gr_{n-2}(B(G) \otimes \Q)$,  while 
by Corollary \ref{cor:alex-rat-chen} we have $\theta^{\rat}_n(G)=\dim_{\Q}
\gr_{n-2}(B_{\rat}(G) \otimes \Q)$ for all $n\ge 2$, and we are done.
\end{proof}

\subsection{An isomorphism between completions}
\label{subsec:complete-aiq}
Set $R_0=\Z[G_{\abf}]$ and $I_0=I_\Z(G_{\abf})$. 
Let $\widehat{R_0}$ be the completion of $R_0$ 
with respect to the $I_0$-adic filtration, and 
let $\widehat{B_{\rat}}$ be the $I_{0}$-adic completion of the 
rational Alexander invariant $B_{\rat}=B_{\rat}(G)$, 
viewed as a module over $\widehat{R_0}$.

The projection map $\nu\colon G_{\ab}\surj G_{\abf}$ induces a 
surjective ring map, $\tilde\nu\colon R\surj R_0$. 
By Proposition \ref{prop:bq-tf}, the inclusion 
$G'\inj G'_{\rat}$ induces a $\tilde\nu$-morphism,  
$\kappa\colon B\to B_{\rat}$, which becomes 
a surjection upon tensoring with $\Q$, but not necessarily 
an isomorphism. The next proposition shows that, upon completion, 
the map $\kappa\otimes \Q$ does become an isomorphism. 
This result overlaps with \cite[Proposition 2.4]{DHP14}, which 
is both slightly more general (in that it replaces $G'_{\rat}$ with 
an arbitrary subgroup $H\le G$ containing $G'$ as a finite-index 
normal subgroup), and less general (in that it assumes $G$ to be 
finitely generated, which we don't). Our proof, though, is very 
much different.

\begin{theorem}
\label{thm:hat-kappa}
Let $G$ be a group with $b_1(G)<\infty$. Then the map 
$\kappa\colon B(G)\to B_{\rat}(G)$ yields 
\begin{enumerate}[itemsep=2pt]
\item An isomorphism 
$\hat{\kappa} \otimes \Q\colon \widehat{B(G)\otimes \Q} \isom 
\widehat{B_{\rat}(G)\otimes \Q}$ of filtered modules covering the 
filtered ring map $\widehat{\tilde\nu}\otimes \Q\colon \widehat{R}\otimes \Q\surj  
\widehat{R_0}\otimes \Q$.
\item 
An isomorphism 
 $\gr(\kappa)\otimes \Q\colon \gr(B(G) \otimes \Q) 
\isom \gr(B_{\rat}(G)\otimes \Q)$ covering the map of graded rings 
$\gr(\tilde\nu)\colon \gr(R) \otimes \Q\surj \gr(R_0)\otimes \Q$.
\end{enumerate}
\end{theorem}

\begin{proof}
The ring epimorphism $\tilde\nu\colon R\surj  R_0$ preserves augmentation 
ideals; therefore, it induces a filtration-preserving ring map, $\widehat{\tilde\nu} 
\colon \widehat{R}\surj  \widehat{R_0}$. Likewise, the map 
$\kappa\colon B\to B_{\rat}$ is compatible with the 
$I$-adic and $I_0$-adic filtrations in source and target, and thus 
induces a filtration-preserving $\widehat{\tilde\nu}$-morphism 
on completions, $\hat{\kappa}\colon \widehat{B} \to \widehat{B_{\rat}}$.

This morphism induces a morphism between the 
corresponding associated graded modules, 
$\gr(\kappa)\colon \gr(B) \to \gr(B_{\rat})$, covering the ring map 
$\gr(\tilde\nu)\colon \gr(R)\to \gr(R_0)$. 
On the other hand, we have by Proposition \ref{prop:bl-grlie} a morphism 
of graded Lie algebras, $\Phi\colon \gr(G)\to \gr^{\rat}(G)$, which 
induces an isomorphism $\Phi\otimes \Q\colon  
\gr(G)\otimes \Q\to \gr^{\rat}(G)\otimes \Q$.

For each $n\ge 0$ we have a commuting diagram, 
\begin{equation}
\label{eq:map-brb-grq}
\begin{tikzcd}[column sep=46 pt]
\gr_n(B(G)\otimes \Q) 
\ar[r, "\gr_n(\kappa)\otimes \Q"] \ar[d, "\cong"]
&\gr_n(B_{\rat}(G)\otimes \Q)  \ar[d, "\cong"]  
\\
\gr_{n+2}(G)\otimes \Q\ar{r}{\Phi_{n+2}\otimes \Q}[swap]{\cong}
& \gr^{\rat}_{n+2}(G)\otimes \Q\, ,
\end{tikzcd}
\end{equation}
where the vertical arrows are the isomorphisms provided by Corollaries 
\ref{cor:alex-chen} and \ref{cor:alex-rat-chen}.
It follows that the top arrow is an isomorphism, for each $n\ge 0$, 
Hence, the map $\gr(\kappa)\otimes \Q\colon \gr(B \otimes \Q) 
\to \gr(B_{\rat}\otimes \Q)$ is an isomorphism. A standard argument 
(see e.g.~\cite[Lemma 2.4]{SW-ejm}) now implies that 
the map $\hat{\kappa} \otimes \Q\colon \widehat{B\otimes \Q} \to 
\widehat{B_{\rat}\otimes \Q}$ is an isomorphism, too.
\end{proof}

\begin{remark}
\label{rem:grobner}
Setting $r=b_{1}(G)$ and picking a generating set $t_{1},\dots, t_{r}$ 
for $G_{\abf}$ allows us to identify the ring $\Z[G_{\abf}]$ with the ring 
of Laurent polynomials $\Z[t_{1}^{\pm 1}, \dots , t_{r}^{\pm 1}]$ and 
the ring $\gr(\Z[G_{\abf}])$ with the polynomial ring 
$\Z[x_{1},\dots ,x_{r}]$, where $x_i$ corresponds to $t_i-1$. 
This permits direct computation of the Hilbert series 
of $\gr(B (G)\otimes \Q)\cong \gr(B_{\rat} (G)\otimes \Q)$, 
and thus of the Chen ranks $\theta_n(G)=\theta^{\rat}_n(G)$, too, 
via standard methods of commutative algebra, based on the 
use of Gr\"{o}bner bases. We refer to \cite{CSc-adv,CS-tams99,Su-conm, 
SW-aam} for detailed explanations and examples on 
how this algorithm works in various settings.
\end{remark}

\subsection{A modular Massey correspondence}
\label{subsec:massey-p}
For the rest of this section, we fix a prime $p$. The next two results 
are mod-$p$ analogues of Massey's correspondence.

\begin{lemma}
\label{lem:massey-p}
Let $1\to K\to G\to Q\to 1$ be an extension of groups, and assume 
$K$ is an elementary abelian $p$-group. 
Define a filtration $\{K_n\}_{n\ge 0}$ on $K$ inductively, 
by setting $K_0=K$ and $K_{n+1}=[G,K_n]$, and  
let $I=I_{\Z_p}(H_1(Q;\Z_p))$. Then 
$K_n=I^n K$ for all $n\ge 0$.
\end{lemma}

\begin{proof}
By assumption, $K$ is the underlying additive group of a $\Z_p$-vector 
space. Thus, the monodromy action of the extension defines the structure of 
a $\Z_p[H_1(Q;\Z_p)]$-module on $K$. The argument now proceeds as 
in the proof for Lemma \ref{lem:massey}.
\end{proof}

\begin{theorem}
\label{thm:massey-al-p}
Let $G$ be a group, and let $I=I_{\Z_p}(H_1(G;\Z_p))$. 
Then $I^n B_{p}(G) = \gamma^{p}_{n+2} (G/G''_{p})$, 
for all $n\ge 0$. 
\end{theorem}

\begin{proof}
Consider the extension $1\to G'_{p}/G''_{p} \to G/G''_{p}\to G/G'_{p}\to 1$, 
and recall that the mod-$p$ Alexander invariant 
$B_{p}(G)$ is equal to $G'_{p}/G''_{p}$, an elementary abelian $p$-group. 
Using now Lemma \ref{lem:massey-p}, the conclusion 
follows as in the proof of Theorem \ref{thm:massey-alexinv}. 
\end{proof}

Finally, assume that $\dim_{\Z_p} H_1(G;\Z_p)$ is finite, and let 
$\theta^p_n(G)=\dim_{\Z_p} \gr^p_n(G/G''_{p})$ be the mod-$p$ 
Chen ranks of $G$. Also  
let $S=\gr(\Z_p[H_1(G;\Z_p)])$ be the associated graded ring 
with respect to the $I$-adic filtration on $\Z_p[H_1(G;\Z_p)]$, and let 
$\gr(B_{p}(G))$ be the associated graded $S$-module with respect 
to the $I$-adic filtration on $B_{p}(G)$. As a corollary to 
Theorem \ref{thm:massey-al-p}, we obtain the following 
formula relating the Hilbert series of this graded module 
to the generating series for the mod-$p$ Chen ranks of $G$:
\begin{equation}
\label{eq:hilb-b-chen-p}
\Hilb (\gr(B_{p} (G)),t) = \sum_{n\ge 0} \theta^p_{n+2} (G) t^n\, .
\end{equation}

\section{Split extensions and lower central series}
\label{sect:split}
In this section we restrict our attention to split extensions, 
$G=K\rtimes_{\varphi} Q$, and discuss the relationship 
between the lower central series and the associated graded 
Lie algebras of the factors and those of the extension.

\subsection{A generalized Falk--Randell theorem}
\label{subsec:fr}
We start with a recent result of 
Guaschi and Pereira \cite{GP}, which expresses the lower 
central series of a split extension in terms of the lower central series 
of the factors, as well as the extension data.  

\begin{theorem}[\cite{GP}]
\label{thm:gu-pe}
Let $G=K\rtimes_{\varphi} Q$ be a split extension of groups.  
For each $n\ge 1$ there is a split extension, 
$\gamma_n(G)=L_n\rtimes_{\varphi} \gamma_n(Q)$, 
where $L_1=K$ and $L_n\le K$ is the subgroup generated by 
$[K,L_{n-1}]$, $[K, \gamma_{n-1}(Q)]$, and $[L_{n-1},Q]$.
\end{theorem}

In \cite{Su-lcs} we give a more streamlined proof of this theorem, which 
exploits the fact (proved there) that the series $L=\{L_n\}_{n\ge 1}$ 
is an N-series for the group $K$.

Following Falk and Randell \cite{FR}, we say that a split extension 
$G=K\rtimes_{\varphi} Q$ is an {\em almost direct product}\/ if 
$Q$ acts trivially on $K_{\ab}=H_1(K;\Z)$. This condition  is equivalent to 
$\varphi(x)(a) \cdot a^{-1}  \in K'$, for all $x\in Q$ and $a\in K$. 
Viewing $K$ and $Q$ as subgroups of $G$ as explained in \S\ref{subsec:mono}, 
the condition can be written as $[K,Q]\subseteq \gamma_2(K)$.  
As shown in \cite[Proposition 6.3]{BGG11}, the property of being 
an almost direct product does not depend on the choice of a 
splitting for the extension.  

When $Q$ acts trivially on $K_{\ab}$, we prove in \cite{Su-lcs} 
that $L_n=\gamma_n(K)$. In view  of Theorem \ref{thm:gu-pe}, 
this recovers the following well-known result of Falk and Randell \cite{FR}.

\begin{theorem}[\cite{FR}]
\label{thm:falk-ran}
Suppose $G=K\rtimes_{\varphi} Q$ is an almost direct product of 
groups. Then,
\begin{enumerate}
\item \label{fr1}
$\gamma_n(G)=\gamma_n(K)\rtimes_{\varphi} \gamma_n(Q)$, for all $n\ge 1$.
\\[-10pt]
\item \label{fr2}
$\gr(G)=\gr(K) \rtimes_{\bar\varphi} \gr(Q)$.
\end{enumerate}
\end{theorem}

Under additional assumptions, we have the following corollary, 
which will be needed in the proof of Theorem \ref{thm:alex-abex}.

\begin{corollary}
\label{cor:fr-ab}
Let $G=K\rtimes_{\varphi} Q$ be an almost direct product, and 
assume $Q$ is abelian.  Then,
\begin{enumerate}[itemsep=1.5pt]
\item \label{aa1}
$\gamma_n(K)=\gamma_n(G)$ for all $n\ge 2$ and $\gr_{\ge 2}(K)=\gr_{\ge 2}(G)$. 
\item \label{aa2}
If, moreover, $b_1(G)<\infty$, then $\phi_n(K)=\phi_n(G)$ for all $n\ge 2$.
\end{enumerate}
\end{corollary}

\begin{proof}
Since $Q$ is abelian, we have that $\gamma_n(K)=\{1\}$ for all $n\ge 2$. 
Thus, $\gr_1(Q)=Q$ and $\gr_{\ge 2}(Q)=0$. The first claim now follows from 
Theorem \ref{thm:falk-ran}. 

By part \eqref{aa1}, we have that $\gr_n(K)\otimes \Q\cong \gr_n(G)\otimes \Q$ 
for $n\ge 2$. Since $b_1(G)<\infty$, the discussion from \S\ref{subsec:lcs} shows 
that  all these vector spaces are finite-dimensional. The second claim now follows from 
the definition \eqref{eq:lcs-ranks} of the LCS ranks. 
\end{proof} 

\subsection{A rational Falk--Randell theorem}
\label{subsec:rat-fr}
Returning to the general case of a split extension, $G=K\rtimes_{\varphi} Q$, 
we exploit in \cite{Su-lcs} the fact that the sequence of subgroups $L=\{L_n\}_{n\ge 1}$ 
defined above is an N-series for $K$ in order to show (using \cite{Mass})  that 
the series $\ssqrt{L}=\{\ssqrt{L_n}\}_{n\ge 1}$ is also an N-series for $K$. 
Building on this observation, and adapting the method of proof of 
Theorems \ref{thm:gu-pe} and \ref{thm:falk-ran} to this situation, 
we establish in \cite{Su-lcs} a rational version of the aforementioned results, 
as follows.

\begin{theorem}[\cite{Su-lcs}]
\label{thm:fr-rational}
Let $G=K\rtimes_{\varphi} Q$ be a split extension of groups. 
\begin{enumerate}[itemsep=2pt]
\item \label{fr1-q}
For each $n\ge 1$, there is a split extension  
$\gamma^{\rat}_n(G)=\ssqrt{L_n} \rtimes_{\varphi} \gamma^{\rat}_n(Q)$.

\item \label{fr2-q} 
If $Q$ acts trivially on $K_{\abf}$, then $\ssqrt{L} =\gamma^{\rat}(K)$, and 
$\gamma^{\rat}_n(G)=\gamma^{\rat}_n(K)\rtimes_{\varphi} \gamma^{\rat}_n(Q)$, 
for all $n\ge 1$; moreover, $\gr^{\rat}(G)=\gr^{\rat}(K) \rtimes_{\bar\varphi} \gr^{\rat}(Q)$.
\end{enumerate}
\end{theorem} 

Under additional assumptions, we have the following corollary, 
which will be needed in the proof of Theorem \ref{thm:alex-lcs-ngq}.

\begin{corollary}
\label{cor:fr-abf-q}
Let $G=K\rtimes_{\varphi} Q$ be a split extension. Assume 
$Q$ is torsion-free abelian and acts trivially on $K_{\abf}$.  
Then,
\begin{enumerate}[itemsep=1.5pt]
\item \label{aq1}
$\gamma^{\rat}_n(K)=\gamma^{\rat}_n(G)$ 
for all $n\ge 2$ and $\gr^{\rat}_{\ge 2}(K)=\gr^{\rat}_{\ge 2}(G)$. 
\item \label{aq2}
If, moreover, $b_1(G)<\infty$, then $\phi_n(K)=\phi_n(G)$ for all $n\ge 2$.
\end{enumerate}
\end{corollary}

\begin{proof}
Since $Q$ is torsion-free abelian, $\gamma^{\rat}_2(Q)=\{1\}$, 
and so $\gamma^{\rat}_n(Q)=\{1\}$ for all $n\ge 2$. 
The first claim now follows from Theorem \ref{thm:fr-rational}. 

Next, assume that $b_1(G)<\infty$. Then, by the discussion 
in \S\ref{subsec:lcs-stallings}, all the graded pieces of $\gr^{\rat}(G)\otimes \Q$
are finite-dimensional. Moreover, by Corollary \ref{cor:lcs-chen-q}, 
$\phi_n(G)=\phi^{\rat}_n(G)= \dim_{\Q} \gr^{\rat}_{n}(G)\otimes \Q$, and 
the second claim now follows from the first one.
\end{proof} 

\subsection{A mod-$p$ Falk--Randell theorem}
\label{subsec:p-fr}
Finally, we also prove in \cite{Su-lcs} mod-$p$ versions of the above 
theorems, recovering in the process a result of Bellingeri and Gervais 
from \cite{BG}. Given a split extension of groups, $G=K\rtimes_{\varphi} Q$, 
and a prime $p$, we define a sequence of subgroups of $K$, denoted 
$\big\{L^{\p}_n\big\}_{n\ge 1}$, by setting $L^{\p}_1=K$ and letting 
\begin{equation}
\label{eq:subgroup-p}
L^{\p}_{n+1}=\big\langle \big(L^{\p}_{n}\big)^p, \big[K,L^{\p}_{n}\big], 
\big[K, \gamma^{\p}_n(Q)\big], \big[L^{\p}_{n},Q\big]\big\rangle \, .
\end{equation}
Clearly, the series $L^{\p}=\big\{L^{\p}_n\big\}_{n\ge 1}$ is a {\em $p$-torsion 
series}, in the sense that $\big(L^{\p}_{n}\big)^p\subseteq L^{\p}_{n+1}$ for 
all $n\ge 1$. Moreover, we show in \cite{Su-lcs} that this is an N-series for $K$.

\begin{theorem}[\cite{Su-lcs}]
\label{thm:gr-ext-p}
Let $G=K\rtimes_{\varphi} Q$ be a split extension of groups  
and let $p$ be a prime. For each $n\ge 1$, there is then a split extension, 
$\gamma^{p}_n(G)=L^p_n \rtimes_{\varphi} \gamma^{p}_n(Q)$.
\end{theorem}

When $Q$ acts trivially on $H_1(K;\Z_p)$, we show in \cite{Su-lcs} that 
$L^p=\gamma^{p}(K)$. The next result (originally proved in \cite{BG}) 
then follows from Theorem \ref{thm:gr-ext-p}.

\begin{theorem}[\cite{BG}]
\label{thm:fr-p}
Let $G=K\rtimes_{\varphi} Q$ be a split extension, and assume 
$Q$ acts trivially on $H_1(K;\Z_p)$. Then,
\begin{enumerate}
\item \label{fr1-p}
$\gamma^{p}_n(G)=\gamma^{p}_n(K)\rtimes_{\varphi} \gamma^{p}_n(Q)$, 
for all $n\ge 1$.
\\[-10pt]
\item \label{fr2-p}
$\gr^{p}(G)=\gr^{p}(K) \rtimes_{\bar\varphi} \gr^{p}(Q)$.
\end{enumerate}
\end{theorem} 

Under additional assumptions, we have the following corollary, 
which will be needed in the proof of Theorem \ref{thm:alex-lcs-ngp}.

\begin{corollary}
\label{cor:fr-ab-p}
Let $G=K\rtimes_{\varphi} Q$ be a split extension of groups. Assume 
$Q$ is an elementary abelian $p$-group which acts trivially on $H_1(K;\Z_p)$.  
Then, 
\begin{enumerate}[itemsep=1.5pt]
\item \label{ap1}
 $\gamma^{p}_n(K)=\gamma^{p}_n(G)$ for all $n\ge 2$ and 
$\gr^{p}_{\ge 2}(K)=\gr^{p}_{\ge 2}(G)$.
\item \label{ap2}
If, moreover, $b^p_1(G)<\infty$, then
$\phi^{p}_n(K)=\phi^{p}_n(G)$ for all $n\ge 2$.
\end{enumerate}
\end{corollary}

\begin{proof}
Since $Q$ is an elementary abelian $p$-group, Example \ref{ex:p-abel} 
shows that $\gamma^{p}_n(Q)=\{1\}$ for all $n\ge 2$. The first claim 
now follows from Theorem \ref{thm:fr-p}. 

Next, assume that $b^p_1(G)<\infty$. Then, by the discussion from \S\ref{subsec:lcs-p}, 
all the graded pieces of $\gr^{p}(G)$ are finite-dimensional $\Z_p$-vector spaces. 
Recalling that $\phi^{p}_n(G)= \dim_{\Z_p} \gr^{p}_{n}(G)$, the second 
claim now follows from claim \eqref{ap1}.
\end{proof} 

\section{Ab-exact sequences}
\label{sect:ab-exact}

We now return to the general setting of not necessarily split extensions, 
and extend the notion of almost direct product to this broader context. 
For the resulting extensions, we prove one of our main results, 
which relates the Alexander invariant of a normal subgroup 
$K\triangleleft G$ to that of $G$, under suitable assumptions 
on the quotient group, $Q=G/K$. 

\subsection{Ab-exact sequences}
\label{subsec:ab-exact}
Let $H_*(G;\Z)$ denote the integral homology groups of $G$. 
As is well-known, $H_1(G,\Z)\cong G_{\ab}$, 
while $H_2(G;\Z)$, also known as the Schur multiplier of $G$, 
may be identified via the Hopf formula with 
$R\cap [F,F]/[R,F]$, where $G=F/R$ is a free presentation for $G$; 
see \cite{Brown, HS, Karpilovsky, Leedham-McKay}. 

A useful tool in the homological study of group extensions 
is the $5$-term exact sequence of Stallings \cite{St}, which 
may be derived from the exact sequence of low-degree terms 
in the Lyndon--Hochschild--Serre spectral sequence. 
Given an extension such as \eqref{eq:abc-exact}, there is an 
associated exact sequence, 
\begin{equation}
\label{eq:stallings-5}
\begin{tikzcd}[column sep=17.5pt]
 H_2(G;\Z)\ar[r, "\pi_*"]& H_2(Q;\Z) \ar[r, "\delta"]  &  
 H_1(K;\Z)_{Q} \ar[r, "\iota_*"] 
& H_1(G;\Z) \ar[r, "\pi_*"] & H_1(Q;\Z) \ar[r]  & 0\, ,
\end{tikzcd}
\end{equation}
where  $H_1(K;\Z)_{Q}=K/[G,K]$ denotes the coinvariants under the  
action of $Q$ on $K_{\ab}$ described in \S\ref{subsec:mono}. When this action 
is trivial, the group in the middle coincides with $H_1(K;\Z)$, and $\iota_*$ may 
be identified with $\iota_{\ab}\colon K_{\ab}\to G_{\ab}$. 
Since conjugation by an element of $G$ acts trivially 
on $G_{\ab}$, and therefore on $K_{\ab}$ if $\iota_{\ab}$ 
is injective, the next lemma follows.

\begin{lemma}
\label{lem:trivial action}
For a group extension, $\begin{tikzcd}[column sep=18pt]
\!\!1\ar[r] &  K \ar[r, "\iota"]
& G \ar[r, "\pi"] & Q \ar[r]  & 1,\!
\end{tikzcd}$ the following two conditions are equivalent. 
\begin{enumerate}
\item \label{ab1}
The group $Q$ acts trivially on $K_{\ab}$ and the 
map $\delta\colon H_2(Q;\Z)\to H_1(K;\Z)$ is zero.
\item \label{ab2}
The sequence
$\begin{tikzcd}[column sep=20pt]
0\ar[r] & K_{\ab} \ar[r, "\iota_{\ab}"]
& G_{\ab} \ar[r, "\pi_{\ab}"] & Q_{\ab} \ar[r] & 0
\end{tikzcd}$ is exact.
\end{enumerate}
\end{lemma}

Following \cite{DSY17}, we say that a sequence $1\to K\to G\to Q\to 1$ is 
{\em $\ab$-exact}\/ if any one of the two equivalent conditions of Lemma 
\ref{lem:trivial action} is satisfied.

\begin{remark}
\label{rem:central ex}
Suppose the extension is central, i.e, 
$K$ lies in the center of $G$. Then $Q$ acts trivially on $K=K_{\ab}$, and the 
connecting homomorphism $\delta\colon H_2(Q;\Z)\to K$ corresponds via 
the universal coefficients theorem to the cohomology class $\bar\delta\in H^2(Q;K)$ 
that classifies the central extension. Consequently, if a central extension is 
non-trivial, then the map $\delta$ is not zero, and so the extension is not $\ab$-exact. 
For instance, if $G$ is the Heisenberg group from Example \ref{ex:heis-alex}, 
then $G=F_2/\gamma_3 F_2$, and so $G$ is a central extension of 
$G_{\ab}\cong \Z^2$ by $G'\cong \Z$, with extension class 
$\bar\delta$ a generator of $H^2(\Z^2;\Z)= \Z$. 
\end{remark}

Next, we give an example of a non-central, non-split, ab-exact extension. 
We thank Thomas Koberda for help with finding this example.

\begin{example}
\label{ex:double-heis}
Let $K_i=\langle x_i,y_i,z_i\mid [x_i,y_i]=z_i, [x_i,z_i]=[y_i,z_i]=1\rangle$ 
be two copies of the Heisenberg group, and let $G=K_1\circ K_2$ 
be their central product, obtained from 
$K_1\times K_2$ by identifying the corresponding 
centers, $Z(K_i)=\langle z_i\rangle$. 
We then have a short exact sequence, $1\to K_1\to G\to \Z^2\to 1$.
It is readily seen that this is a non-central, non-split extension. 
Moreover,  $H_2(\Z^2;\Z)=\bwedge^2 \Z^2=\Z$, generated by 
$x_2\wedge y_2$, and $H_1(K_1;\Z)=\Z^2$, generated by $x_1, y_1$;  
thus, the map $\delta\colon H_2(\Z^2;\Z)\to H_1(K_1;\Z)$ is the zero map.
\end{example}

For split exact sequences, we have the following criterion for determining 
$\ab$-exactness. 

\begin{proposition}
\label{prop:ab-exact-split}
A split exact sequence $1\to K \to  G \to Q\to 1$ is $\ab$-exact if 
and only if $Q$ acts trivially on $K_{\ab}$; that is, $G=K\rtimes Q$ 
is an almost direct product.
\end{proposition}

\begin{proof}
The forward implication follows at once from 
Lemma \ref{lem:trivial action}, part \eqref{ab1}.  
For the backwards implication, 
Theorem \ref{thm:falk-ran}, part \eqref{fr2} yields a (split) exact 
sequence of abelian groups, $0\to \gr_1(K) \to \gr_1(G) \to \gr_1(Q) \to 0$, 
and this sequence clearly coincides with the one from 
Lemma \ref{lem:trivial action}, part \eqref{ab2}.
\end{proof}

It is easy to build split extensions which are not $\ab$-exact; 
the fundamental group of the Klein bottle, 
$G=\langle t,a\mid tat^{-1} =a^{-1} \rangle$,  
is of this sort. Here are two constructions that produce large  
classes of extensions which are both split exact and $\ab$-exact.

\begin{example}
\label{ex:raag-bb}
Let $G_{\Gamma}=\langle v\in V \mid \text{$[v,w] = 1$ if $\{v,w\} \in E$}\rangle$ 
be the {\em right-angled Artin group}\/ associated to a finite (simple) graph $\Gamma$ 
on vertex set $V$ and edge set $E$.  To avoid trivialities, we will always assume 
that $\abs{V}>1$. Consider the homomorphism $\pi\colon 
G_{\Gamma}\surj \Z$ that sends each generator $v\in V$ to $1\in \Z$, and let 
$N_{\Gamma}=\ker (\pi)$ be the corresponding {\em Bestvina--Brady group}.  
We then have a split exact sequence, 
\begin{equation}
\label{eq:raag-bb}
\begin{tikzcd}[column sep=24pt]
1\ar[r] & N_{\Gamma}\ar[r, "\iota"]
& G_{\Gamma} \ar[r, "\pi"] & \Z\ar[r] & 1 . 
\end{tikzcd}
\end{equation} 
As shown in \cite{BB}, the group $N_{\Gamma}$ is finitely generated 
if and only if $\Gamma$ is connected; likewise, $N_{\Gamma}$ is finitely presented 
if and only if the flag complex $\Delta_{\Gamma}$ is simply connected. 
Furthermore, as shown in \cite[Proposition 5.3]{PS-jlms07}, if $\Gamma$ 
is connected, then the group $\Z$ acts trivially on $H_1(N_{\Gamma};\Z)$, 
and so this sequence is $\ab$-exact.
\end{example}

\begin{example}
\label{ex:planes}
Let $\A=\{H_1,\dots, H_n\}$ be an arrangement of transverse planes through 
the origin of $\R^4$. Its complement, $X=\R^4\setminus \bigcup_{i=1}^{n} H_i$, 
deform-retracts onto the complement of the singularity link of $\A$ (a link 
of $n$ great circles in $S^3$). 
Moreover, as noted in \cite[Proposition 4.4]{MS-top}, $X$ is homotopy 
equivalent to the total space of a bundle over $S^1$ with fiber  
$D^2 \setminus \{\text{$n-1$ points}\}$ and monodromy given by a pure braid 
automorphism. Thus, the arrangement group, $G=\pi_1(X)$, fits into a split extension,  
$1\to F_{n-1} \to G \to \Z\to 1$, with monodromy given by an element   
of the pure braid group $P_{n-1}$.  Since pure braids act trivially 
on $H_1(F_{n-1},\Z)$, the extension is $\ab$-exact. 
\end{example}

\begin{example}
\label{ex:en}
Another class of links in $S^3$ with trivial algebraic monodromy is given 
by Eisenbud and Neumann in \cite[\S16]{EN}. The simplest example is the 
$3$-component singularity link of the analytic function 
$f\colon \C^2\to \C$, $f(z,w)=\bar{z}\bar{w}(z^3+y^2)$. The  
link is fibered, with fiber a thrice punctured sphere; since the monodromy 
acts trivially on $H_1$, the link group  is an $\ab$-exact split 
extension of the form $G=F_2\rtimes \Z$.
\end{example}

\subsection{Ab-exact sequences and Alexander invariants}
\label{subsec:ab-exact-alex}

We now relate the Alexander invariant and the derived subalgebra 
of the associated graded Lie algebra of a group $G$ with those of 
a normal subgroup $K\triangleleft G$, under suitable assumptions. 
We start with a lemma. 

\begin{lemma}
\label{lem:ab-exact-derived}
If $\begin{tikzcd}[column sep=16pt]
\!1\ar[r] & K\ar[r, "\iota"]
& G \ar[r, "\pi"] & Q\ar[r] & 1\!
\end{tikzcd}$ is an $\ab$-exact sequence, then its restriction 
to derived subgroups, $\begin{tikzcd}[column sep=16pt]
\!1\ar[r] & K'\ar[r, "\iota'"]
& G' \ar[r, "\pi'"] & Q'\ar[r] & 1\!\!
\end{tikzcd}$, is an exact sequence.
\end{lemma}

\begin{proof}
Follows from the exactness of the sequence 
$\begin{tikzcd}[column sep=18.3pt]
\!0\ar[r] & K_{\ab} \ar[r, "\iota_{\ab}"]
& G_{\ab} \ar[r, "\pi_{\ab}"] & Q_{\ab} \ar[r] & 0
\end{tikzcd}$ and the Nine Lemma in the category of groups.
\end{proof}

For a homomorphism $\psi\colon G\to H$, we let 
$\bar\psi\colon G/G'' \to H/H''$ be the induced 
homomorphism on maximal metabelian quotients.

\begin{theorem}
\label{thm:alex-abex}
Let $\begin{tikzcd}[column sep=16pt]
\!1\ar[r] & K\ar[r, "\iota"]
& G \ar[r, "\pi"] & Q\ar[r] & 1\!
\end{tikzcd}$ 
be an $\ab$-exact sequence. Assuming $Q$ is abelian, the 
following hold.
\begin{enumerate}[itemsep=2pt]
\item \label{ng1}
The inclusion $\iota\colon K\inj G$ restricts to an equality, $K'=G'$.%
\item \label{ng2}
The induced map on Alexander invariants, 
$B(\iota)\colon B(K) \to B(G)$, gives rise to a 
$\Z[K_{\ab}]$-linear isomorphism, $B(K) \to B(G)_{\iota}$.
\item \label{ng3}
The sequence
$\begin{tikzcd}[column sep=16pt]
\!1\ar[r] & K/K'' \ar[r, "\bar\iota"]
& G/G'' \ar[r, "\bar\pi"] & Q\ar[r] & 1\!
\end{tikzcd}$
is also $\ab$-exact.
\item  \label{ng4}
If, moreover, $G_{\ab}$ is finitely generated, then 
$\theta_n(K)\le \theta_n(G)$ for all $n\ge 1$.
\end{enumerate}
\end{theorem}

\begin{proof}
\eqref{ng1} 
By Lemma \ref{lem:ab-exact-derived}, we have an exact sequence, 
$\begin{tikzcd}[column sep=16pt]
\!1\ar[r] & K'\ar[r, "\iota'"]
& G' \ar[r, "\pi'"] & Q'\ar[r] & 1\!\!\end{tikzcd}$. But 
$Q'=\{1\}$ by assumption, and so $K'=G'$.
\vs

\eqref{ng2} Since $K'=G'$, we must also have $(K')'=(G')'$. Hence, 
$K'/K''=G'/G''$, showing that the map $B(K)\to  B(G)_{\iota}$ is 
indeed a $\Z[K_{\ab}]$-linear isomorphism.
\vs

\eqref{ng3}  As we just saw, $K''=G''$; hence, the map $\bar\iota$ 
is injective, and so the sequence in question is exact.  Under the  
identifications $(K/K'')_{\ab}=K_{\ab}$,  $(G/G'')_{\ab}=G_{\ab}$, 
and $Q_{\ab}=Q$, the maps $\bar\iota_{\ab}$ and $\bar\pi_{\ab}$ 
coincide with $\iota_{\ab}$ and $\pi_{\ab}$, respectively, and the 
claim follows.

\eqref{ng4} 
By our $\ab$-exactness assumption, the homomorphism 
$\iota_{\ab}\colon K_{\ab}\to G_{\ab}$ is injective; 
therefore, $\theta_1(K)\le \theta_1(G)$ and $\tilde\iota_{\ab}$ 
is also injective. It follows from part \eqref{ng2} 
that the map $B(\iota)\colon B(K)\to B(G)$ factors as 
an isomorphism of $\Z[K_{\ab}]$-modules, $B(K)\isom B(G)_{\iota}$, 
followed by the identity map of $B(G)$, viewed as covering 
the ring map $\tilde\iota_{\ab}$. 
Passing to associated graded modules, the morphism 
$\gr(B(\iota))\colon \gr(B(K))\to \gr(B(G))$ factors as 
an isomorphism of $\gr(\Z[K_{\ab}])$-modules, 
followed by the identity map of $\gr(B(G))$, viewed as covering 
the ring map $\gr(\tilde\iota_{\ab})$. By Lemma \ref{lem:inj}, 
$\gr(\tilde\iota_{\ab})$ is injective; therefore, $\gr(B(\iota))$ 
is also injective, and so the rank of $\gr_n(B(K))$ 
is less or equal to the rank of $\gr_n(B(G))$ for $n\ge 0$. 
Formula \eqref{eq:hilb-b-chen} now implies 
that $\theta_n(K)\le \theta_n(G)$ for $n\ge 2$, and we are done.
\end{proof}

When the above $\ab$-exact sequence splits, more can be said.

\begin{corollary}
\label{cor:alex-abex-split}
Let $\begin{tikzcd}[column sep=16pt]
\!1\ar[r] & K\ar[r, "\iota"]
& G \ar[r, "\pi"] & Q\ar[r] & 1\!
\end{tikzcd}$ 
be a split, $\ab$-exact sequence, and assume $Q$ is abelian. Then 
\begin{enumerate}
\item \label{aas1}
The map $\iota$ induces isomorphisms of graded Lie algebras, 
$\gr_{\ge 2}(K) \isom \gr_{\ge 2}(G)$ and 
$\gr_{\ge 2}(K/K'')  \isom  \gr_{\ge 2}(G/G'')$. 
\item  \label{aas2}
If, moreover, $b_1(G)<\infty$, then $\phi_n(K)=\phi_n(G)$ and 
$\theta_n(K)=\theta_n(G)$ for all $n\ge 2$.
\end{enumerate}
\end{corollary}

\begin{proof}
Let $\sigma\colon Q\to G$ be a splitting of $\pi$.
By Theorem \ref{thm:alex-abex}, part \eqref{ng3}, 
the extension $1\to K/K''\to G/G'' \to Q\to 1$ 
is $\ab$-exact; it also admits a splitting, obtained by 
composing the projection $G\to G/G''$ with $\sigma$.  
Thus, by Proposition \ref{prop:ab-exact-split}, both extensions 
are almost direct products.  Since $Q$ is assumed to be abelian, 
both claims now follow from Corollary \ref{cor:fr-ab}.
\end{proof}

\begin{remark}
\label{rem:bb-artin}
For extensions of the form $1\to N_{\Gamma} \to G_{\Gamma} \to \Z\to 1$, 
with $G_{\Gamma}$ the right-angled Artin group and 
 $N_{\Gamma}$ the Bestvina--Brady group associated 
to a finite connected graph $\Gamma$ as in 
Example \ref{ex:raag-bb}, Theorem \ref{thm:alex-abex}, 
parts \eqref{ng1}--\eqref{ng3}
and Corollary \ref{cor:alex-abex-split} recover 
Propositions 4.2 and 5.4  and Theorem 5.6 from \cite{PS-jlms07}. 
\end{remark}

\section{Abf-exact sequences}
\label{sect:abf-exact}

In this section we give analogues of the above results for the rational lower central 
series, the rational derived series, and the rational Alexander invariant.

\subsection{Abf-exact sequences}
\label{subsec:split-q}
We start with a lemma/definition, the proof of which is exactly similar to 
that of Lemma \ref{lem:trivial action}. 

\begin{lemma}
\label{lem:trivial action-q}
For an exact sequence $\begin{tikzcd}[column sep=16pt]
\!1\ar[r] & K\ar[r, "\iota"]
& G \ar[r, "\pi"] & Q\ar[r] & 1 ,\!
\end{tikzcd}$ the following two conditions are equivalent.
\begin{enumerate}
\item \label{abf1}
The group $Q$ acts trivially on $K_{\abf}$ and the 
composite $\delta_0\colon H_2(Q;\Z)\xrightarrow{\delta} 
H_1(K;\Z) \surj H_1(K;\Z)/\Tors$ is zero.
\item \label{abf2}
The sequence
$\begin{tikzcd}[column sep=22pt]
\!0\ar[r] & K_{\abf} \ar[r, "\iota_{\abf}"]
& G_{\abf} \ar[r, "\pi_{\abf}"] & Q_{\abf} \ar[r] & 0\!
\end{tikzcd}$ is exact.
\end{enumerate}
\end{lemma}

We say that the sequence $1\to K\to G\to Q\to 1$ is {\em $\abf$-exact}\/ if 
any one of the above two conditions is satisfied. Evidently, $\ab$-exactness implies 
$\abf$-exactness, though the converse is not true, as illustrated by the 
infinite dihedral group, $D_{\infty}=\Z\rtimes \Z_2$. 

Let $\delta_{\Q}\colon H_2(Q;\Q)\to H_1(K;\Q)$ be the connecting 
homomorphism in the $5$-term exact sequence 
\eqref{eq:stallings-5} with $\Q$-coefficients. In the case when 
$K_{\abf}$ has finite rank, we have the following, more convenient 
criterion for $\abf$-exactness. 

\begin{lemma}
\label{lem:trivial-rat-action}
Let $1\to K\to G\to Q\to 1$ be a group extension, and 
suppose $K_{\abf}$ is finitely generated. Then the extension 
is $\abf$-exact if and only if $Q$ acts trivially on $H_1(K;\Q)$ 
and the map $\delta_{\Q} \colon H_2(Q;\Q)\to H_1(K;\Q)$ is zero.
\end{lemma}

\begin{proof}
First note that $H_1(K;\Q)=K_{\abf}\otimes \Q$ and the action of $Q$ on 
$H_1(K;\Q)$ is obtained by extension of scalars from the action of 
$Q$ on $K_{\abf}$. Likewise, the map $\delta_{\Q}$ is obtained by 
extension of scalars from $\delta_0$. The forward implication follows 
at once (for any $K$). 

For the reverse implication, note that our assumption on $K_{\abf}$  
implies that this group is a (maximal rank) lattice in the finite-dimensional 
$\Q$-vector space $H_1(K;\Q)=K_{\abf}\otimes \Q$.  Thus, if $Q$ acts 
trivially on $H_1(K;\Q)$, it must also act trivially on $K_{\abf}$, and likewise, 
if $\delta_{\Q}=0$, then $\delta_0=0$.
\end{proof}

Without the finite generation assumption on $K_{\abf}$,  
the triviality of the action of $Q$ on $H_1(K;\Q)$ does not insure 
triviality of the action on $K_{\abf}$.  We illustrate this phenomenon 
with an example.
 
\begin{example}
\label{ex:baumslag-solitar}
For each $n\ge 2$, let $G=\text{BS}(1,n)$ be the metabelian Baumslag--Solitar 
group with presentation $G=\langle t,a\mid tat^{-1}=a^n\rangle$. In 
the extension $1\to G' \to G\to G_{\ab} \to 1$, the abelianization is isomorphic 
to $\Z$,  generated by the image of $t$, while the derived subgroup is 
isomorphic to $\Z[1/n]$, normally generated by $a$. Thus, the 
extension is split exact, with monodromy given by $a\mapsto a^n$. 
Clearly, $\Z$ acts trivially on $\Z[1/n]\otimes \Q=\Q$, though it acts 
non-trivially on the torsion-free, yet non-finitely generated abelian 
group $(G')_{\abf}=\Z[1/n]$. 
\end{example}

Nevertheless, we have the following criteria that insure $\abf$-exactness 
of a split exact sequence.

\begin{proposition}
\label{prop:abf-exact-split}
Let $1\to K\to G \to Q\to 1$ be a split exact sequence.
\begin{enumerate}
\item \label{ssq1} 
The sequence is $\abf$-exact if and only if $Q$ acts trivially on $K_{\abf}$.
\item \label{ssq2} 
If $K_{\abf}$ is finitely generated, then the sequence is 
$\abf$-exact if and only if $Q$ acts trivially on $H_1(K;\Q)$. 
\end{enumerate}
\end{proposition}

\begin{proof}
The proof of the first claim is similar to the one of 
Proposition \ref{prop:ab-exact-split}, using 
Theorem \ref{thm:fr-rational}, part  \eqref{fr2-q}, instead. 
The second claim now follows from Lemma \ref{lem:trivial-rat-action}.
\end{proof}

We call semidirect products $G=K\rtimes_{\varphi} Q$ that satisfy condition 
\eqref{ssq1} from above, {\em rational almost-direct products}. Clearly, 
any split extension $G=K\rtimes Q$ with $K$ finite is of this type.
Using Proposition \ref{prop:bl-grlie} and Theorem \ref{thm:fr-rational}, 
we obtain the following corollary.

\begin{corollary}
\label{cor:fr-abf-rat}
Let $G=K\rtimes_{\varphi} Q$ be a split extension such that 
$K_{\abf}$ is finitely generated and $Q$ acts trivially on $H_1(K;\Q)$. Then 
$\gr(G)\otimes \Q =\gr(K)\otimes \Q \rtimes_{\bar\varphi} \gr(Q) \otimes \Q$.
\end{corollary}

\begin{example}
\label{ex:artin-ker}
Let $\Gamma$ be a connected, finite simple graph, and 
let $\chi\colon G_{\Gamma} \surj \Z$ be an arbitrary epimorphism. 
The subgroup $N_{\chi}\coloneqq \ker (\chi)$ 
is called an {\em Artin kernel}; it generalizes the 
Bestvina--Brady construction, and fits into a 
split exact sequence, 
\begin{equation}
\label{eq:artin-ker}
\begin{tikzcd}[column sep=24pt]
1\ar[r] & N_{\chi}\ar[r, "\iota"]
& G_{\Gamma} \ar[r, "\chi"] & \Z\ar[r] & 1 . 
\end{tikzcd}
\end{equation}
Suppose $\Z$ acts trivially on $H_1(N_{\chi};\Q)$---a condition which can 
efficiently be tested by means of  \cite[Theorem 6.2]{PS-adv09}. 
Then, as shown in \cite[Lemma 9.1(1)]{PS-adv09}, the group $N_{\chi}$ is 
finitely generated; thus, by Proposition \ref{prop:abf-exact-split}, 
part \eqref{ssq2}, the above sequence is $\abf$-exact.  
Furthermore, as shown in \cite[Lemma 9.1(3)]{PS-adv09}, the group $\Z$ also acts 
trivially on $H_1(N_{\chi};\Z)$, and so the sequence \eqref{eq:artin-ker} 
is actually $\ab$-exact. Applying Corollary \ref{cor:alex-abex-split} in 
this setting recovers Proposition 9.2 from \cite{PS-adv09}. 
\end{example}

As we just saw, for extensions of type \eqref{eq:artin-ker}, there is 
no difference between $\ab$-exactness and $\abf$-exactness. 
In general, though, the two concepts are different, even when the 
group $K$ is torsion-free. An example is provided by $G=K\rtimes \Z$ with 
$K=\langle t,a\mid tat^{-1} =a^{-1} \rangle$ and $\Z=\langle u \rangle$ 
acting by $utu^{-1}=ta$ and $uau^{-1}=a$, which is a rational almost direct 
product, but not an almost direct product. 

\subsection{Abf-exact sequences and Alexander invariants}
\label{subsec:abf-exact-alex}

We now relate the rational Alexander invariant and the derived rational 
associated graded Lie algebra of a group to those of a normal 
subgroup, under suitable assumptions. We start with a lemma, 
whose proof is similar to that of Lemma \ref{lem:ab-exact-derived}.

\begin{lemma}
\label{lem:ab-exact-derived-rat}
If $\begin{tikzcd}[column sep=16pt]
\!\!1\ar[r] & K\ar[r, "\iota"]
& G \ar[r, "\pi"] & Q\ar[r] & 1\!\!
\end{tikzcd}$ is an $\abf$-exact sequence, then its restriction 
to $\Q$-derived subgroups, $\begin{tikzcd}[column sep=16pt]
\!\!1\ar[r] & K'_{\rat}\ar[r, "\iota'"]
& G'_{\rat} \ar[r, "\pi'"] & Q'_{\rat}\ar[r] & 1\!\!
\end{tikzcd}$, is an exact sequence.
\end{lemma}

\begin{theorem}
\label{thm:alex-lcs-ngq}
Let $\begin{tikzcd}[column sep=16pt]
1\ar[r] & K\ar[r, "\iota"]
& G \ar[r, "\pi"] & Q\ar[r] & 1
\end{tikzcd}$  be an $\abf$-exact sequence, and assume 
$Q$ is a torsion-free abelian group. Then,
\begin{enumerate}[itemsep=2pt]
\item \label{ngq1}
The inclusion $\iota\colon K\inj G$ restricts to an equality 
$K'_{\rat}=G'_{\rat}$.
\item \label{ngq2}
The induced map on rational Alexander invariants, 
$B_{\rat}(\iota)\colon B_{\rat}(K) \to B_{\rat}(G)$, 
gives rise to a $\Z[K_{\abf}]$-linear isomorphism,  
$B_{\rat}(K) \to B_{\rat}(G)_{\iota}$.

\item \label{ngq3}
The sequence
$
\begin{tikzcd}[column sep=16pt]
\!\!1\ar[r] & K/K''_{\rat} \ar[r, "\bar\iota"]
& G/G''_{\rat} \ar[r, "\bar\pi"] & Q \ar[r] & 0\!
\end{tikzcd}
$
is also $\abf$-exact.

\item  \label{ngq4}
If, moreover, $G_{\abf}$ is finitely generated, then 
$\theta_n(K)\le \theta_n(G)$ for all $n\ge 1$.
\end{enumerate}
\end{theorem}

\begin{proof}
\eqref{ngq1}  
Lemma \ref{lem:ab-exact-derived-rat}, together with the assumption 
that $Q'_{\rat}=\{1\}$ shows that $K'_{\rat}=G'_{\rat}$.
\vs

\eqref{ngq2}  
It follows that $K''_{\rat}=G''_{\rat}$, too, and so 
$K'_{\rat}/K''_{\rat}=G'_{\rat}/G''_{\rat}$, whence the claim.
\vs

\eqref{ngq3}  
Since $K''_{\rat}=G''_{\rat}$, the map $\bar\iota$ is 
injective, and so the sequence in question is exact.  Under the  
identifications $(K/K''_{\rat})_{\abf}=K_{\abf}$ and 
$(G/G''_{\rat})_{\abf}=G_{\abf}$, the map $\bar\iota_{\abf}$ 
coincides with $\iota_{\abf}$, and the claim follows.

\eqref{ngq4} 
By assumption, the map $\iota_{\abf}\colon K_{\abf}\to G_{\abf}$ 
is injective; therefore, $\tilde\iota_{\abf}\colon \Z[K_{\abf}]\to \Z[G_{\abf}]$, 
is also injective, and $\theta_1(K)\le \theta_1(G)$.  It follows from part \eqref{ngq2} 
that the map $B_{\rat}(\iota)\colon B_{\rat}(K)\to B_{\rat}(G)$ factors as 
an isomorphism of $\Z[K_{\abf}]$-modules, $B_{\rat}(K)\isom B_{\rat}(G)_{\iota}$, 
followed by the identity map of $B_{\rat}(G)$, viewed as covering 
the ring map $\tilde\iota_{\abf}$. Hence, the morphism 
$\gr(B_{\rat}(\iota))\colon \gr(B_{\rat}(K))\to \gr(B_{\rat}(G))$ factors as 
an isomorphism of $\gr(\Z[K_{\abf}])$-modules, 
followed by the identity map of $\gr(B_{\rat}(G))$, viewed as covering 
the ring map $\gr(\tilde\iota_{\abf})$. The proof of Lemma \ref{lem:inj}  
shows that $\gr(\tilde\iota_{\abf})$ is injective; thus, $\gr(B_{\rat}(\iota))$ 
is also injective. Corollary \ref{cor:chen-alrat} now implies 
that $\theta_n(K)\le \theta_n(G)$ for $n\ge 2$, and we are done.
\end{proof}

\begin{corollary}
\label{cor:alex-abfex-split}
Let $\begin{tikzcd}[column sep=16pt]
\!1\ar[r] & K\ar[r, "\iota"]
& G \ar[r, "\pi"] & Q\ar[r] & 1\!
\end{tikzcd}$ 
be a split, $\abf$-exact sequence, and assume   
$Q$ is torsion-free abelian. Then 
\begin{enumerate}[itemsep=2pt]
\item \label{ngq4a} 
The map $\iota$ induces isomorphisms of graded Lie algebras,
$\gr^{\rat}_{\ge 2}(K) \isom \gr^{\rat}_{\ge 2}(G)$ and 
$\gr^{\rat}_{\ge 2}(K/K'') \isom  \gr^{\rat}_{\ge 2}(G/G'')$. 
\item  \label{ngq4b} 
If $b_1(G)<\infty$,  then $\phi_n(K)=\phi_n(G)$ and 
$\theta_n(K)=\theta_n(G)$ for all $n\ge 2$.
\end{enumerate}
\end{corollary}

\begin{proof}
By assumption, the given extension 
is $\abf$-exact and admits a splitting, say, $\sigma\colon Q\to G$.
By Theorem \ref{thm:alex-lcs-ngq}, part \eqref{ngq3}, 
the extension $1\to K/K''_{\rat} \to G/G''_{\rat} \to Q\to 1$ 
is $\abf$-exact; it is also split exact, with splitting obtained 
by composing the projection $G\to G/G''_{\rat}$ with $\sigma$.  
Thus, by Proposition \ref{prop:abf-exact-split}, both extensions 
are $\Q$-almost direct products.  Since the group $Q$ is assumed 
to be torsion-free abelian, both claims now follow from 
Corollary \ref{cor:fr-abf-q}. 
\end{proof}

The above corollary has the following topological consequence. 
Given a space $X$ and a map $f\colon X\to X$, we let 
$T_f=X\times [0,1]/(x,0)\sim (f(x),1)$ be the mapping torus of $f$. 

\begin{corollary}
\label{cor:circle-fibration}
Let $X$ be a connected CW-complex such that $b_1(X)<\infty$, 
let $f\colon X\to X$ be a map inducing the identity on $H_1(X;\Q)$. 
Then $\phi_n(\pi_1(X))=\phi_n(\pi_1(T_f))$ and 
$\theta_n(\pi_1(X))=\theta_n(\pi_1(T_f))$ for all $n\ge 2$.
\end{corollary}

\begin{proof}
The projection map $\pr_1\colon X\times [0,1]\to X$ induces 
a bundle map, $T_f\to S^1$, with fiber $X$ and monodromy $f$. 
The homotopy long exact sequence for the fibration 
$X\to T_f\to S^1$ yields a short exact sequence, 
$1\to \pi_1(X)\to \pi_1(T_f)\to \Z\to 1$. Clearly, all 
the assumptions of Corollary \ref{cor:alex-abfex-split} 
are satisfied for this sequence, and the claim follows.
\end{proof}

\section{$p$-Exact sequences}
\label{sect:p-exact}

In this section we give analogues of the results from the previous two sections 
for the mod-$p$ lower central series, the derived $p$-series, and the mod-$p$ 
Alexander invariant. We start with a lemma/definition. Let 
$\delta_p\colon H_2(Q;\Z_p)\to H_1(K;\Z_p)$ be the connecting 
homomorphism in the $5$-term Stallings exact sequence 
\eqref{eq:stallings-5} with $\Z_p$-coefficients.

\begin{lemma}
\label{lem:trivial action-p}
For an exact sequence $\begin{tikzcd}[column sep=16pt]
\!1\ar[r] & K\ar[r, "\iota"]
& G \ar[r, "\pi"] & Q\ar[r] & 1\!
\end{tikzcd}$ and a prime $p$, 
the following two conditions are equivalent.
\begin{enumerate}
\item \label{abp1}
The group $Q$ acts trivially on $H_1(K;\Z_p)$ and the 
homomorphism $\delta_p\colon H_2(Q;\Z_p)\to H_1(K;\Z_p)$ is zero.
\item \label{abp2}
The sequence
$\begin{tikzcd}[column sep=18pt]
\!0\ar[r] & H_1(K;\Z_p)\ar[r, "\iota_{*}"]
&H_1(G;\Z_p) \ar[r, "\pi_{*}"] & H_1(Q;\Z_p) \ar[r] & 0\!\!
\end{tikzcd}$ is exact.
\end{enumerate}
\end{lemma}

We say that the sequence $1\to K\to G\to Q\to 1$ 
is {\em $p$-exact }\/ if any one of the above two 
conditions is satisfied. A non-example is given by 
the central, non-split extension $0\to \Z_p\to \Z_{p^2}\to \Z_p\to 0$; 
the split exact sequence $1\to \Z_3\to S_3\to \Z_2\to 1$ is 
$2$-exact but not $3$-exact. Entirely similar arguments 
as before yield the following results.

\begin{lemma}
\label{lem:ab-exact-derived-p}
If $1\to K\to G\to Q\to 1$ is a $p$-exact sequence, then its restriction 
to $p$-derived subgroups, $1\to K'_{p} \to 
G'_{p} \to Q'_{p} \to 1$, is an exact sequence.
\end{lemma}

\begin{proposition}
\label{prop:p-exact-split}
A split exact sequence $1\to K \to G\to Q\to 1$ is 
$p$-exact if and only if $Q$ acts trivially on $K_{\ab}\otimes\Z_p=H_1(K;\Z_p)$.
\end{proposition}

We call semidirect products $G=K\rtimes Q$ such as these 
{\em mod-$p$ almost-direct products}. Every almost direct 
product is a mod-$p$ almost-direct product (for any prime $p$), 
but the converse is not true. For instance, take again the 
Klein bottle group, $G=\langle t,a\mid tat^{-1} =a^{-1} \rangle$; 
then $\Z=\langle t\rangle$ acts non-trivially on $\Z=\langle a\rangle$, 
but acts trivially on $\Z\otimes \Z_2=\Z_2$. 

\begin{theorem}
\label{thm:alex-lcs-ngp}
Let $1\to K \to G\to Q\to 1$ be a $p$-exact sequence. Assume 
$Q$ is an elementary abelian $p$-group. Then,
\begin{enumerate}[itemsep=2pt]
\item \label{ngp1}
The inclusion $\iota\colon K\inj G$ restricts to an equality 
$K'_{p}=G'_{p}$.

\item \label{ngp2}
The induced map on $p$-Alexander invariants, $B_{p}(K) \to B_{p}(G)$, 
gives rise to a $\Z_p[H_1(K;\Z_p)]$-linear isomorphism, 
$B_{p}(K) \to B_{p}(G)_{\iota}$.

\item \label{ngp3}
The sequence $\begin{tikzcd}[column sep=16pt]
\!\!1\ar[r] & K/K''_{p} \ar[r, "\bar\iota"]
& G/G''_{p} \ar[r, "\bar\pi"] & Q \ar[r] & 0\!
\end{tikzcd}$ is also $p$-exact.

\item \label{ngp4} 
If, moreover, $b_1^p(G)<\infty$, then 
$\theta^p_n(K)\le \theta^p_n(G)$ for all $n\ge 1$.
\end{enumerate}
\end{theorem}

\begin{proof}
\eqref{ngp1}
By Lemma \ref{lem:ab-exact-derived-p}, the given sequence induces 
an exact sequence at the level of $p$-derived subgroups, 
$1\to K'_{p} \to G'_{p} \to Q'_{p} \to 1$.  Since $Q'_p=\langle Q^p, Q'\rangle$, 
our assumption on $Q$ implies that $ Q'_{p}=\{1\}$, and the claim is proved. 

\eqref{ngp2}
It follows from part \eqref{ngp1} that $K''_{p}=G''_{p}$, too; hence, 
$K'_{p}/K''_{p}=G'_{p}/G''_{p}$.

\eqref{ngp3}  
Since $K''_p=G''_p$, the map $\bar\iota$ is injective, and so the 
sequence in question is exact.  Upon identifying 
$H_1(K/K'';\Z_p)=H_1(K;\Z_p)$, $H_1(G/G'';\Z_p)=H_1(G,\Z_p)$, and 
$H_1(Q;\Z_p)=Q$, the maps $\bar\iota$ and $\bar\pi$ coincide with 
$\iota_{*}$ and $\pi_*$, respectively, and the claim follows.

\eqref{ngp4} 
By assumption, the map $\iota_{*}\colon H_1(K;\Z_p)\to H_1(G;\Z_p)$ 
is injective; therefore, its extensions to groups rings, $\tilde\iota_{*}$, 
is also injective and $\theta^p_1(K)\le \theta^p_1(G)$.  By part \eqref{ngq2}, 
the map $B_{p}(\iota)\colon B_{p}(K)\to B_{p}(G)$ factors as 
an isomorphism of $\Z_p[H_1(K;\Z_p)]$-modules, $B_{p}(K)\isom B_{p}(G)_{\iota}$, 
followed by the identity map of $B_{p}(G)$, viewed as covering 
the ring map $\tilde\iota_{*}$. Proceeding as before, 
we infer from  \eqref{eq:hilb-b-chen-p} that $\theta^p_n(K)\le \theta^p_n(G)$ 
for $n\ge 2$, and we are done.
\end{proof}

\begin{corollary}
\label{cor:alex-lcs-ngp-split}
Let $1\to K \to G\to Q\to 1$ be a split, $p$-exact sequence. Assume 
$Q$ is an elementary abelian $p$-group. Then,
\begin{enumerate}[itemsep=3pt]
\item \label{ngp4a} 
The inclusion $\iota\colon K\inj G$ induces isomorphisms of graded 
Lie algebras over $\Z_p$, 
$\gr^{p}_{\ge 2}(K) \isom \gr^{p}_{\ge 2}(G)$ and 
$\gr^{p}_{\ge 2}(K/K'') \isom  \gr^{p}_{\ge 2}(G/G'')$.
\item \label{ngp4b} 
If $b_1^p(G)<\infty$, then 
$\phi^{p}_n(K)=\phi^{p}_n(G)$ and $\theta^{p}_n(K)=\theta^{p}_n(G)$ 
for all $n\ge 2$.
\end{enumerate}
\end{corollary}

\begin{proof}
By Theorem \ref{thm:alex-lcs-ngp}, part \eqref{ngp3}, 
the extension $1\to K/K''_p\to G/G''_p \to Q\to 1$ 
is $p$-exact; it is also split exact, with splitting obtained by 
composing the projection $G\surj G/G''_p$ with a splitting 
$Q\to G$ of the extension $1\to K \to G\to Q\to 1$. 
Thus, by Proposition \ref{prop:p-exact-split}, 
both extensions  are $p$-almost direct products.  
Since $Q$ is an elementary 
abelian $p$-group and acts trivially on $H_1(K;\Z_p)$, both 
claims now follow from Corollary \ref{cor:fr-ab-p}.
\end{proof}

\part{Characteristic varieties}
\label{part:cjl}

\section{Jump loci for rank $1$ local systems}
\label{sect:cvs}

We now switch our attention to the cohomology jump loci associated 
to a finitely generated group. We start with a quick review 
of some of the basic theory of the characteristic varieties.

\subsection{A stratification of the character group}
\label{subsec:char-strat}

Throughout this section, $G$ will be a finitely generated group. 
The character group, $\TT_G=\Hom(G,\C^*)$, is an abelian, 
complex  algebraic group, with identity $1$ the trivial representation. The  
coordinate ring of $\TT_G$ is the group algebra $\C[G_{\ab}]$;  
thus, we may identify $\TT_G$ with $\Spec(\C[G_{\ab}])$, 
the maximal spectrum of this $\C$-algebra. Since each 
character $\rho\colon G\to \C^*$ factors through the abelianization 
$G_{\ab}$, the map $\ab\colon G\surj G_{\ab}$ induces an isomorphism, 
$\ab^*\colon \TT_{G_{\ab}} \isom \TT_{G}$.  

Let $X$ be a connected CW-complex with finite $1$-skeleton 
and with $\pi_1(X)=G$. 
Upon identifying a point $\rho \in \TT_G$ with a rank one local 
system $\C_{\rho}$ on $X$, we define for each $k\ge 1$ the 
{\em depth $k$ characteristic variety}\/ of $G$ as 
\begin{equation}
\label{eq:cvar}
\VV_k(G)\coloneqq \{ \rho \in \TT_G \mid  
\dim_{\C} H_1(X, \C_{\rho}) \ge k \} .  
\end{equation}
Clearly, $1\in \VV_k(G)$ if and only if $b_1(G)\ge k$. Furthermore, 
we have a descending filtration of the character group, 
\begin{equation}
\label{eq:cvar-strat}
\TT_G \supseteq \VV_1(G)  \supseteq \VV_2(G)  
\supseteq \cdots \supseteq \VV_k(G)  \supseteq \cdots  .
\end{equation}

Since a classifying space $K(G,1)$ may be constructed by attaching to $X$ 
cells of dimension $3$ and higher, it is straightforward to verify that the 
sets $\VV_k(G)$ do not depend on the choice of a space $X$ as above. 
Furthermore, since $H_1(X, \C_{\rho}) \cong H^1(X, \C_{\rho^{-1}})$, we 
may replace in \eqref{eq:cvar} homology with cohomology and obtain 
the same sets.

Denoting by $\TT_G^0$ the identity component of  $\TT_G$, we have 
an isomorphism $\abf^*\colon \TT_{G_{\abf}} \isom \TT_{G}^0$.  It is readily 
seen that $\TT^0_G$ is a complex affine torus of dimension $r=\rank G_{\ab}$, 
and that $\T_G$ is a disjoint union of such tori, indexed by the finite group 
$\Tors (G_{\ab})$. We define the {\em restricted characteristic varieties}\/ 
to be the traces of the depth-$k$ characteristic varieties on this complex torus,
\begin{equation}
\label{eq:cvar-restricted}
\WW_k(G)\coloneqq \V_k(G)\cap \T_G^0 \, .
\end{equation}
By construction, we have an inclusion $\WW_k(G)\subseteq \VV_k(G)$, 
which becomes an equality if $G_{\ab}$ is torsion-free. 
The inclusion may be strict, in general.

\begin{example}
\label{ex:b1=0}
Suppose $b_1(G)=0$; then $\TT_G$ is a finite set, 
in bijection with $G_{\ab}$, while  $\TT^0_G=\{1\}$. 
Although in this case $1\notin \VV_1(G)$, 
and so $\WW_1(G)=\emptyset$, the set 
$\VV_1(G)$ may be non-empty.  For instance, 
if $G=\Z_2*\Z_2$, then $\TT_G=\{(\pm 1,\pm 1)\}$ 
and $\VV_1(G)=\{(-1,-1)\}$.
\end{example}

\subsection{Alexander matrices and Fitting ideals}
\label{subsec:fitting}
In Lemma 2.2.3 and Corollary 2.4.3 from \cite{Hi97}, E.~Hironaka 
showed that the characteristic varieties of a finitely presented group $G$ 
are Zariski closed subsets of the character group $\T_G$. For completeness, 
we provide a shorter proof of this important result, 
in the more general context that we have adopted here 
(see also \cite[Proposition 2.4]{DPS-imrn} for a related
argument). 

Given a finitely generated module $M$ over a commutative ring $R$, 
we let $\Fitt_k(M)$ denote the Fitting ideal of codimension 
$k-1$ minors in a presentation matrix for $M$.  In the next 
lemma, $M$ will be the Alexander module 
$A(G)=\Z[G_{\ab}]\otimes_{\Z[G]} I(G)$ from \eqref{eq:alex-mod}, 
viewed as a module over the ring $R=\Z[G_{\ab}]$. For an ideal 
$\mathfrak{a}$ of $\C[G_{\ab}]=R\otimes \C$, we denote by 
$V(\mathfrak{a})\subset \T_G$ its zero-set.
 
\begin{lemma}[\cite{Hi97}]
\label{lem:hironaka}
Let $G$ be a finitely generated group. Then, for all $k\ge 1$,
\[
\V_{k}(G)=V(\Fitt_{k+1}(A(G)\otimes \C)) \, ,
\]
at least away from $1\in \T_G$, with equality at $1$ for $k<b_1(G)$. 
\end{lemma}

\begin{proof}
Pick a presentation for $G$ with generators $x_1,\dots, x_m$; 
let $X$ be the corresponding presentation $2$-complex, and let  
$(C_{\hdot}(X^{\ab};\Z), \partial^{\ab})$ be the $\Z[G_{\ab}]$-equivariant 
chain complex of the maximal abelian cover $X^{\ab}$, as displayed   
in \eqref{eq:abcover-cc}.
By definition, a character $\rho\colon G_{\ab} \to \C^*$ belongs to 
$\VV_{k}(G)$ precisely when 
$\rank \partial^{\ab}_{2}(\rho) + \rank \partial^{\ab}_{1}(\rho) \le m-k$, 
where the evaluation of $\partial^{\ab}_i$ at 
$\rho$ is obtained by applying the ring morphism $\Z[G_{\ab}]\to \C$, 
$g\mapsto \rho(g)$ to each entry.  Hence, $\VV_k(X)$ is the 
zero-set of the ideal of minors of size $m-k+1$ of the block-matrix 
$\partial^{\ab}_{2} \oplus \partial^{\ab}_{1}$.  

Now, $\partial^{\ab}_{1}(\rho)=0$ if and only if $\rho=1$. Therefore, 
if $\rho$ is a non-trivial character, then $\rho\in \V_k(G)$ if and only if 
$\rank \partial^{\ab}_{2}(\rho) \le m-k-1$, or, equivalently, all codimension 
$k$ minors of $\partial^{\ab}_{2}$ vanish when evaluated at $\rho$. 
But we know from \eqref{eq:alexmod} that
$A(G)$ is the cokernel of $\partial_2^{\ab}$, and so the claim is proved 
for $\rho\ne 1$. Finally, the evaluation of the chain 
complex \eqref{eq:abcover-cc} at $\rho=1$ is simply $C_{\hdot}(X;\Z)$, and 
the last claim follows.
\end{proof}

A similar result holds for the restricted characteristic varieties, 
with the Alexander module replaced by its rational counterpart, 
$A_{\rat}(G)=\Z[G_{\abf}]\otimes_{\Z[G]} I(G)$.

\begin{lemma}
\label{lem:hironaka-q}
Let $G$ be a finitely generated group. Then, for all $k\ge 1$, 
\[
\WW_{k}(G)=V(\Fitt_{k+1}(A_{\rat}(G)\otimes \C)) \, ,
\]
at least away from $1\in \T^0_G$, with equality at $1$ for $k<b_1(G)$. 
\end{lemma}

\begin{proof} 
Recall from \S\ref{subsec:alexinv-top-q} that 
$\partial^{\abf}=\partial^{\ab} \otimes_{\Z[G_{\ab}]} \Z[G_{\abf}]$.
Therefore, for a character $\rho\colon G_{\abf} \to \C^*$, we have that 
$\partial^{\abf}(\rho)=\partial^{\ab}(\rho)$. Hence, $\rho$ belongs to 
$\WW_{k}(G)$ precisely when 
$\rank \partial^{\abf}_{2}(\rho) + \rank \partial^{\abf}_{1}(\rho) \le m-k$. 

On the other hand, we know from Lemma \ref{lem:alexmod-q} 
that $A_{\rat}(G)\otimes \Q=\coker (\partial_2^{\abf} \otimes \Q)$. 
Proceeding as in the proof Lemma \ref{lem:hironaka} yields the 
desired conclusions. 
\end{proof}

\section{Alexander varieties}
\label{sect:alex-vars}

We now relate the characteristic varieties of a finitely generated group $G$ 
with the support loci of the exterior powers of the Alexander invariant of $G$. 

\subsection{Support loci of Alexander invariants}
\label{subsec:supp-vars} 
We start with some basic notions from commutative algebra.
Let $R$ be a commutative ring, and let $M$ be an $R$-module. 
The  {\em support}\/ of $M$, denoted $\supp(M)$, consists of 
those maximal ideals $\m\in \Spec(R)$ for which the 
localization $M_{\m}$ is non-zero. Supports are additive, 
in the following sense: If $0\to M\to N\to P\to 0$ 
is an exact sequence of $R$-modules, then 
$\supp(N)=\supp(M)\cup \supp(P)$. 
We denote by $\ann_R(M)$ the annihilator ideal of $M$. 
The following lemma is well-known; see e.g. \cite{Eisenbud}. 

\begin{lemma}
\label{lem:ann-supp}
Suppose $M$ is a finitely generated $R$-module. Then
\begin{enumerate}
\item \label{ann1}
$\supp(M)=V(\ann_R(M))$.
\item \label{ann2}
$V\big(\!\ann_R \big(\bwedge^k M\big)\big)=V(\Fitt_k(M))$.
\end{enumerate}
\end{lemma}

In order to analyze in more depth the characteristic varieties 
$\VV_k(G)$, it is useful to consider the complexified Alexander invariant 
$B(G)\otimes \C$, viewed as a module over the ring $\C[G_{\ab}]$, 
and its exterior powers, $\bwedge^k B(G)\otimes \C$. The 
support loci of these modules, 
\begin{equation}
\label{eq:alex-vars}
\YY_k(G)= \supp \big( \bwedge^k B(G)\otimes \C \big) ,
\end{equation}
are called the {\em Alexander varieties}\/ of $G$. By construction, 
these sets form a descending filtration by Zariski closed subsets 
of the character group, $\T_{G}=\Spec(\C[G_{\ab}])$. 

Likewise, the study of the restricted characteristic varieties $\WW_k(G)$ 
is related to the $\C[G_{\abf}]$-module $B_{\rat}(G)\otimes \C$ 
and its exterior powers. The support loci of these modules, 
\begin{equation}
\label{eq:alex-vars-bis}
\ZZ_k(G)= \supp \big( \bwedge^k B_{\rat}(G)\otimes \C \big) ,
\end{equation}
are subvarieties of the character torus, $\T^0_{G}=\Spec(\C[G_{\abf}])$. 
The next lemma will be useful in analyzing these support loci.

\begin{lemma}
\label{lem:supp-map}
Let $M$ and $N$ be modules over commutative rings $R$ and $S$, respectively, 
and let $\psi\colon M\surj N$ be a surjective morphism covering 
a surjective ring map, $\varphi\colon R\surj S$.  Then the induced morphism 
on maximal spectra, $\varphi^*\colon \Spec(S) \inj \Spec(R)$, restricts to 
embeddings $\supp(\bwedge^k N) \inj \supp(\bwedge^k M)$ for all $k\ge 1$.
\end{lemma}

\begin{proof}
By \eqref{eq:cover-factor}, the map $\psi$ factors as a composite, 
$M \to N_{\varphi}\to N$.  Taking exterior powers of the first map, 
we obtain epimorphisms 
$\bigwedge^k_R M\surj \bigwedge^k_R N_{\varphi}$. 
Since the ring map $\varphi \colon R\surj S$ is surjective, the module
$\bigwedge^k_R N_{\varphi}$ is obtained  from $\bigwedge^k_S N$ 
by restriction of scalars.
Therefore, the map $\varphi$ restricts to a surjection,  
$\ann_R \!\big(\bigwedge^k_R M \big) \surj 
\ann_S \!\big(\bigwedge^k_S N \big)$, and the 
claim follows.
\end{proof}

By Proposition \ref{prop:bq-tf}, we have an epimorphism 
$\kappa\otimes \C\colon B(G)\otimes \C \surj B_{\rat}(G)\otimes \C$ 
which covers the ring map $\tilde\nu\colon \C[G_{\ab}] \surj \C[G_{\abf}]$. 
Applying the previous lemma, we obtain the following corollary. 

\begin{corollary}
\label{cor:supp-restrict}
For a finitely generated group $G$, the inclusion $\T_G^0\inj \T_G^{}$ 
restricts to inclusions $\ZZ_k(G)\inj \YY_k(G)$, for all $k\ge 1$. 
\end{corollary}

If $G_{\ab}$ is torsion-free, the map $\kappa\otimes \C$ is an isomorphism, 
and so $\ZZ_k(G)=\YY_k(G)$.  As illustrated by the next example, this equality 
may not hold when $\Tors(G_{\ab})\ne 0$.  

\begin{example}
\label{ex:b1=0-bis}
Consider again the group $G=\Z_2*\Z_2$, with character group 
$\TT_G=\{(\pm 1,\pm 1)\}$. Then 
$B(G)=\Z[x_1^{\pm 1},x_2^{\pm 1}]/(1+x_1,1+x_2)$, and so 
$\YY_1(G)=\{(-1,-1)\}$, whereas $B_{\rat}(G)=0$, and so 
$\ZZ_1(G)=\emptyset$.
\end{example}

\subsection{Exterior powers in exact sequences}
\label{subsec:commalg}
The next lemma is the key algebraic ingredient in our analysis of the higher-depth 
Alexander varieties.  The lemma is well-known in the case when $R=\k$ is a field and 
the modules are finite-dimensional $\k$-vector spaces. Nevertheless, we could not 
find a reference in the generality that we need here; thus, we provide a detailed proof.

\begin{lemma}
\label{lem:ext-mod}
Let $0\to M\xrightarrow{\alpha} N\xrightarrow{\beta}  P\to 0$ 
be an exact sequence of modules over a commutative ring $R$. 
For each $k\ge 1$, the exterior power $\bwedge^{k}  N$ 
admits a decreasing filtration by $R$-submodules, 
\begin{equation}
\label{eq:filt1}
\bwedge^{k} N = F^k_0 \supseteq F^k_1 
\supseteq \cdots \supseteq F^k_{k+1}=0\, ,
\end{equation}
such that 
\begin{equation}
\label{eq:filt2}
F_i^k/F_{i+1}^k \cong \bwedge^i  M \otimes  
\bwedge^{k-i} P 
\end{equation}
for $0\le i\le k$.
\end{lemma}

\begin{proof}
For $k\ge 1$ and $0\le i\le k$, define an $R$-linear map 
$\alpha^{k}_{i}\colon \bwedge^i M\otimes \bwedge^{k-i} N \to \bwedge^{k} N$ 
by $\alpha^{k}_{i} (u \otimes v)=\bbwedge^i\alpha(u) \wedge v$, and set 
$F^k_i \coloneqq \im(\alpha^{k}_{i})$. Clearly, this defines a filtration on the  
module $\bwedge^{k} N$ such that \eqref{eq:filt1} holds. To show 
that \eqref{eq:filt2} does also hold, we use \cite[Proposition A.2.2(d)]{Eisenbud}, 
from which we extract the following statement: for each $j\ge 1$, there is 
an exact sequence
\begin{equation}
\label{eq:eisenbud}
\begin{tikzcd}[column sep=16pt]
M\otimes \bwedge^{j-1} N  \ar[rr, "\alpha^{j-1}_{1}"]
&& \bwedge^j N  \ar[rr, "\bbwedge^j \beta"]  &&  
 \bwedge^j  P  \ar[r] & 0\, .
\end{tikzcd}
\end{equation}
(For $j=1$, this is the original exact sequence.) Now fix $i\ge 0$ and 
set $j=k-i$; tensoring the sequence \eqref{eq:eisenbud} with $\bwedge^i M$, 
we obtain the exact sequence at the top of the following diagram. 
{\small{
\begin{equation}
\label{eq:eis-cd}
\begin{tikzcd}[column sep=14.5pt]
&[-35pt]\bwedge^i M \otimes M\otimes \bwedge^{k-i-1} N  
\ar[d, twoheadrightarrow, pos=.4, "\tilde\alpha^k_{i+1}"]
\ar[r, "\id \otimes\alpha^{k-i-1}_{1}"]
&[28pt] \bwedge^i M \otimes  \bwedge^{k-i} N  
\ar[d, twoheadrightarrow, pos=.4, "\alpha^k_i"]
\ar[r, "\!\id\otimes \bbwedge^{k-i} \beta"]  
&[30pt] \bwedge^i M \otimes \bwedge^{k-i}  P \ar[equal]{d} \ar[r] & 0
\phantom{\, .}
\\
0\ar[r]&[-38pt] F^k_{i+1}  \ar[r]
&F^k_{i} \ar[r, "\rho^k_i"]  
& \bwedge^i M \otimes \bwedge^{k-i}  P  \ar[r] & 0 \, .
\end{tikzcd}
\end{equation}
}}

In this diagram, the map $\id\otimes \bbwedge^{k-i} \beta$ factors through 
the map $\rho^k_i$ which sends 
$\bbwedge^i\alpha(u) \wedge v$ to $u \otimes \bbwedge^{k-i}  \beta (v)$, 
while $\tilde\alpha^k_{i+1}=\alpha^k_{i+1}\circ \pi_i$, where  
$\pi_i\colon \bwedge^i M \otimes M\surj  \bwedge^{i+1} M$ 
is the canonical projection. It is readily seen that diagram \eqref{eq:eis-cd}
commutes, and therefore $\rho^k_i$ induces an isomorphism 
$F_i^k/F_{i+1}^k \isom  \bwedge^i M \otimes  \bwedge^{k-i} P$. 
This completes the proof.
\end{proof}

\subsection{Characteristic varieties and support loci}
\label{subsec:cv-alex}
It has been known for a long time that the characteristic varieties and 
the Alexander varieties of spaces and groups are intimately related. 
For instance, it was shown in \cite{DSY17, PS-plms10} that 
$\V_1(G)=\YY_1(G)$ and $\WW_1(G)=\ZZ_1(G)$, at least away 
from $1$. Those proofs, based on a change-of-rings spectral 
sequence argument, are very specific to depth $k=1$ and do not 
generalize to higher depths. We give here a proof valid in all depths $k\ge 1$.  
The proof is modeled on the proof of \cite[Proposition 0.2]{Li};  
since the argument given there is not complete 
(a proof of Lemma \ref{lem:ext-mod} 
is missing) and not quite in the generality we need (it assumes $G$ 
is finitely presented), we provide full details. 

\begin{theorem}
\label{thm:cvb}
Let $G$ be a finitely generated group. Then, for all $k\ge 1$, 
\begin{equation}
\label{eq:cvar-ann}
\V_k(G)=
\supp \big(\bwedge^k B(G)\otimes \C \big) ,
\end{equation}
at least away from the identity $1\in \T_{G}$.
\end{theorem}

\begin{proof}
Let $B=B(G)$ and $A=A(G)$, viewed as modules over $R=\Z[G_{\ab}]$, 
and let $I=I_{\Z}(G_{\ab})$ be the 
augmentation ideal. By \eqref{eq:crowell}, we have an exact 
sequence of $R$-modules, $0\to B\to A\to I\to 0$. Fix a maximal 
ideal $\m\in \Spec(R)$. Localization is an exact functor; hence, 
localizing at $\m$ yields an exact sequence of $R_{\m}$-modules, 
\begin{equation}
\label{eq:crowell-loc}
\begin{tikzcd}[column sep=16pt]
0\ar[r]& B_{\m} \ar[r]&  A_{\m} \ar[r]&  I_{\m}\ar[r]&  0 \, .
\end{tikzcd}
\end{equation}
By Lemma \ref{lem:ext-mod}, there is a 
filtration by $R_{\m}$-submodules, 
$\bwedge^{k}  A_{\m}=F^k_0\supseteq F^k_1  \supseteq \cdots$, 
with successive quotients  
$F_i^k/F_{i+1}^k \cong \bwedge^i  B_{\m} \otimes_{R_{\m}}  
\bwedge^{k-i} I_{\m}$. 

On the other hand, localizing at $\m$ the exact sequence 
$0\to I \to R\xrightarrow{\varepsilon} \Z \to 0$, we get the 
exact sequence $0\to I_{\m} \to R_{\m} \to \Z_{\m} \to 0$. 
Assuming $\m\ne I$, we have that $\Z_{\m}=0$,
and so $I_{\m}=R_{\m}$. Hence, $ \bwedge^{j} I_{\m}$ 
is isomorphic to $R_m$ if $j=0$ or $1$ and is equal to $0$ if $j>1$. 

Putting things together, we infer the following: for every $k\ge 1$ 
and for every maximal ideal $\m\ne I$, we have an exact sequence 
\begin{equation}
\label{eq:local}
\begin{tikzcd}[column sep=18pt]
0\ar[r]& \bwedge^{k} B_{\m} \ar[r]& \bwedge^{k}  A_{\m} 
\ar[r]& \bwedge^{k-1} B_{\m} \ar[r]& 0 \, .
\end{tikzcd}
\end{equation}

By additivity of supports, it follows that 
$\supp \!\big(\bwedge^k A\otimes \C \big)=
\supp \!\big(\bwedge^{k-1} B\otimes \C \big)$, at least away from 
$\supp(I)=\{1\}$. On the other hand, by Lemma \ref{lem:ann-supp}, 
we have that $\supp \big(\bwedge^{k} A\otimes \C \big)=
V(\Fitt_{k}(A\otimes \C))$. Applying now Lemma \ref{lem:hironaka} 
completes the proof.
\end{proof}

The next corollary sharpens a result that goes back to the work 
of Dwyer and Fried \cite{DF}, and was further developed in 
\cite{PS-plms10,Su-imrn,SYZ-pisa}. In all those results, only the 
depth $k=1$ was considered; the novelty here is that we work 
with arbitrary depth. 

\begin{corollary}
\label{cor:trivial-cv}
Let $G$ be a finitely generated group. For each $k\ge 1$, the following 
conditions are equivalent.
\begin{enumerate}
\item The characteristic variety $\VV_k(G)$ is a finite subset of $\T_G$. 
\item The $\C$-vector space $\bwedge^k (B(G)\otimes \C)$ is finite-dimensional.
\end{enumerate}
\end{corollary}

\begin{proof}
As is well-known (see e.g.~\cite{DF,PS-mrl, SYZ-pisa}), 
a finitely generated module $M$ over an affine $\C$-algebra 
$R$ has finite support if and only if $\dim_{\C} M<\infty$. 
Letting $M=\bwedge^k B(G)\otimes \C$, viewed as a 
module over the $\C$-algebra $R=\C[G_{\ab}]$, the claim 
follows from Theorem \ref{thm:cvb}.
\end{proof}

\begin{example}
\label{ex:knots}
Let $K$ be a tame knot in $S^3$, and let $G$ be the fundamental 
group of the knot complement. Since $G_{\ab}=\Z$, we may identify 
$\T_G=\C^*$. The variety $\VV_1(G)$ consists 
of $1$, together with the roots of the Alexander polynomial 
of the knot, $\Delta_{K} \in \Z[t^{\pm 1}]$; in particular, $\VV_1(G)$ 
is  finite.   By the above corollary, the $\C$-vector space 
$B(G)\otimes \C$ is finite-dimensional; in fact,
as is well-known, its dimension is equal to $\deg \Delta_K$.
\end{example}

\subsection{Restricting to the character torus}
\label{subsec:cv-alex-0}
We now restrict our attention to the identity component of the 
character group, $\T_G^0$, and prove analogous results 
for the restricted characteristic varieties, $\WW_k(G)=\VV^k(G)\cap \T_G^0$, 
and the corresponding support loci, 
$\ZZ_k(G)= \supp \big( \bwedge^k B_{\rat}(G)\otimes \C \big)$.

\begin{theorem}
\label{thm:cvbq}
Let $G$ be a finitely generated group. Then, for all $k\ge 1$. 
\begin{equation}
\label{eq:cvar-ann-q}
\WW_k(G)=
\supp \big(\bwedge^k B_{\rat}(G)\otimes \C \big) ,
\end{equation}
at least away from the identity $1\in \T^0_{G}$.
\end{theorem}

\begin{proof}
Let $B_{\rat}= B_{\rat}(G)\otimes \C$  and 
$A_{\rat}=A_{\rat}(G)\otimes \C$, and let 
$I_{0}$ be the augmentation ideal of $\C[G_{\abf}]$.
By Lemma \ref{lem:alexmod-q}, part \eqref{amq3}, we have an 
an exact sequence of $\C[G_{\abf}]$-modules,
$0\to B_{\rat}\to A_{\rat}\to I_{0}\to 0$. 
Continuing as in the proof of Theorem \ref{thm:cvb}, we find that 
$\supp \!\big(\bwedge^k A_{\rat} \big)=
\supp \!\big(\bwedge^{k-1} B_{\rat}\big)$, at least away 
from $\{1\}$.  On the other hand, by Lemma \ref{lem:ann-supp}, 
we have that 
$\supp \big(\bwedge^{k} A_{\rat}\big)=V(\Fitt_{k}(A_{\rat}))$. 
Applying now Lemma \ref{lem:hironaka-q} completes the proof.
\end{proof}

It follows at once from Theorems \ref{thm:cvb} and \ref{thm:cvbq} that 
$\ZZ_k(G)=\YY_k(G)\cap \T^0_{G}$.  Moreover, we have the following 
corollary, which, in view of Lemma \ref{lem:bq}, part \eqref{b2}, 
sharpens results from \cite{DF,PS-plms10,Su-imrn,SYZ-pisa}.

\begin{corollary}
\label{cor:trivial-cw}
Let $G$ be a finitely generated group. For each $k\ge 1$, the following 
conditions are equivalent.
\begin{enumerate}
\item The characteristic variety $\WW_k(G)$ is a finite subset of $\T^0_G$. 
\item The $\Q$-vector space $\bwedge^k B_{\rat}(G)\otimes \Q$ is 
finite-dimensional.
\end{enumerate}
\end{corollary}

\begin{proof}
Follows from Theorem \ref{thm:cvbq} using the same argument 
as in the proof of Corollary \ref{cor:trivial-cv}, applied this time 
to the vector space $M=\bwedge^k B_{\rat}(G)\otimes \C$, viewed 
as a module over the $\C$-algebra $R=\C[G_{\abf}]$.
\end{proof}

As an application, we obtain the following corollary.

\begin{corollary}
\label{cor:trivial-cw-chen}
Let $G$ be a finitely generated group, and suppose 
$\WW_1(G)$ is finite. Then the Chen ranks $\theta_n(G)$ 
vanish for $n\gg 0$.
\end{corollary}

\begin{proof}
By Corollary \ref{cor:chen-alrat}, we have that 
$\theta_n(G)=\dim_{\Q} \gr_{n-2}( B_{\rat}(G)\otimes \Q)$ 
for all $n\ge 2$. On the other hand, since  
$\WW_1(G)$ is finite, Corollary \ref{cor:trivial-cw} implies that 
$B_{\rat}(G)\otimes \Q$ is finite-dimensional. Hence, the associated 
graded vector space $\gr(B_{\rat}(G)\otimes \Q)$ is also finite-dimensional;  
thus, its graded pieces must vanish in sufficiently high degrees.
\end{proof}

\section{Characteristic varieties in group extensions}
\label{sect:cv-ext}

We exhibit in this section several relationships 
between the characteristic varieties of a finitely generated group $G$, 
of a normal subgroup $K$, and of the quotient group $Q=G/K$.

\subsection{Homomorphisms and jump loci}
\label{subsec:func-cv}
Let $\alpha\colon G\to H$ be a homomorphism between two 
finitely generated groups. Letting $\alpha_{\ab}\colon G_{\ab}\to H_{\ab}$ 
be the induced map on abelianizations,  its linear extension to 
group algebras, $\tilde\alpha\colon \C[G_{\ab}] \to \C[H_{\ab}]$, 
defines a morphism between the corresponding maximal spectra. Under our 
previous identifications, this morphism coincides with the induced map on 
character groups, $\alpha^*\colon \T_H\to \T_G$, 
given by $\alpha^*(\rho)(g)=\alpha(\rho(g))$.  
In general, this morphism may not send $\V_k(H)$ to $\V_k(G)$, 
even when $\alpha$ is injective. Here is a simple example.

\begin{example}
\label{ex:z-f2}
Let $F_n$ be the free group of rank $n\ge 2$, and let $\Z<F_n$ be a cyclic 
subgroup.  The inclusion $\iota\colon \Z\to F_n$ induces a surjective morphism 
on character tori, $\iota^*\colon (\C^*)^n \surj \C^*$. 
This morphism sends $\V_1(F_n)=\cdots = \V_{n-1}(F_n)=(\C^*)^n$ onto $\C^*$ 
and $ \V_{n}(F_n)=\{1\}$ to $\{1\}$;  
on the other hand, $\V_1(\Z)=\{1\}$ while  $\V_k(\Z)=\emptyset$ for $k>1$, 
so the map $\iota^*$ does not preserve characteristic varieties, for any depth 
$k\le n$.
\end{example}

Under certain assumptions, though, the characteristic varieties are 
preserved.  We present now one such situation, and will return to this issue  
in Theorem \ref{thm:cv-abf}, where a completely different situation 
will be analyzed. 

Suppose $\pi\colon G\to Q$ is surjective; then clearly $\pi^*\colon \T_Q\to \T_G$ 
is injective.  Moreover, $\pi^*$ sends the characteristic varieties of $Q$ 
to those of $G$. We proved this assertion in \cite[Lemma 2.13]{Su-imrn}, 
starting from the jump loci definition \eqref{eq:cvar} of the characteristic 
varieties, and using a spectral sequence argument. We give here another, 
self-contained proof of this result, based on the support loci interpretation 
from the previous section. 

\begin{proposition}[\cite{Su-imrn}]  
\label{prop:v1-nat}
Let  $G$ be a finitely generated group, and let $\pi\colon G\surj Q$ 
be a surjective homomorphism.  Then the induced morphism between 
character groups, $\pi^*\colon \T_Q\inj\T_G$, 
restricts to embeddings $\VV_k(Q) \inj \VV_k(G)$ for all $k\ge 1$.
\end{proposition}

\begin{proof}
The homomorphism $\pi_{\ab}\colon G_{\ab}\surj Q_{\ab}$ 
extends linearly to a ring map, $\tilde\pi_{\ab}\colon R\surj S$, 
between the (commutative) rings $R=\Z[G_{\ab}]$ 
and  $S=\Z[Q_{\ab}]$. Moreover, 
the map $\pi\colon G\surj Q$ induces an epimorphism 
$B(\pi)\colon B(G)\surj B(Q)$ which covers the map $\tilde\pi$. 
The claim now follows from Lemma \ref{lem:supp-map} and 
Theorem \ref{thm:cvb} upon complexifying all these rings and modules 
and taking supports of exterior powers.
\end{proof}
 
\subsection{Jump loci in $\ab$-exact and $\abf$-exact sequences}
\label{subsec:cv-ab-abf} 
Let $K\triangleleft G$ be a normal subgroup, and assume both 
$G$ and $K$ are finitely generated.  The inclusion map 
$\iota\colon K\inj G$ induces a surjective algebraic morphism between 
character groups, $\iota^*\colon \TT_{G} \surj \TT_{K}$, 
which restricts to a surjective map between the 
identity components of those groups, 
$\iota^*\colon \TT^0_{G} \surj \TT^0_{K}$.  The next 
theorem shows that, under appropriate triviality assumptions 
on the monodromy of the resulting extension, these maps preserve 
the respective characteristic varieties.

\begin{theorem}
\label{thm:cv-abf}
Let $\begin{tikzcd}[column sep=14pt]
\!\!1\ar[r] & K\ar[r, "\iota"]
& G \ar[r] & Q\ar[r] & 1\!\!
\end{tikzcd}$ 
be an exact sequence of finitely generated groups. 

\begin{enumerate}
\item \label{cv1}
If the sequence is $\ab$-exact and $Q$ is abelian, 
then the map $\iota^*\colon \TT_{G} \to\TT_{K}$ 
restricts to maps $\iota^* \colon \VV_k(G)\to \VV_k(K)$ 
for all $k\ge 1$; furthermore, the map 
$\iota^* \colon \VV_1(G)\to \VV_1(K)$ is a surjection.

\item \label{cv2}
If the sequence is $\abf$-exact and $Q$ is torsion-free abelian, 
then the map $\iota^*\colon \TT^0_{G} \surj \TT^0_{K}$ 
restricts to maps $\iota^* \colon \WW_k(G)\to \WW_k(K)$ 
for all $k\ge 1$; furthermore, the map 
$\iota^* \colon \WW_1(G)\to \WW_1(K)$ is a surjection.
\end{enumerate}
\end{theorem}

\begin{proof}
Let $R=\Z[K_{\ab}]$ and $S=\Z[G_{\ab}]$. The induced map 
between Alexander invariants, $B(\iota)\colon B(K)\to B(G)$, 
covers the ring map $\tilde\iota_{\ab} \colon R\to S$ obtained from 
$\iota_{\ab}\colon K_{\ab}\to G_{\ab}$ by linear extension to group rings.
Furthermore, $B(\iota)$ may be viewed as the 
composite $B(K) \to B(G)_{\iota} \to B(G)$, where the first arrow is a map 
of $R$-modules and the second arrow is the identity map of $B(G)$, thought of 
as a map covering $\tilde\iota_{\ab}$.  Taking exterior $k$-powers, we may realize 
the map $\bwedge^k B(\iota)$ as the composite 
\begin{equation}
\label{eq:bwedge-alex}
\begin{tikzcd}[column sep=20pt]
\bwedge^k_R\, B(K)\ar[r]& \bwedge^k_R\, B(G)_{\iota}
\ar[r]& \bwedge^k_S\, B(G) .
\end{tikzcd}
\end{equation}

To prove claim \eqref{cv1}, first note that the map $\iota_{\ab}$ is injective, 
since the given sequence is $\ab$-exact. Thus, the map $\tilde\iota_{\ab}$ 
is also injective, and the induced map between character 
groups, $\iota^*\colon \TT_{G} \to\TT_{K}$, is surjective.
Furthermore, since $Q$ is abelian, 
Theorem \ref{thm:alex-abex}, part \eqref{ng2} shows that 
the map $B(K) \to B(G)_{\iota}$ is an $R$-isomorphism. 
Therefore, the map $\tilde\iota_{\ab}$ restricts to a map   
$\ann_R\! \big(\bigwedge^k_R B(K) \big) \to 
\ann_S\! \big(\bigwedge^k_S B(G) \big)$. 
Applying Theorem \ref{thm:cvb}, we obtain a map $\VV_k(G)\to \VV_k(K)$
after tensoring with $\C$ and taking zero-sets. By a previous remark, this 
map coincides with the restriction of $\iota^*$ to $\VV_k(G)$.

When $k=1$, the map $B(G)_{\iota} \to B(G)$ is injective; 
the argument above then shows that the map 
$\iota^* \colon \VV_1(G)\to \VV_1(K)$ is surjective, and  
the proof of the first claim is complete.

To prove claim \eqref{cv2}, first note that the map $\iota_{\abf}$ 
is injective, since by assumption our sequence is $\abf$-exact.  
Thus, the induced morphism, $\iota^*\colon \TT^0_{G} \to\TT^0_{K}$, 
is surjective.  Moreover, Theorem \ref{thm:cvbq} insures 
that $\WW_k(G)$ coincides, at least away from $1$, with the 
support of the $\C[G_{\abf}]$-module $\bwedge^k \big(B_{\rat}(G)\otimes \C\big)$. 
Since $Q$ is torsion-free abelian, 
Theorem \ref{thm:alex-lcs-ngq}, part \eqref{ngq2} yields a 
$\Z[K_{\abf}]$-isomorphism, $B_{\rat} (K) \to B_{\rat} (G)_{\iota}$. 
The second claim now follows as above.
\end{proof}

It is clear that, in the above theorem, we need to make some triviality 
assumptions on the action of $Q$ in first homology, for otherwise we 
may well have $b_1(K)>b_1(G)$.  When this happens, the map 
$\iota^*\colon \TT_{G} \to\TT_{K}$ is not surjective, making it 
unlikely that it would restrict to a surjection from $\VV_1(G)$ 
to $\VV_1(K)$. We illustrate this point with a simple example, 
and will expand on this issue in \S\ref{subsec:cv-split-discuss}.

\begin{example}
\label{ex:cv-free}
Let $G=F_n$ be a free group of rank $n>1$; as observed previously, 
$\V_1(G)=\T_G$ in this case.  Now let $\pi\colon F_n\surj \Z_m$ be an 
epimorphism $(m>1)$; then $K=\ker(\pi)$ is isomorphic to $F_{nm-m+1}$. 
Thus, the inclusion $\iota^*\colon \TT_{G} \to\TT_{K}$ is not surjective, 
and neither is its restriction to the characteristic varieties, 
$\iota^*\colon \V_1(G) \to \V_1(K)$.
\end{example}

\begin{corollary}
\label{cor:cv-ab-df}
Let $1\to K\to G\to Q\to 1$ be an $\ab$-exact sequence 
of finitely generated groups. Suppose $Q$ is abelian and 
$\V_1(G)$ is finite. Then $\dim_{\Q} B(K)\otimes \Q<\infty$ 
and $\theta_n(K)=0$ for $n\gg 0$.
\end{corollary}

\begin{proof}
By Theorem \ref{thm:cv-abf}, part \eqref{cv1}, 
the inclusion $\iota\colon K\to G$ induces a surjective morphism, 
$\iota^* \colon \V_1(G) \surj \V_1(K)$.  Since $\V_1(G)$ is a finite 
set, the same must be true for $\V_1(K)$, and the first claim follows 
from Corollary \ref{cor:trivial-cv}.  Now note that 
$\WW_1(K)=\V_1(K)\cap \T^0_K$ is also finite, and so the second 
claim follows from Corollary \ref{cor:trivial-cw-chen}; alternatively, 
the second claim follows from the first one and Corollary \ref{cor:chen-alrat}.
\end{proof}

\begin{corollary}
\label{cor:cv-abf-df}
Let $1\to K\to G\to Q\to 1$ be an $\abf$-exact sequence of finitely generated 
groups. Suppose $Q$ is torsion-free abelian and $\WW_1(G)$ is finite. Then 
$\dim_{\Q} B_{\rat}(K)\otimes \Q<\infty$ and $\theta_n(K)=0$ for $n\gg 0$.
\end{corollary}

\begin{proof}
This is proved in like fashion: the first claim follows from 
Theorem \ref{thm:cv-abf}, part \eqref{cv2} and Corollary \ref{cor:trivial-cw}, 
while the second claim follows again from Corollary \ref{cor:trivial-cw-chen} 
(or from the first one and Corollary \ref{cor:chen-alrat}).
\end{proof}

\subsection{Jump loci in split-exact sequences}
\label{subsec:cv-split} 
For split extensions, Theorem \ref{thm:cv-abf} admits a slightly more 
convenient formulation. 

\begin{corollary}
\label{cor:cv-abf-semi}
Let $\begin{tikzcd}[column sep=14pt]
\!\!1\ar[r] & K\ar[r, "\iota"]
& G \ar[r] & Q\ar[r] & 1\!\!
\end{tikzcd}$ 
be a split exact sequence of finitely generated groups. 

\begin{enumerate}
\item \label{semi-cv1}
If $Q$ is abelian and acts trivially on $H_1(K;\Z)$, 
then the map $\iota^*\colon \TT_{G} \to\TT_{K}$ 
restricts to maps $\iota^* \colon \VV_k(G)\to \VV_k(K)$ 
for all $k\ge 1$; furthermore, the map 
$\iota^* \colon \VV_1(G)\to \VV_1(K)$ is a surjection.

\item \label{semi-cv2}
If $Q$ is torsion-free abelian and acts trivially on $H_1(K;\Q)$, 
then the map $\iota^*\colon \TT^0_{G} \surj \TT^0_{K}$ 
restricts to maps $\iota^* \colon \WW_k(G)\to \WW_k(K)$ 
for all $k\ge 1$; furthermore, the map 
$\iota^* \colon \WW_1(G)\surj \WW_1(K)$ is a surjection.
\end{enumerate}
\end{corollary}

\begin{proof}
Both claims follow from Theorem \ref{thm:cv-abf}, together with 
Proposition \ref{prop:ab-exact-split} for the first one, and 
Proposition \ref{prop:abf-exact-split}, part \eqref{ssq2} for the 
second one.
\end{proof}

Here is a large class of examples where Corollary \ref{cor:cv-abf-semi} 
applies. 

\begin{example}
\label{ex:raag-cv}
Let $\Gamma$ be a connected, finite graph, and let 
$1\to N_{\Gamma} \xrightarrow{\iota} G_{\Gamma} \to \Z\to 1$ 
be the $\ab$-exact exact sequence from \eqref{eq:raag-bb}. 
In \cite[Lemma 8.3(i)]{PS-jlms07}, it was shown that the 
map $\iota^*\colon \TT_{G_{\Gamma}} \surj\TT_{N_{\Gamma}}$ 
restricts to a surjection 
$\iota^* \colon \VV_1(G_{\Gamma})\surj \VV_1(N_{\Gamma})$, 
provided $\pi_1(\Delta_{\Gamma})=0$. 
Corollary \ref{cor:cv-abf-semi} recovers this result, 
without this additional assumption.
\end{example}

\subsection{Discussion}
\label{subsec:cv-split-discuss} 
For the rest of this section, we discuss the necessity of the assumptions 
we made in Corollary \ref{cor:cv-abf-semi}, and thus, implicitly, in 
Theorem \ref{thm:cv-abf}, too. The first example shows why it is necessary 
to assume that the action of $Q$ on $H_1(K;\Z)$ ought to be trivial, 
even when $Q=\Z$ and $b_1(K)=b_1(G)$. 

\begin{example}
\label{ex:klein}
Let $G=\langle t,a\mid tat^{-1}=a^{-1}\rangle$ be 
the fundamental group of the Klein bottle. 
Then $G=\Z\rtimes_{\varphi} \Z$, where the monodromy automorphism $\varphi$
acts by inversion on the subgroup $K=\Z=\langle a\rangle$, and 
$G_{\ab}=\Z\oplus \Z_2$. The inclusion $\iota\colon K\to G$ induces a 
surjection $\iota^*\colon \T_G\to \T_K$, where $\T_G=\C^*\times \{\pm 1\}$ 
and $\T_K=\C^*$. This map sends $\V_1(G)=\{(1,1),(-1,1)\}$ onto $\{\pm 1\}\subset \C^*$, 
but $\{\pm 1\}$ is not contained in $\V_1(K)=\{1\}$. 
\end{example}

The next example shows that the condition requiring $Q$ to be  
abelian is crucial for Corollary \ref{cor:cv-abf-semi} to hold, even 
when $Q$ acts trivially on abelianization.

\begin{example}
\label{ex:braids}
Let $P_n$ be the Artin pure braid group on $n\ge 4$ strands. We then have 
a split exact sequence, $1\to F_{n-1} \xrightarrow{\iota} P_n \to P_{n-1}\to 1$, 
with monodromy given by the Artin embedding, $P_{n-1}\inj \Aut (F_{n-1})$. 
Since pure braids act trivially on $H_1(F_{n-1};\Z)$, the sequence is $\ab$-exact; 
moreover, both $P_n$ and $F_{n-1}$ are finitely generated, though of course 
$P_{n-1}$ is {\em not}\/ abelian.  It is known that $\dim(\V_1(P_n))=2$ 
(see e.g.~\cite{Su-toul} and references therein), whereas 
$\dim(\V_1(F_{n-1}))=n-1$. Thus, the morphism  
$\iota^*\colon \T_{P_n} \to \T_{F_{n-1}}$ does {\em not}\/ restrict to 
a surjection $\V_1(P_n)\to \V_1(F_{n-1})$. 
\end{example}

\begin{remark}
\label{rem:higher-depth}
Given an $\ab$-exact sequence as in Corollary \ref{cor:cv-abf-semi}, 
part \eqref{semi-cv1}, the morphisms $\iota^*\colon \VV_k(G)\to \VV_k(K)$ 
may fail to be surjective for $k>1$. The reason is that exterior powers do 
not necessarily commute with restriction of scalars, and so the map 
$\bwedge^k_R\, B(K)\to \bwedge^k_S\, B(G)_{\iota}$ 
may fail to be injective.  
\end{remark}

Computations done in \cite[Example 6.5]{Su-revroum} may be used 
to show that the phenomenon mentioned in the above remark occurs 
in the context of Milnor fibrations of complex hyperplane arrangements. 
This topic will be discussed in more detail in \cite{Su-mfmono}. 
We give here a different kind of example; we will return to it in 
\S\ref{subsec:res-discuss} in order to illustrate a related point 
regarding the higher-depth resonance varieties.

\begin{example}
\label{ex:high-depth-cv}
Let $\A=\A(2134)$ be the arrangement of transverse planes through 
the origin of $\R^4$ defined in complex coordinates by the function 
$f(z,w)=zw(z-w)(z-2\bar{w})$.  
As noted in Example \ref{ex:planes}, the arrangement group, $G$, 
fits into an $\ab$-exact sequence, $1\to F_{3} \xrightarrow{\iota} G \to \Z\to 1$, 
and so all the hypothesis of Corollary \ref{cor:cv-abf-semi}, 
part \eqref{semi-cv1} are satisfied.  As shown in \cite{MS-top,MS-imrn}, 
the variety $\V_1(G)$ consists of two codimension $1$ subtori in $(\C^*)^4$, 
which are sent by $\iota^*$ 
onto $\V_1(F_3)=(\C^*)^3$, as predicted. On the other hand, $\V_2(G)$ 
consists of two $1$-dimensional subtori and a $2$-dimensional translated 
subtorus; thus, its image under $\iota^*$ is strictly contained in 
$\V_2(F_3)=(\C^*)^3$.
\end{example}

\part{Holonomy and resonance}
\label{part:resonance}

\section{Holonomy Lie algebras and formality properties}
\label{sect:formal}

In this section we review some basic notions regarding the holonomy 
Lie algebra and the Malcev Lie algebra of a finitely generated group $G$, 
and discuss the graded formality and $1$-formality properties of such groups.

\subsection{The holonomy Lie algebra of a group}
\label{subsec:holonomy}
Let $G$ be a group such that the maximal torsion-free abelian quotient 
$G_{\abf}$ is finitely generated. 
There is then another graded Lie algebra that can be associated to it, 
besides the ones already mentioned in \S\ref{sect:grg}. 
This Lie algebra is much easier to understand, in that it uses only 
information about the cohomology ring of $G$ encoded in the 
cup-product map $\cup_G\colon H^1(G)\wedge H^1(G) \to H^2(G)$. 

More precisely, let $\L=\Lie(G_{\abf})$ be the free Lie algebra on $G_{\abf}$. 
This is a (positively) graded Lie algebra, with grading given by bracket length. 
Writing $\L=\bigoplus_{n\ge 1} \L_n$, we have $\L_1=G_{\abf}$ 
and $\L_2=G_{\abf}\wedge G_{\abf}$.  Taking the dual of the 
cup-product map and writing $H^{\vee}\coloneqq \Hom(H,\Z)$, 
we obtain the comultiplication map, 
\begin{equation}
\label{eq:nabla}
\begin{tikzcd}[column sep=20pt]
\cup_G^{\vee}\colon H^2(G)^{\vee} \ar[r]& (H^1(G)\wedge H^1(G))^{\vee}\cong 
G_{\abf}\wedge G_{\abf}\, .
\end{tikzcd}
\end{equation} 

Following \cite{Chen77,Markl-Papadima, PS-imrn04, PS-mathann06, SW-jpaa}, 
we define the {\em holonomy Lie algebra}\/ of $G$, denoted by $\h(G)$,  
as the quotient 
\begin{equation}
\label{eq:holo-lie}
\h(G) = \Lie(G_{\abf})/(\im(\cup_G^{\vee}))
\end{equation}
of the free Lie algebra $\L=\Lie(G_{\abf})$ by the Lie ideal generated by 
the image of $\cup_G^{\vee}$, viewed as a subgroup of $\L_2$.
The holonomy Lie algebra inherits a natural grading from the 
free Lie algebra, which is compatible with the Lie bracket.  
By construction, $\h(G)$ is a quadratic Lie algebra: 
it is generated in degree $1$ by $G_{\abf}$, and all the 
relations are in degree $2$. In particular, the derived Lie 
subalgebra, $\h(G)'$, coincides with $\h_{\ge 2}(G)$. 
As noted in \cite{SW-jpaa}, the projection map 
$G\surj G/\gamma_n(G)$ induces an isomorphism 
$\h(G)\isom \h(G/\gamma_n(G))$ for all $n\ge 3$. 
Consequently, the holonomy Lie algebra of $G$ depends 
only on its second nilpotent quotient, $G/\gamma_3 (G)$. 

This construction is functorial. Indeed, 
let $\alpha\colon G \to H$ be a homomorphism between two 
groups as above; then the induced homomorphism 
$\alpha_{\abf}\colon G_{\abf} \to H_{\abf}$ extends 
to a morphism $\L(\alpha_{\abf})\colon \L(G_{\abf}) \to \L(H_{\abf})$
between the respective free Lie algebras. The map $\alpha$ 
also induces a morphism between cohomology rings, and 
thus sends $\im(\cup_G^{\vee})$ to $\im(\cup_H^{\vee})$. 
Consequently, $\L(\alpha_{\abf})$ induces a morphism of 
graded Lie algebras, $\h(\alpha) \colon \h(G)\to \h(H)$; 
it is easily checked that $\h(\beta\circ\alpha)=\h(\beta)\circ\h(\alpha)$.

A notable fact about the holonomy Lie algebra
is its relationship to the associated graded Lie algebra, 
as embodied in the next theorem. 

\begin{theorem}[\cite{Markl-Papadima,PS-imrn04,SW-jpaa}]
\label{thm:holo-epi}
For every group $G$ such that $G_{\abf}$ is finitely generated,  
there exists a natural epimorphism of graded Lie algebras, 
$\Psi\colon \h(G) \surj \gr(G)$, which induces 
isomorphisms in degrees $1$ and $2$ and descends to epimorphisms  
$\Psi^{(r)}\colon \h(G)/\h(G)^{(r)} \surj \gr(G/G^{(r)})$ for all $r\ge 2$.  
\end{theorem}

This result is stated and proved in the references cited, in various 
degrees of generality.  Essentially the same proof works in the slightly 
more general context adopted here. 

It follows from the above theorem that the Lie algebra $\h(G)/\h(G)''$ 
maps surjectively onto the Chen Lie algebra $\gr(G/G^{''})$. Thus, if 
we define the {\em holonomy Chen ranks}\/ of $G$ as $\bar\theta_n(G)\coloneqq 
\rank \left(\h(G)/\h(G)''\right)_n$, we have that $\bar\theta_n(G)\ge \theta_n(G)$, 
for all $n\ge 1$.

\subsection{Graded formality}
\label{subsec:gr-formal}

Suppose now that $G$ is a group with $b_1(G)<\infty$, 
and let $\k$ be a field of characteristic $0$. 
We may then define in a completely analogous fashion the 
holonomy Lie algebra of $G$ with coefficients in $\k$, 
denoted $\h(G;\k)$, as the free Lie algebra on $H_1(G;\k)$, 
modulo the ideal generated by the image of the comultiplication map, 
$H_2(G;\k) \to H_1(G;\k)\wedge H_1(G;\k)$.  

An analogue of Theorem \ref{thm:holo-epi} holds over $\k$ even when 
$b_1(G)<\infty$:  There exists a natural epimorphism of graded Lie algebras, 
$\Psi_{\k}\colon \h(G;\k) \surj \gr(G)\otimes \k$, which induces 
isomorphisms in degrees $1$ and $2$ and descends to epimorphisms  
$\Psi^{(r)}_{\k}\colon \h(G;\k)/\h(G;\k)^{(r)} \surj \gr(G/G^{(r)})\otimes \k$\, 
for all $r\ge 2$.  Thus, if we set $\bar\theta_n(G)\coloneqq 
\dim_{\k} \left(\h(G;\k)/\h(G;\k)''\right)_n$, 
we have that $\bar\theta_n(G)\ge \theta_n(G)$, 
for all $n\ge 1$.  When $G_{\abf}$ is finitely generated, 
$\h(G;\k)=\h(G)\otimes \k$, and the holonomy Chen ranks 
$\bar\theta_n(G)$ defined here coincide with the ones 
from \S\ref{subsec:holonomy}. 

Following \cite{SW-jpaa, SW-forum}, we say that a finitely generated 
group $G$ is {\em graded formal}\/ if the map 
$\Psi_{\k} \colon \h(G;\k) \surj \gr(G)\otimes \k$ is an isomorphism. 
This condition is equivalent to  $\gr(G)\otimes \k$ being a quadratic Lie algebra. 
As shown in \cite[Theorem 5.11]{SW-forum}, if $K\le G$ is a retract 
of a graded formal group $G$, then $K$ is also graded formal. The next 
theorem gives another condition insuring that graded formality is 
inherited by a subgroup. 

\begin{theorem}
\label{thm:abf-gr-formal}
Let $G=K\rtimes Q$ be a split extension of finitely generated groups, 
and assume $Q$ acts trivially on $H_1(K;\k)$. If $G$ is graded formal, 
then $K$ is also graded formal.
\end{theorem}

\begin{proof}
An analogous result was proved in \cite[Theorem 5.12]{SW-forum}, 
under the more restrictive assumption that $Q$ should act trivially 
on $K_{\ab}$. Starting from the hypothesis that the map 
$\Psi_{\k}\colon \h(G;\k) \surj \gr(G)\otimes \k$ 
is an isomorphism, the proof uses the splitting 
$\gr(G)\to \gr(K)$ provided by Falk and Randell's 
Theorem \ref{thm:falk-ran} to conclude that the map 
$\Psi_{\k}\colon \h(K;\k) \surj \gr(K)\otimes \k$ 
is also an isomorphism.

In our more general setting, Proposition \ref{prop:abf-exact-split} insures 
that $Q$ acts trivially on $K_{\abf}$. Therefore, the same proof works here, 
using instead the splitting $\gr(G)\otimes \k\to \gr(K)\otimes \k$ 
induced by the splitting  $\gr^{\rat}(G)\to \gr^{\rat}(K)$ from 
Theorem \ref{thm:fr-rational}.
\end{proof}

\subsection{Malcev completion}
\label{subsec:malcev}
In \cite{Qu}, Quillen associated to every group $G$ 
a filtered Lie algebra over the rationals, $\m(G)$, called the 
{\em Malcev Lie algebra}\/ of $G$. Let $I=I_{\Q}(G)$ be the 
augmentation ideal of the group algebra of $G$, and let 
\begin{equation}
\label{eq:qg-completion}
\widehat{\Q[G]}=\varprojlim_{n} \Q[G]/I^n
\end{equation}
be the completion of $\Q[G]$ with respect to the $I$-adic filtration. 
The usual Hopf algebra structure on the group algebra extends to 
the completion, making $\widehat{\Q[G]}$ into a complete 
Hopf algebra.  By definition, $\m(G)$ is the Lie algebra of 
primitive elements in $\widehat{\Q[G]}$, endowed with the 
induced filtration and the (compatible) Lie bracket $[x,y]=xy-yx$. 
It is readily seen that this construction is functorial.

Let $\gr(\m(G))$ be the associated graded Lie algebra 
with respect to the lower central series filtration of $\m(G)$. 
Quillen then showed in \cite{Quillen68} that 
\begin{equation}
\label{eq:quillen-iso}
\gr(\m(G))\cong \gr(G)\otimes \Q \, . 
\end{equation}

For more on this topic, we refer 
to \cite{DHP14, DPS-duke, PS-imrn04, PS-formal, SW-forum, SW-ejm}.  
We will only recall here two results which will prove to be useful to us. 
The first result, proved in \cite{Hain} and recorded in \cite{DHP14}, 
is a filtered version of a celebrated result of Stallings (\cite[Theorem~7.3]{St}).

\begin{theorem}[\cite{DHP14, Hain, St}] 
\label{thm:hain-stallings}
Let $\alpha\colon G\to H$ be a homomorphism that induces an isomorphism 
$H_1(G;\Q)\isom H_1(H;\Q)$ and an epimorphism $H_2(G;\Q)\surj H_2(H;\Q)$. 
Then $\alpha$ induces an isomorphism of filtered Lie algebras, 
$\m(G)\isom \m(H)$, and thus, and isomorphism of graded Lie algebras, 
$\gr(G)\otimes \Q\isom \gr(H)\otimes \Q$.
\end{theorem}

Let $I$ be the augmentation ideal of the ring $R=\Q[G_{\ab}]$. 
The second result, proved in \cite[Proposition 5.4]{DPS-duke}, 
relates the $I$-adic completion of the rationalized Alexander 
invariant, $B(G)\otimes \Q$, viewed as a module over $\widehat{R}$, 
to the Malcev Lie algebra of $G$. For a subset 
$\mathfrak{a} \subset \m(G)$, we denote by 
$\overline{\mathfrak{a}}$ its closure in the topology 
defined by the filtration on $\m(G)$.

\begin{theorem}[\cite{DPS-duke}]
\label{thm:alex-malcev}
For a finitely generated group $G$, there is a filtration-preserving 
$\widehat{R}$-linear isomorphism, $\widehat{B(G)\otimes \Q}\cong 
\overline{\m(G)'} /\overline{\m(G)''}$.
\end{theorem}

The following is an immediate corollary of the above two theorems.

\begin{corollary} 
\label{cor:ac-map}
Let $\alpha\colon G\to H$ be a homomorphism between two finitely 
generated groups.  Suppose $\alpha$ induces an isomorphism 
$H_1(G;\Q)\isom H_1(H;\Q)$ and an epimorphism $H_2(G;\Q)\surj H_2(H;\Q)$. 
Then $\alpha$ induces a filtration-preserving isomorphism 
$\widehat{B(G)\otimes \Q}\isom \widehat{B(H)\otimes \Q}$, 
and thus, a degree-preserving isomorphism 
$\gr(B(G)\otimes \Q)\isom \gr(B(H)\otimes \Q)$.
\end{corollary}

In particular, as noted in \cite{DHP14}, if $G$ is finitely generated,  
$K\le G$ has finite index, and the inclusion 
map $\iota\colon K\to G$ induces an isomorphism on $H_1(-;\Q)$, 
then it also induced an isomorphism between completed 
Alexander invariants (over $\Q$). 

\subsection{The $1$-formality property}
\label{subsec:formal}
A finitely generated group $G$ is said to be 
{\em $1$-formal}\/ if $\m(G)$ is isomorphic (as a filtered Lie algebra) 
to $\widehat{\h(G; \Q)}$, the completion of the rational holonomy 
Lie algebra of $G$ with respect to its lower central series filtration.  
 
Now let $\Psi_{\Q}\colon \h(G; \Q) \surj \gr(G)\otimes \Q$  
and $\Psi^{(r)}_{\Q} \colon \h(G;\Q)/\h(G;\Q)^{(r)} 
\surj \gr(G/G^{(r)})\otimes \Q$ be the functorial morphisms of graded 
Lie algebras from \S\ref{subsec:gr-formal}. 
If $G$ is a $1$-formal group, then all these morphisms
are isomorphisms; this basic result follows from \cite{Qu} 
and \cite{PS-imrn04}, respectively (see also \cite{PS-formal,SW-jpaa,SW-forum}). 
Consequently, $G$ is $1$-formal if and only if $G$ is graded formal 
and $\m(G)$ is isomorphic to the completion of its associated graded 
Lie algebra (\cite{SW-forum}). It is still an open question as to 
when the $1$-formality property propagates through group extensions.  
Here is one instance when formality is inherited by (normal) subgroups.

\begin{theorem}[\cite{DP-pisa}]
\label{thm:1-formal n-cover}
Let $K \triangleleft G$ be a normal subgroup.  
Suppose $G$ is $1$-formal, the quotient $Q=G/K$ is finite, 
and $Q$ acts trivially on $H_1(K;\Q)$.  Then $K$ is also $1$-formal. 
\end{theorem}

More generally, one may consider the $1$-formality problem for short 
exact sequences of the form $1\to K \to G \to  Q\to 1$, where both 
$G$ and $K$ are finitely generated groups. If $K$ is $1$-formal, 
the group $G$ does not have to be $1$-formal, as illustrated 
by the Heisenberg group $G=\Z^2\rtimes \Z$ (see Example \ref{ex:heisenberg}). 
Likewise, if $G$ is $1$-formal, the group $K$ need not be $1$-formal;  in fact, 
$K$ is need not be $1$-formal, even if $Q=G/K$ is finite, 
see for instance \cite{DP-pisa,Su-toul}.
On the other hand, we have the following positive result.

\begin{theorem}[\cite{SW-forum}] 
\label{thm:1-formal-retract}
Let $G$ be a $1$-formal group, and let $K\triangleleft G$ be a 
normal subgroup. If $K$ is a retract of $G$, then $K$ is also $1$-formal. 
\end{theorem}

\section{Infinitesimal Alexander invariants and the BGG correspondence}
\label{sect:inf-alex}

We now use the holonomy Lie algebra  
to construct a graded module over a symmetric algebra  
which can be viewed as the infinitesimal version of the 
Alexander invariant. Using the BGG correspondence, 
we also construct an infinitesimal version of the Alexander 
module, and relate the two modules via a Crowell-type exact sequence.

\subsection{Infinitesimal Alexander invariant}
\label{subsec:inf-alex-inv}
Let $G$ be a group and assume that $G_{\abf}$ is finitely generated. 
We will denote by $\Sym(G_{\abf})$ the symmetric algebra on this 
free abelian group of finite rank. Note that $\Sym(G_{\abf})$ is naturally 
isomorphic to $\gr(\Z[G_{\abf}])$. In concrete terms, if we identify 
$G_{\abf}$ with $\Z^r$, where $r=b_1(G)$, then $\Sym(G_{\abf})$ 
gets identified with the polynomial ring $\Z[x_1,\dots,x_r]$. 

A homomorphism $\alpha\colon G \to H$ between two groups 
as above induces a homomorphism $\alpha_{\abf}\colon G_{\abf} \to H_{\abf}$, 
which extends to a ring map, 
$\tilde\alpha_{\abf}\colon \Sym(G_{\abf})\to  \Sym(H_{\abf})$. 
If we indentify these symmetric algebras with the corresponding 
polynomial rings, the map $\tilde\alpha_{\abf}$ is  the 
linear change of variables defined by the matrix of  $\alpha_{\abf}$. 
Consequently, if  $\alpha_{\abf}$ is injective (respectively, surjective), 
then $\tilde\alpha_{\abf}$ is also injective (respectively, surjective).

Now let $\h(G)$ be the holonomy Lie algebra of $G$. 
Following the approach from \cite{PS-imrn04} 
(see also \cite{DP-ann,PS-jtop,PS-crelle,SW-mz,SW-aam}), 
we define the {\em infinitesimal Alexander invariant}\/ of 
$G$ to be the quotient group
\begin{equation}
\label{eq:b-lin}
\B(G)\coloneqq \h(G)'/\h(G)'', 
\end{equation}
viewed as a graded module over $\Sym(G_{\abf})$. The 
module structure comes from the exact sequence 
\begin{equation}
\label{eq:holo-ses}
\begin{tikzcd}[column sep=18pt]
0\ar[r]& \h(G)'/\h(G)'' \ar[r]& \h(G)/\h(G)'' \ar[r]&
\h(G)/\h(G)' \ar[r]& 0
\end{tikzcd}
\end{equation}
via the adjoint action of $\h(G)/\h(G)' =\h_1(G) =G_{\abf}$ 
on $\h(G)'/\h(G)''$ given by $g\cdot \bar{x} = \overline{[g,x]}$ 
for $g\in \h_1(G)$ and $x\in \h(G)'$, and with the grading inherited 
from the one on $\h(G)$. 
When $G$ admits a finite, commutator-relators presentation, 
this module is isomorphic to the ``linearization" of $B(G)$, 
cf.~\cite[Proposition 9.3]{PS-imrn04}. 

The above construction is functorial. More precisely, if $\alpha\colon G \to H$ 
is a homomorphism between two groups as above, then $\alpha$ 
induces a morphism of graded Lie algebras, $\h(\alpha) \colon \h(G)\to \h(H)$, 
which preserves the respective derived series. Hence, the restriction 
$\h'(\alpha) \colon \h(G)'\to \h(H)'$ induces a map $\B(\alpha)\colon 
\B(G)\to \B(H)$. A routine check shows that $\B(\alpha)$ is a 
morphism of modules covering the ring map 
$\tilde\alpha_{\abf}\colon \Sym(G_{\abf})\to  \Sym(H_{\abf})$, and 
that $\B(\beta\circ\alpha)=\B(\beta)\circ\B(\alpha)$.

Denoting by $\B(H)_{\alpha}$ the module obtained from $\B(H)$ by 
restriction of scalars along $\tilde\alpha_{\abf}$, 
we may view the map $\B(\alpha)$ as the composite 
$\B(G) \to \B(H)_{\alpha} \to \B(H)$, 
where the first arrow is a $\Sym(G_{\abf})$-linear map and the second 
arrow is the identity map of $\B(H)$, thought of as covering the ring 
map $\tilde\alpha_{\abf}$. 

Consider now the  {\em canonical element}, 
$\omega_G \in G_{\abf}^{\vee} \otimes G_{\abf}^{\:}$; 
by definition, this tensor corresponds to the identity 
automorphism of $G_{\abf}$ under the tensor-hom adjunction. Identifying 
$G_{\abf}^{\vee}$ with $H^1(G;\Z)$ and $G_{\abf}$ with the degree $1$ piece 
of $S=\Sym(G_{\abf})$, multiplication by $\omega_G$ in the graded ring 
$H^{\hdot}(G;\Z)\otimes S$ restricts to a $S$-linear map, 
$\cdot \omega_G\colon H^1(G;\Z)\otimes S \to H^2(G;\Z)\otimes S$. 
By analogy with \eqref{eq:alexmod}, we define the {\em infinitesimal  
Alexander module}\/ of $G$ to be the cokernel of the $S$-dual of this map,  
\begin{equation}
\label{eq:inf-amod}
\begin{tikzcd}[column sep=20pt]
\AA(G)\coloneqq \coker \big( (\cdot \omega_G)^{*} \colon 
 H^2(G;\Z)^{\vee} \otimes S \ar[r]& G_{\abf} \otimes S \big) \,.
 \end{tikzcd}
\end{equation}

If $\alpha\colon G \to H$ is a group homomorphism, it is readily seen that  
$\big(\alpha_{\abf}^{\vee} \otimes \tilde\alpha_{\abf}^{\:}\big) \circ (\cdot \omega_G)^* 
= (\cdot \omega_H)^* \circ \big(H^2(\alpha)^{\vee}\otimes \tilde\alpha_{\abf}^{\:}\big)$. 
Consequently, $\alpha$ induces a morphism of modules, 
$\AA(\alpha)\colon \AA(G)\to \AA(H)$, which covers the 
ring map $\tilde\alpha_{\abf}$.

\subsection{The BGG correspondence}
\label{subsec:BGG}
All these notions may be extended to the more general context of groups $G$ 
with $b_1(G)<\infty$, provided we work over a field $\k$ of characteristic $0$. 
Specifically, let us define the infinitesimal Alexander invariant with 
coefficients in $\k$ as the quotient 
\begin{equation}
\label{eq:b-lin-k}
\B(G;\k)= \h(G;\k)'/\h(G;\k)''  ,
\end{equation} 
viewed as a graded module over the $\k$-algebra $\Sym(H_1(G;\k))$. 
By a previous remark, when $G_{\abf}$ is finitely 
generated we have that $\h(G;\k)=\h(G)\otimes \k$, and so 
$\B(G;\k)=\B(G)\otimes \k$. In the 
particular case when $G$ itself is finitely generated, the module 
$\B(G;\C)$ coincides with the {\em Koszul module}\/ introduced 
in \cite{PS-crelle}, and further studied in \cite{AFPRW1,  AFPRW2, AFRS}.

We write $H^{\hdot}=H^{\hdot}(G;\k)$ for the cohomology algebra of $G$ 
with coefficients in $\k$, and $H_{i}=H_i(G;\k)$ for the dual 
$\k$-vector spaces.  So let $S=\Sym(H_1)$ be the symmetric 
$\k$-algebra on $H_1$, viewed as the coordinate ring of $H^1$.  
Let $E^{\hdot}=\bwedge H^1$
be the exterior algebra on $H^1$; the identification $E^1=H^1$ 
extends to a map of graded rings, $E \to H$, which turns 
$H$ into an $E$-module.  The Bernstein--Gelfand--Gelfand 
correspondence (see \cite[\S7B]{Ei}) yields a cochain complex of 
free $S$-modules, 
\begin{equation}
\label{eq:univ aomoto}
\begin{tikzcd}[column sep=24pt]
(H^{\hdot}\otimes_{\k} S,\delta)\colon 
H^{0}\otimes_{\k} S \ar[r, "\delta^0_H"]
&H^{1} \otimes_{\k} S \ar[r, "\delta^1_H"]
&H^{2} \otimes_{\k} S \ar[r] & \cdots.
\end{tikzcd}
\end{equation}

Upon identifying $H_1$ with $S_1$, the differentials in \eqref{eq:univ aomoto} 
are given by multiplication by the canonical element, $\omega \in H^1 \otimes_{\k} H_1$. 
For $i\ge 1$, let $\partial_i^{H}\colon H_i\otimes_{\k} S \to H_{i-1}\otimes_{\k} S$ 
be the map $S$-dual to $\delta^{i-1}_H$.  Much as before, we define 
the infinitesimal Alexander module of $G$ over $\k$ to be the $S$-module 
\begin{equation}
\label{eq:inf-amod-k}
\AA(G,\k)\coloneqq \coker (\partial_2^{H}) \, .
\end{equation}
Clearly, $\AA(G;\k)=\AA(G)\otimes \k$, provided $G_{\abf}$ is finitely 
generated.

Finally, if $\alpha\colon G\to H$ is a homomorphism, we define as before morphisms 
$\B(\alpha;\k)$ and $\AA(\alpha;\k)$ covering the ring map $\tilde\alpha_{\abf}\colon 
\Sym(H_1(G;\k)) \to \Sym(H_1(H;\k))$. When $G_{\abf}$ and $H_{\abf}$ are finitely 
generated, $\B(\alpha;\k)=\B(\alpha)\otimes \k$ and $\AA(\alpha;\k)=\AA(\alpha)\otimes \k$.

\subsection{An infinitesimal Crowell exact sequence}
\label{subsec:inf-crowell}
Pursuing our analogy with the classical theory of Alexander invariants 
and modules, we now set up the infinitesimal counterpart (over a field 
$\k$ of characteristic $0$) of Crowell's exact sequence \eqref{eq:crowell}.

\begin{theorem}
\label{thm:inf-crowell}
Let $G$ be a group with $b_1(G)<\infty$. Set $S=\Sym(H_1(G;\k))$, 
and let $\II(G;\k)$ be the maximal ideal of $S$ at $0$.  There is then a 
natural exact sequence of $S$-modules, 
\begin{equation}
\label{eq:ag-bg}
\begin{tikzcd}[column sep=18pt]
0\ar[r]& \B(G;\k)\ar[r]&\AA(G;\k)\ar[r]&\II(G;\k)\ar[r]& 0 \, .
\end{tikzcd}
\end{equation}
\end{theorem}

\begin{proof}
As noted in \S\ref{subsec:BGG}, the BGG correspondence 
yields cochain complexes of free 
$S$-modules, $(H^{\hdot}\otimes_{\k} S, \delta_H)$ and 
$(E^{\hdot}\otimes_{\k} S, \delta_E)$. 
Dualizing these complexes and setting $E_i=(E^i)^{\vee}$ and 
$\partial_{i+1}=(\delta^i)^{*}$, the functoriality of the 
correspondence gives a commuting diagram, 
\begin{equation}
\label{eq:koszul}
\begin{tikzcd}[column sep=24pt, row sep=24pt]
& H_2\otimes_{\k} S\ar[r, "\partial_2^H"] \ar[d, "\nabla_H \otimes S"] 
& H_1\otimes_{\k} S\ar[r, "\partial_1^H"] \ar[equal]{d}
& H_0\otimes_{\k} S  \ar[equal]{d} \phantom{,}
\\
E_3\otimes_{\k} S\ar[r, "\partial_3^E"]  & E_2\otimes_{\k} S\ar[r, "\partial_2^E"] 
& E_1\otimes_{\k} S\ar[r, "\partial_1^E"] & E_0\otimes_{\k} S ,
\end{tikzcd}
\end{equation}
where the bottom row is the beginning of the standard Koszul complex, while 
$\nabla_H\colon H_2\to E_2=H_1\wedge H_1$ is the comultiplication map 
$\cup_G^{\vee}\otimes \k$ from \eqref{eq:nabla}.
As shown in \cite[Theorem 6.2]{PS-imrn04}, the $S$-module $\B(G;\k)$ 
admits the presentation 
\begin{equation}
\label{eq:bbmod-pres}
\begin{tikzcd}[column sep=20pt]
(E_3 \oplus H_2)\otimes_{\k} S \ar["\partial_3^E + \nabla_G \otimes S"]{r} 
&[30pt] E_2\otimes_{\k} S \ar[r]&\B(G;\k)  \ar[r]&  0 \, .
\end{tikzcd}
\end{equation}

By definition, $\AA(G;\k)$ is equal to $\coker (\partial_2^{H})$. 
We infer from \eqref{eq:koszul} that the Koszul differential $\partial_2^E$ 
descends to a map $\B(G;\k)\to \AA(G;\k)$. Noting that $E_0\otimes_{\k} S=S$, 
we may identify $\II(G;\k)$ with $\im(\partial_1^E)$. A diagram chase 
now yields the exact sequence \eqref{eq:ag-bg}. 

Finally, let $\alpha\colon G\to H$ be a group homomorphism. 
It is readily checked that the restriction of $\AA(\alpha; \k)$ to 
$\B(G;\k)$ coincides with $\B(\alpha;\k)$, and induces the map 
$\tilde\alpha_{\abf}$ on augmentation ideals. 
This verifies the naturality of \eqref{eq:ag-bg}, 
and completes the proof.
\end{proof}

\subsection{Chen ranks and $1$-formality}
\label{subsec:chen-formal}
Recall we defined in \S\ref{subsec:gr-formal} the holonomy 
Chen ranks of a group $G$ with $b_1(G)<\infty$ as  $\bar\theta_n(G)=
\dim_{\k} \left(\h(G;\k)/\h(G;\k)''\right)_n$, where $\k$ is a field of 
characteristic $0$. The proof of \cite[Proposition 8.1]{SW-mz}---%
adapted to our slightly more general setting---shows that the 
following infinitesimal version of Massey's correspondence holds. 

\begin{proposition}[\cite{SW-mz}]
\label{prop:inf-massey-corr}
Let $G$ be a group with $b_1(G)<\infty$. Then 
\begin{equation}
\label{eq:inf-massey-corr}
\bar\theta_{n}(G)=\dim_{\k} \B_{n-2}(G;\k), \text{ for all $n\ge 2$} .
\end{equation}
\end{proposition}

Recall also that the Chen ranks 
$\theta_{n}(G)$ satisfy the inequality $\theta_{n}(G)\le \bar\theta_{n}(G)$.
As we shall see below, when the group $G$ is $1$-formal, 
this becomes an equality for all $n$. The key  
result towards establishing this fact (proved 
in \cite[Theorem 5.6]{DPS-duke} and further enhanced  
in \cite{SW-mz}) uses the $1$-formality hypothesis to 
construct a functorial isomorphism between the 
Alexander invariant and its infinitesimal version, at the level of 
completions (over $\Q$).

\begin{theorem}[\cite{DPS-duke}]
\label{thm:linalex-c}
Let $G$ be a $1$-formal group.  There is then a natural, 
filtration-preserving isomorphism of completed modules, 
$\widehat{B(G)\otimes \Q} \cong \widehat{\B(G)\otimes \Q}$. 
\end{theorem}

Passing to associated graded modules, we obtain as an 
immediate corollary the following result.

\begin{corollary}
\label{cor:linalex-gr}
If $G$ is $1$-formal, then $\gr(B(G) \otimes \Q)\cong \B(G) \otimes\Q$, 
as graded modules over $\gr(\Q[G_{\ab}]) \cong \Sym(H_1(G;\Q))$.
\end{corollary}

The next corollary was first proved in \cite[Theorem 4.2]{PS-imrn04}; 
we provide here a short proof, based on the above results.

\begin{corollary}[\cite{PS-imrn04}]
\label{cor:ps-chen}
If $G$ is $1$-formal, then $\theta_n(G)=
\bar\theta_n(G)=\dim_{\Q} \B_{n-2}(G;\Q)$ 
for all $n\ge 2$.
\end{corollary}
\begin{proof}
The classical Massey correspondence, as summarized in 
formula \eqref{eq:hilb-b-chen}, implies that 
$\theta_n(G)$ is equal to 
$\dim_{\Q} \gr_{n-2}(B(G)\otimes \Q)$. 
In turn, Corollary \ref{cor:linalex-gr} allows 
us to replace the dimension of this vector space 
by the dimension of $\B_{n-2}(G;\Q)$, which 
we know from Proposition \ref{prop:inf-massey-corr} 
is equal to $\bar\theta_n(G)$.
\end{proof}

\section{Resonance varieties}
\label{sect:res}

The resonance varieties of a group are a different kind of jump loci, 
built solely from cohomological information in low degrees. In this section, 
we relate these varieties to the support loci of the infinitesimal Alexander 
invariant,  and discuss some of their properties. 

\subsection{A stratification of the first cohomology group}
\label{subsec:res-strat}
Let $G$ be a group, and let $H^{\hdot}=H^{\hdot}(G;\C)$ be its cohomology 
algebra over $\C$.  For our purposes here, we will only consider 
the truncated algebra $H^{\le 2}$; moreover, 
we will assume that $b_1(G)=\dim_{\C} H^1$ is finite.
For each element $a\in H^1$, we have $a^2=0$, 
and so left-multiplication by $a$ defines a cochain complex, 
\begin{equation}
\label{eq:aomoto}
\begin{tikzcd}[column sep=26pt]
(H , \delta_a)\colon  \ 
H^0\ar[r, "\delta^0_a"] & H^1\ar[r, "\delta^1_a"]& H^2 , 
\end{tikzcd}
\end{equation}
with differentials $\delta^i_a(u)=a\cdot u$ for $u\in H^i$. 
It is readily checked that the specialization of the cochain complex 
\eqref{eq:univ aomoto} at $a$ coincides with \eqref{eq:aomoto}; see  
for instance \cite{DSY17,Su-edinb}.

The  resonance varieties witness the extent to which this 
complex fails to be exact in the middle. More precisely, for each $k\ge 1$, 
the {\em depth $k$ resonance variety}\/ of $G$ is defined as 
\begin{equation}
\label{eq:rv}
\RR_k(G)\coloneqq \{a \in H^1 \mid \dim_{\C} H^1(H, \delta_a) \ge k\}.
\end{equation}

These sets are homogeneous algebraic subvarieties of the affine space 
$H^1=H^1(G;\C)$. Clearly, $0\in \RR_k(G)$ 
if and only if $k\le b_1(G)$; in particular, $\RR_1(G)= \emptyset $ 
if and only if $b_1(G)=0$.  Furthermore, we have a descending filtration, 
\begin{equation}
\label{eq:res-strat}
H^1(G;\C) \supseteq \RR_1(G)  \supseteq \RR_2(G)  \supseteq \cdots 
 \supseteq \RR_{r}(G) \supseteq \RR_{r+1}(G)=\emptyset  ,
\end{equation}
where $r=b_1(G)$.  A linear subspace $U\subset H^1$ is said to be 
isotropic if the restriction of $\cup_G\colon H^1\wedge H^1\to H^2$  
to $U\wedge U$ vanishes; that is, $ab=0$ for all $a,b\in U$. 
As noted in \cite[Lemma 2.2]{Su-edinb}, the variety 
$\RR_{k}(G)$ contains every  isotropic subspace of $H^1$ 
whose dimension is at most $k+1$; moreover, $\RR_1(G)$ 
is the union of all isotropic planes in $H^1$.

The following (well-known) lemma shows that the resonance varieties 
are determinantal varieties of the infinitesimal Alexander module, and 
thus, Zariski closed subsets of the affine space $H^1$. Proofs 
in various levels of generality have been given, for instance,  
in \cite{MS00,PS-mrl,Su-edinb}. We give here a quick proof, 
in a slightly greater generality, along 
the lines of the proof of Lemma \ref{lem:hironaka}.

\begin{lemma}
\label{lem:res-fitt}
Let $G$ be a group with $b_1(G)<\infty$. Then, for all $k\ge 1$, 
\[
\RR_{k}(G)=V(\Fitt_{k+1}(\AA(G;\C))) \, ,
\]
at least away from $0\in H^1(G;\C)$, with equality at $0$ for $k<b_1(G)$. 
\end{lemma}

\begin{proof}
Let $a\in H^1$. By the above discussion, we have that 
$\delta^i_H(a)=\delta^i_a$. Thus, $a$ belongs to $\RR_k(G)$ if and only if 
$\rank \partial_{2}^H(a) + \rank \partial_{1}^H(a) \le b_1(G) - k$. 
But $H_0=\C$ and $\partial_{1}^H(a) =0$ if and only if $a=0$. 
Since $\AA(G,\C)=\coker (\partial_2^{H})$, the lemma follows.
\end{proof}

\subsection{Resonance and exterior powers}
\label{subsec:res-ext}

The next result identifies the depth-$k$ resonance variety of a group as the 
support locus of the $k$-th exterior power of its infinitesimal Alexander 
invariant. The result was first proved in \cite[Lemma~4.2]{DPS-serre} 
for finitely presented groups, using a specialization argument; 
in fact, the same proof works for finitely generated groups, see 
\cite[Lemma~5.1]{DP-ann}. We offer here a completely different 
proof, in greater generality, using the BGG correspondence and 
the localization approach from the proof of Theorem \ref{thm:cvb}. 

\begin{theorem}[\cite{DP-ann, DPS-serre}]
\label{thm:res-supp}
Let $G$ be a group with $b_1(G)<\infty$. Then 
\begin{equation}
\label{eq:res-supp}
\RR_k(G) = \supp \big( \bwedge^k \B(G)\otimes \C\big) 
\end{equation}
for all $k\ge 1$, at least away from $0\in H^1(G;\C)$.
\end{theorem}

\begin{proof}
Let $0\to \B \to \AA \to \II \to 0$ be the infinitesimal Crowell exact 
sequence from Theorem \ref{thm:inf-crowell}, over $\k=\C$. 
Localizing at a maximal ideal $\m\in \Spec(S)$, we obtain 
an exact sequence of $S_{\m}$-modules, 
$0\to \B_{\m} \to \AA_{\m} \to \II_{\m} \to 0$. 
On the other hand, localizing at $\m$ the exact sequence 
$0\to \II \to S\to \C \to 0$, we get the 
exact sequence $0\to \II_{\m} \to S_{\m} \to \C_{\m} \to 0$. 
Assuming $\m\ne \II$, we have that $\C_{\m}=0$,
and so $\II_{\m}=S_{\m}$. Hence, $ \bwedge^{j} \II_{\m}$ 
is isomorphic to $S_m$ if $j=0,1$ and is equal to $0$ if $j>1$. 

Lemma \ref{lem:ext-mod} yields an exact sequence, 
$0\to \bwedge^{k+1} \B_{\m} \to \bwedge^{k+1}  \AA_{\m} 
\to \bwedge^{k} \B_{\m} \to 0$, 
from which we deduce that $\supp \!\big(\bwedge^k \AA\big)=
\supp \!\big(\bwedge^{k-1} \B\big)$, at least away 
from $\supp(\II)=\{0\}$.  
On the other hand, by Lemma \ref{lem:ann-supp}, 
we have that $\supp \big(\bwedge^{k} \AA \big)=V(\Fitt_{k}(\AA))$. 
Applying now Lemma \ref{lem:res-fitt} completes the proof.
\end{proof}

As a consequence of Theorem \ref{thm:res-supp}, we infer that the resonance 
varieties only depend on the holonomy Lie algebra of the group. More precisely, 
let $G_1$ and $G_2$ be two groups with finite first Betti number, and suppose 
that $\h(G_1; \C)\cong \h(G_2;\C)$, as graded Lie algebras. 
There is then a linear isomorphism, $H^1(G_1;\C) \cong H^1(G_2;\C)$, 
restricting to isomorphisms $\RR_k(G_1)\cong \RR_k(G_2)$ for all $k\ge 1$.

\subsection{The Tangent Cone formula}
\label{subsec:res-formal}

Let us identify the tangent space to 
the character group $\T_G=H^1(G;\C^*)$ with the linear space $H^1(G;\C)$. 
We denote by $\TC_1(\V_k(G))$ the tangent cone at the identity to the 
characteristic variety $\V_k(G)$; clearly, this set coincides with $\TC_1(\WW_k(G))$. 
It is known that $\TC_1(\WW_k(G))$ is always a (homogeneous) subvariety 
of the resonance variety $\RR_k(G)$. The basic relationship between the 
characteristic and resonance varieties in the $1$-formal setting is 
encapsulated in the ``Tangent Cone formula"  from \cite[Theorem A]{DPS-duke}, 
which we recall next.

\begin{theorem}[\cite{DPS-duke}]
\label{thm:tcone}
If $G$ is a $1$-formal group, then $\TC_1(\WW_k(G))=\RR_k(G)$, 
for all $k\ge 1$.
\end{theorem}

Far-reaching generalizations of this theorem are now known, 
but this is all we will need for our purposes here. 

\subsection{Vanishing resonance}
\label{subsec:res-zero}
Particularly interesting---and important for many applica\-tions---is the case 
when the resonance vanishes. In fact, as shown in \cite[Theorem 3.3]{PS-mrl} and 
\cite[Propositions 2.6 and 2.7]{PS-crelle}, the case when $\RR_1(G)=\{0\}$  
is generic, in a sense that can be made very precise.

The next result generalizes \cite[Theorem 5.2]{DP-ann}, which is only valid 
in depth $k=1$, and is only proved there for finitely generated groups $G$ 
(as is  \cite[Lemma 2.4]{PS-crelle}, too).

\begin{corollary}
\label{cor:zero-res}
Let $G$ be a group with $b_1(G)<\infty$. For each $k\ge 1$, the following 
conditions are equivalent.
\begin{enumerate}[itemsep=1pt]
\item The resonance variety $\RR_k(G)$ is empty or equal to $\{0\}$. 
\item  The  $\Q$-vector space $\bwedge^k \B(G;\Q)$ 
is finite-dimensional.
\end{enumerate}
\end{corollary}

\begin{proof}
The claim follows from Theorem \ref{thm:res-supp} using the same argument 
as in the proof of Corollary \ref{cor:trivial-cv}, applied this time 
to the vector space $M=\bwedge^k \B(G;\C)$, viewed as a 
module over the $\C$-algebra $S=\Sym(H_1(G;\C))$, and taking 
into the account the homogeneity of $\RR_k(G)$.
\end{proof}

\begin{corollary}
\label{cor:trivial-res-holo-chen}
Let $G$ be a group with $b_1(G)<\infty$, and suppose 
$\RR_1(G)=\{0\}$. Then the holonomy Chen ranks $\bar\theta_n(G)$ 
vanish for $n$ sufficiently large ($n\gg 0$).
\end{corollary}

\begin{proof}
By Proposition \ref{prop:inf-massey-corr}, we have that 
$\bar\theta_n(G)=\dim_{\Q} \B_{n-2}(G;\Q)$ for all $n\ge 2$. 
The claim now follows from Corollary \ref{cor:zero-res}, in 
the case when $k=1$.
\end{proof}

\begin{corollary}
\label{cor:trivial-res-chen}
Let $G$ be a $1$-formal group, and suppose 
$\RR_1(G)=\{0\}$. Then the Chen ranks $\theta_n(G)$ 
vanish for $n\gg 0$.
\end{corollary}

\begin{proof}
Follows at once from Corollaries \ref{cor:ps-chen} 
and \ref{cor:trivial-res-holo-chen}.
\end{proof}

\begin{remark}
\label{rem:effective}
The range where those Chen ranks vanish has been made precise in 
\cite{AFPRW1}: If $G$ is $1$-formal,  $\RR_1(G)=\{0\}$, and $b_1(G)\ge 3$, 
then $\theta_n(G)=0$ for $n\ge b_1(G)-1$. Generalizations of this 
result to the setting where the resonance does not necessarily vanish  
will be given in \cite{AFRS}. 
\end{remark}

\begin{remark}
\label{rem:zero-res}
By \cite[Theorem C]{DP-ann} and \cite[Corollary 6.3]{PS-mrl}, the 
following holds for any finitely generated group $G$: If 
$\RR_1(G)\subseteq \{0\}$, then $\dim_{\Q} \widehat{B(G)\otimes \Q}<\infty$. 
The converse is not true in general, but it is valid in the case when 
$G$ is $1$-formal.
\end{remark}

\section{Resonance varieties in group extensions}
\label{sect:res-ext}

In this final section we investigate the behavior of the resonance varieties 
under certain kinds of group extensions.

\subsection{Maps between resonance varieties}
\label{subsec:func-res}
Let $\alpha\colon G\to H$ be a homomorphism between two 
finitely generated groups. The induced homomorphism in first cohomology, 
$\alpha^*\colon H^1(H;\C) \to H^1(G;\C)$, may not 
preserve the resonance varieties.
For instance, take a cyclic subgroup  $\Z<F_n$ ($n\ge 2$) as 
in Example \ref{ex:z-f2}. The inclusion $\iota\colon \Z\to F_n$ induces  
a surjective morphism in first cohomology, $\iota^*\colon \C^n \surj \C$; 
this morphism sends $\RR_1(F_n)=\C^n$ onto $\C$, which 
strictly contains $\RR_1(\Z)=\{0\}$. We will further illustrate 
this phenomenon in Example \ref{ex:heisenberg}.

Nevertheless, the resonance varieties enjoy a partial naturality property 
similar to the one possessed  by the characteristic varieties.  

\begin{proposition}[\cite{PS-mathann06, Su-toul}]  
\label{prop:r1-nat}
Let  $G$ be a finitely generated group, and let $\pi\colon G\surj Q$ 
be a surjective homomorphism.   Then the induced homomorphism, 
$\pi^*\colon H^1(Q, \C) \to H^1(G,\C)$, is injective, and restricts 
to embeddings $\RR_k(Q) \inj \RR_k(G)$ for all $k\ge 1$.
\end{proposition}

Starting directly from definition \eqref{eq:rv}, a proof for $k=1$ was 
given in \cite[Lemma 5.1]{PS-mathann06}, while the general case was proved 
in \cite[Proposition A.1]{Su-toul}. Alternatively, the proof of Proposition \ref{prop:v1-nat} 
can be readily adapted to this context, with the $\Z[G_{\ab}]$-module $B(G)$
replaced by the $\Sym[G_{\abf}]$-module $\B(G)$, and with $\RR_k(G)$ 
defined as in \eqref{eq:res-supp}.

\subsection{Resonance in split-exact sequences}
\label{subsec:reson-split}

Our main goal in this section is to relate the infinitesimal Alexander invariants 
and the resonance varieties of a group $G$ to those of a normal subgroup $K$, 
under suitable hypothesis. We start with the case when $G=K\rtimes Q$ is 
a semidirect product.

\begin{theorem}
\label{thm:res-split-abf}
Let $\begin{tikzcd}[column sep=16pt]
\!\!1\ar[r] & K\ar[r, "\iota"]
& G \ar[r, "\pi"] & Q\ar[r] & 1\!\!
\end{tikzcd}$ be a split exact sequence of finitely generated groups 
such that $Q$ acts trivially on $H_1(K;\Q)$ and $G$ 
is graded formal. Then,
\begin{enumerate}[itemsep=3pt]
\item \label{hr1}
$\begin{tikzcd}[column sep=18pt]
\!\!0\ar[r] & \h(K)\otimes \Q\ar[rr, "\!\!\h(\iota)\otimes \Q"]
&& \h(G)\otimes \Q \ar[rr, "\!\!\h(\pi)\otimes \Q"] && \h(Q)\otimes \Q\ar[r] & 0 \!\!
\end{tikzcd}$
is a split exact sequence of graded Lie algebras.
\item \label{hr2}
Suppose $Q$ is abelian.  Then 
\begin{enumerate}[itemsep=2pt,topsep=1pt]
\item \label{bbk1}
The map $\B(\iota)\colon  \B(K)\to \B(G)$ gives  rise to a 
$\Sym(H_1(K;\Q))$-linear isomorphism,  
$\B(K)\otimes \Q\isom \B(G)_{\iota}\otimes \Q$.
\item \label{bbk2}
$\bar\theta_n(K)\le \bar\theta_n(G)$ for all $n\ge 1$. 
\item \label{bbk3}
The morphism $\iota^*\colon H^1(G;\C) \surj H^1(K;\C)$ 
restricts to maps $\iota^* \colon \RR_k(G)\to \RR_k(K)$ for all $k\ge 1$; 
moreover, the map $\iota^* \colon \RR_1(G)\surj \RR_1(K)$ is a 
surjection.
\end{enumerate}
\end{enumerate}
\end{theorem}

\begin{proof} 
Since $K$ is finitely generated and $Q$ acts trivially on $H_1(K;\Q)$,  
Proposition \ref{prop:abf-exact-split} insures that the sequence 
$1\to K\to G\to Q\to 1$ is $\abf$-exact; in particular, $\iota_{\abf}$ 
is injective and $\bar\theta_1(K)=\rank K_{\abf}$ is less or equal to 
$\bar\theta_1(G)=\rank G_{\abf}$.
By Corollary \ref{cor:fr-abf-rat}, the given sequence yields a 
split exact sequence of associated graded $\Q$-Lie algebras, 
$0\to \gr(K)\otimes \Q\to \gr(G)\otimes \Q\to \gr(Q)\otimes \Q\to 0$. 

The assumption that $G$ is graded formal means that the natural 
morphism $\h(G)\otimes \Q\surj \gr(G)\otimes \Q$ is an isomorphism. 
By \cite[Theorem 5.11]{SW-forum}, $Q$ is graded formal, since  it  
is a retract of $G$, a graded formal group. 
Likewise, Theorem \ref{thm:abf-gr-formal} implies that $K$ 
is graded formal. Putting things together shows that the sequence 
from \eqref{hr1} is also a split exact sequence of graded $\Q$-Lie algebras.

Assume now that $Q$ is abelian. We then have $\gr_n(Q)=0$ for all $n\ge 2$, 
and so $\h_n(Q)\otimes \Q=0$ for all $n\ge 2$. Using part \eqref{hr1}, 
we infer that $\h(K)'\otimes \Q\cong \h(G)'\otimes \Q$. Therefore, 
we also have $\h(K)''\otimes \Q\cong \h(G)''\otimes \Q$.  
Thus, the induced map, $\B(K)\otimes \Q\to 
\B(G)_{\iota}\otimes \Q$, is an isomorphism of modules 
over $\Sym(H_1(K;\Q))$, and claim \eqref{bbk1} is established. 

By Proposition \ref{prop:inf-massey-corr}, we have that 
$\bar\theta_{n}(G)=\dim_{\k} \B_{n-2}(G)\otimes \Q$ for $n\ge 2$, and 
likewise for $\bar\theta_{n}(K)$. Therefore, claim \eqref{bbk2} for $n\ge 2$ 
follows from part \eqref{bbk1}, using  the injectivity of the map 
$\tilde\iota_{\abf}\colon \Sym(K_{\abf})\to  \Sym(G_{\abf})$
and a reasoning similar to the proof of Theorem \ref{thm:alex-lcs-ngq}, 
part \eqref{ngq4}.

By Theorem \ref{thm:res-supp}, we have that 
$\RR_k(G) = \supp ( \bwedge^k \B(G)\otimes \C )$ and similarly for  
$\RR_k(K)$, at least away from $0$. Claim \eqref{bbk3} now follows 
from part \eqref{bbk1}, in a manner similar to the proof of Theorem \ref{thm:cv-abf}.
\end{proof} 

\begin{corollary}
\label{cor:zero-res-kg}
Let $1\to K\to G \to Q\to 1$ be a split exact sequence of finitely 
generated groups. Assume $Q$ is abelian and acts trivially on $H_1(K;\Q)$,  
while $G$ is graded formal and $\RR_1(G)\subseteq \{0\}$. Then,  
\begin{enumerate}[itemsep=2pt]
\item \label{vr1}
$\dim_{\Q}\B(K)\otimes \Q<\infty$ and $\RR_1(K)\subseteq \{0\}$. 
\item  \label{vr2}
$\bar\theta_n(K)\le \bar\theta_n(G)$ for all $n\ge 1$  
and  $\bar\theta_n(K)=0$ for $n\gg 0$.
\end{enumerate}
\end{corollary}

\begin{proof}
All claims follow directly from Theorem \ref{thm:res-split-abf}, part \eqref{hr2}, 
and from Corollaries \ref{cor:zero-res} and \ref{cor:trivial-res-holo-chen}.
\end{proof}

\subsection{Discussion and examples}
\label{subsec:res-discuss}
Let us discuss the necessity of some of the assumptions 
we made in Theorem \ref{thm:res-split-abf}. The extension 
$1\to F_n \to P_n \to P_{n-1}\to 1$ ($n\ge 4$) 
from Example \ref{ex:braids} can be used again to show 
that in part \eqref{hr2} we need to assume $Q$ to be abelian. 
In the next example, $Q$ is abelian but acts non-trivially on 
$H_1(K;\Q)$, while $G$ is not graded-formal. 

\begin{example}
\label{ex:heisenberg}
Let $G$ be the Heisenberg group from 
Example \ref{ex:heis-alex} and Remark \ref{rem:central ex}. 
This group can be realized as a split extension of the form 
$\Z^2\rtimes_{\varphi} \Z$, with monodromy given by the matrix 
$\big(\begin{smallmatrix}  1 & 1\\0& 1  \end{smallmatrix}\big)$.
The inclusion $\iota\colon \Z^2 \inj G$ 
induces a homomorphism $\iota^*\colon H^1(G;\C)\to H^1(\Z^2;\C)$ 
which can be identified with the linear map $\iota^*\colon \C^2\to \C^2$ with 
matrix $\big(\begin{smallmatrix}  0 & 0\\0& 1  \end{smallmatrix}\big)$. 
Since $\cup_G=0$, we have that $\RR_1(G)=\C^2$; 
thus, $\iota^*$ does not take $\RR_1(G)$ to $\RR_1(\Z^2)=\{0\}$. 
For comparison, though, note that the map $\iota^*\colon \T_{G}\to \T_{\Z^2}$ 
does take $\V_1(G)=\{1\}$ to $\V_1(\Z^2)=\{1\}$; the discrepancy is due to 
the fact that $G$ is not $1$-formal.
\end{example}

Given an exact sequence as in Theorem \ref{thm:res-split-abf} 
(or as in Theorem \ref{thm:res-abf} below), the morphisms 
$\iota^*\colon \RR_k(G)\to \RR_k(K)$ may fail to be 
surjective for $k>1$. The reason is the same as the 
one given in Remark \ref{rem:higher-depth}: exterior 
powers do not necessarily commute with restriction 
of scalars, and so the map 
$\bwedge^k\, \B(K)\otimes \Q\to \bwedge^k \,\B(G)_{\iota} \otimes \Q$ 
may fail to be injective for $k>1$. We illustrate this phenomenon 
with an example.

\begin{example}
\label{ex:high-depth-res}
Let $G$ be the arrangement group from Example \ref{ex:high-depth-cv}.
As noted previously, this group fits into a split, $\ab$-exact 
sequence, $1\to F_{3} \xrightarrow{\iota} G \to \Z\to 1$.  Moreover, 
the group $G$ is graded formal, see \cite[Theorem 7.6]{SW-jpaa}.
Thus, all the hypothesis of Theorem \ref{thm:res-split-abf} 
are satisfied.  Now, as shown in \cite[Example 6.3]{MS00}, 
the variety $\RR_1(G)$ consists of two hyperplanes in 
$\C^4$, while $\RR_2(G)$ consists of two lines. It is 
readily seen that $\iota^*$ sends $\RR_1(G)$ onto 
$\RR_1(F_3)=\C^3$, as predicted; by dimension reasons, though, 
$\iota^*(\RR_2(G))$ is strictly contained in $\RR_2(F_3)=\C^3$.
\end{example}

\subsection{Resonance in $\ab$-exact and $\abf$-exact sequences}
\label{subsec:reson-abf}

We now relax the split exactness assumption from the preceding 
theorem at the price of making a more stringent formality assumption. 
Due to its rather broad scope, the next theorem is one of the main 
results of this paper.

\begin{theorem}
\label{thm:res-abf}
Let $\begin{tikzcd}[column sep=14pt]
\!\!1\ar[r] & K\ar[r, "\iota"]
& G \ar[r] & Q\ar[r] & 1\!\!
\end{tikzcd}$ be an exact sequence of groups, and assume 
the following hold.
\begin{enumerate}[(i)]
\item  \label{rab1} 
Either the sequence if $\ab$-exact and $Q$ is abelian, 
or the sequence is $\abf$-exact and $Q$ is torsion-free abelian.  
\item \label{rab2} 
Both $G$ and $K$ are $1$-formal.
\end{enumerate}
Then 
\begin{enumerate}[itemsep=2pt]
\item \label{xx1}
The map $\B(\iota)\colon  \B(K)\to \B(G)$ gives rise to a 
$\Sym(H_1(K;\Q))$-linear isomorphism,  
$\B(K)\otimes \Q\isom \B(G)_{\iota}\otimes \Q$.

\item \label{xx2}
$\theta_n(K)\le \theta_n(G)$ for all $n\ge 1$. 

\item \label{xx3}
The morphism $\iota^*\colon H^1(G, \C) \surj H^1(K, \C)$ 
restricts to maps $\iota^* \colon \RR_k(G)\to \RR_k(K)$ for all $k\ge 1$; 
moreover, the map $\iota^* \colon \RR_1(G)\surj \RR_1(K)$ is a 
surjection.
\end{enumerate}
\end{theorem}

\begin{proof} 
Assumption \ref{rab1} says that the hypothesis of either 
Theorem \ref{thm:alex-abex} or Theorem \ref{thm:alex-lcs-ngq} 
are satisfied. In either case, we infer that the map 
$B(\iota)\colon  B(K)\to B(G)$ 
gives rise to a $\Q[H_1(K;\Q)]$-linear isomorphism,  
$B(K)\otimes \Q\to B(G)_{\iota}\otimes \Q$. 
On the other hand, the formality assumption \ref{rab2} 
and Corollary \ref{cor:linalex-gr} yield functorial isomorphisms 
$\gr(B(K))\otimes \Q \cong \B(K)\otimes \Q$ 
and $\gr(B(G))\otimes \Q \cong \B(G)\otimes \Q$. 
Claim \eqref{xx1} readily follows.

Under assumption \ref{rab1}, claim \eqref{xx2} 
follows directly from Theorem \ref{thm:alex-abex}, part \eqref{ng4} 
in the $\ab$-exact case, or from 
Theorem \ref{thm:alex-lcs-ngq}, part \eqref{ngq4} 
in the $\abf$-exact case. Alternatively, the claim 
follows from part \eqref{xx1}, using the formality 
assumption \ref{rab2} and Corollary \ref{cor:ps-chen}. 

Assumption \ref{rab1} also says that the hypothesis of 
Theorem \ref{thm:cv-abf}---from either part \eqref{cv1} or 
part \eqref{cv2}---are satisfied. In both cases, the map $\iota^*\colon 
\T_G\to \T_K$ restricts to maps $\iota^*\colon\WW_k(G)\to \WW_k(K)$ 
for all $k\ge 1$, with the map $\iota^*\colon\WW_1(G)\to \WW_1(K)$ 
being a surjection.  Taking tangent cones at the identities of 
$\T_G$ and $\T_K$, respectively, we infer that 
the homomorphism $\iota^*\colon H^1(G;\C) \surj H^1(K;\C)$ restricts 
to maps $\TC_1(\WW_k(G))\to \TC_1(\WW_k(K))$ for $k\ge 2$ and 
to a surjection $\TC_1(\WW_1(G))\surj \TC_1(\WW_1(K))$. 

Finally, by virtue of our $1$-formality assumptions on $G$ and $K$, 
Theorem \ref{thm:tcone} allows us to replace 
$\TC_1(\WW_k(G))$ with $\RR_k(G)$ and 
$\TC_1(\WW_k(K))$ with $\RR_k(K)$ 
for each $k\ge 1$. This proves claim \eqref{xx3}.
\end{proof}

\begin{corollary}
\label{cor:zero-res-k}
With the notation and assumptions of Theorem \ref{thm:res-abf}, 
suppose $\RR_1(G)\subseteq \{0\}$. Then,
\begin{enumerate}[itemsep=2pt]
\item \label{rr1}
$\dim_{\Q}\B(K)\otimes \Q<\infty$ and $\RR_1(K)\subseteq \{0\}$. 
\item  \label{rr2}
$\theta_n(K)\le \theta_n(G)$ for all $n\ge 1$  
and  $\theta_n(K)=0$ for $n\gg 0$.
\end{enumerate}
\end{corollary}

\begin{proof}
All claims follow from the theorem and from 
Corollaries \ref{cor:zero-res} and \ref{cor:trivial-res-chen}.
\end{proof}

We conclude with a general class of examples where 
Theorem \ref{thm:res-abf} applies.

\begin{example}
\label{ex:raag-res}
For a finite, connected graph $\Gamma$, the right-angled Artin group 
$G_{\Gamma}$ is always $1$-formal (see \cite{PS-mathann06}), 
whereas the Bestvina--Brady group $N_{\Gamma}$ is $1$-formal 
whenever $\pi_1(\Delta_{\Gamma})=0$ (see \cite{PS-jlms07}), or, 
more generally, $H_1(\Delta_{\Gamma};\Q)=0$ (see \cite{PS-adv09}). 
When $\Delta_{\Gamma}$ is simply-connected, Theorem 
\ref{thm:res-abf}, part \eqref{xx2} recovers Lemma 8.3(ii) from 
\cite{PS-jlms07}. 
In the broader setting when $b_1(\Delta_{\Gamma})=0$ yet 
$\pi_1(\Delta_{\Gamma})\ne 0$ (for instance, when $\Gamma$ is the 
$1$-skeleton of a flag triangulation of $\RP^2$), our theorem still applies, 
although in this case $N_{\Gamma}$ is finitely generated yet 
not finitely presented.
\end{example}


\newcommand{\arxiv}[1]
{\texttt{\href{https://arxiv.org/abs/#1}{arXiv:#1}}}

\newcommand{\arxi}[1]
{\texttt{\href{https://arxiv.org/abs/#1}{arxiv:}}
\texttt{\href{https://arxiv.org/abs/#1}{#1}}}

\newcommand{\arxx}[2]
{\texttt{\href{https://arxiv.org/abs/#1.#2}{arxiv:#1.}}
\texttt{\href{https://arxiv.org/abs/#1.#2}{#2}}}

\newcommand{\doi}[1]
{\texttt{\href{https://dx.doi.org/#1}{doi:#1}}}

\renewcommand{\MR}[1]
{\href{https://www.ams.org/mathscinet-getitem?mr=#1}{MR#1}}

\newcommand{\Zbl}[1]
{\href{http://zbmath.org/?q=an:#1}{Zbl #1}}


\begin{thebibliography}{00}

\bibitem{AFPRW1}
M.~Aprodu, G.~Farkas, S.~Papadima, C.~Raicu, J.~Weyman,
\href{https://doi.org/10.1215/00127094-2022-0010}%
{\em Topological invariants of groups and {K}oszul modules}, 
Duke Math. Journal \textbf{171} (2022), no.~19, 2013--2046.
\MR{4484204}

\bibitem{AFPRW2}
M.~Aprodu, G.~Farkas, S.~Papadima, C.~Raicu, J.~Weyman,
\href{https://doi.org/10.1007/s00222-019-00894-1}%
{\em Koszul modules and Green's Conjecture}, 
Invent. Math. \textbf{218} (2019), 657--720.
\MR{4022070}

\bibitem{AFRS}
M.~Aprodu, G.~Farkas, C.~Raicu, A.~Suciu,
{\em An effective proof of the Chen ranks conjecture},
draft (2021).

\bibitem{BL} H.~Bass, A.~Lubotzky, 
\href{https://dx.doi.org/10.1090/conm/169/01651}%
{\em Linear-central filtrations on groups}, 
in: The mathematical legacy of Wilhelm Magnus: groups, 
geometry and special functions (Brooklyn, NY, 1992), 45--98, 
Contemp. Math., vol.~169, Amer. Math. Soc., Providence, RI, 1994. 
\MR{1292897}

\bibitem{BG} P.~Bellingeri, S.~Gervais,
\href{https://dx.doi.org/10.2140/agt.2016.16.547}%
{\em On $p$-almost direct products and residual properties
of pure braid groups of nonorientable surfaces}, Alg. Geom. Topol. 
\textbf{16} (2016) 547--568.
\MR{3470709}

\bibitem{BGG11} P.~Bellingeri, S.~Gervais, J.~Guaschi,
{\em Exact sequences, lower central series and
representations of surface braid groups}, \arxiv{1106.4982v1}.

\bibitem{BB} M.~Bestvina, N.~Brady, 
\href{https://doi.org/10.1007/s002220050168}%
{\em Morse theory and finiteness properties of groups}, 
Invent. Math. \textbf{129} (1997), no.~3, 445--470. 
\MR{1465330} 

\bibitem{Brown} K.S.~Brown, 
\href{https://dx.doi.org/10.1007/978-1-4684-9327-6}%
{\emph{Cohomology of groups}}, Corrected reprint of the 1982 original.  
Graduate Texts in Mathematics, vol.~87, Springer-Verlag, New York, 1994. 
\MR{1324339}

\bibitem{Chen51} K.-T.~Chen, 
\href{https://dx.doi.org/10.2307/1969316}%
{\emph{Integration in free groups}}, Ann. of Math. (2) \textbf{54} 
(1951), no.~1, 147--162.
\MR{0042414} 

\bibitem{Chen77} K.-T.~Chen,
\href{https://dx.doi.org/10.1016/0001-8708(77)90120-7}%
{\em Extension of $C^{\infty}$ function algebra by integrals 
and {M}alcev completion of $\pi_1$}, Adv. in Math. \textbf{23} 
(1977), no.~2, 181--210.
\MR{0458461}

\bibitem{CEP} T.~Church, M.~Ershov, A.~Putman, 
\href{https://dx.doi.org/10.4171/JEMS/1157}%
{\em On finite generation of the Johnson filtrations}, 
J. Eur. Math. Soc. (JEMS) \textbf{24} (2022), no.~8, 2875--2914.
\MR{4416592}

\bibitem{Co04} T.~Cochran, 
\href{https://dx.doi.org/10.2140/agt.2004.4.347}%
{\em Noncommutative knot theory}, 
Algebr. Geom. Topol. \textbf{4} (2004), 347--398. 
\MR{2077670}

\bibitem{CH05} T.~Cochran, S.~Harvey, 
\href{https://dx.doi.org/10.2140/gt.2005.9.2159}%
{\em Homology and derived series of groups}, 
Geom. Topol. \textbf{9} (2005), 2159--2191. 
\MR{2209369}

\bibitem{CH08} T.~Cochran, S.~Harvey, 
\href{https://dx.doi.org/10.1112/jlms/jdn046}%
{\em Homology and derived $p$-series of groups}, 
J. London Math. Soc. \textbf{78} (2008), 677--692. 
\MR{2456898}

\bibitem{CH10} T.~Cochran, S.~Harvey, 
\href{https://dx.doi.org/10.2140/pjm.2010.246.31}%
{\em Homological stability of series of groups}, 
Pacific J. Math. \textbf{246} (2010), nr.~1, 31--47. 
\MR{2645878}

\bibitem{CSc-adv} D.C.~Cohen, H.K.~Schenck,
\href{https://dx.doi.org/10.1016/j.aim.2015.07.023}%
{\em Chen ranks and resonance}, Adv. Math. 
\textbf{285} (2015), 1--27.
\MR{3406494}

\bibitem{CS-tams99} D.C.~Cohen, A.I.~Suciu, 
\href{https://dx.doi.org/10.1090/S0002-9947-99-02206-0}%
{\em Alexander invariants of complex hyperplane arrangements}, 
Trans. Amer. Math. Soc. \textbf{351} (1999), no.~10, 4043--4067. 
\MR{1475679}

\bibitem{Cooper} J.~Cooper, 
\href{https://dx.doi.org/10.1142/S1793525315500120}%
{\em Two mod-$p$ Johnson filtrations}, 
J. Topol. Anal. 7 (2015), no.~2, 309--343.
\MR{3326304}

\bibitem{Cr65}  R.H.~Crowell, 
\href{https://www.jstor.org/stable/24901277}%
{\em Torsion in link modules}, 
J. Math. Mech. \textbf{14} (1965), 289--298. 
\MR{0174606}

\bibitem{DSY17} G.~Denham, A.I.~Suciu, S.~Yuzvinsky, 
\href{https://dx.doi.org/10.1007/s00029-017-0343-5}%
{\em Abelian duality and propagation of resonance}, 
Selecta Math. \textbf{23} (2017), no.~4, 2331--2367. 
\MR{3703455}

\bibitem{DHP14} A.~Dimca, R.~Hain, S.~Papadima, 
\href{https://dx.doi.org/10.4171/JEMS/447}%
{\em The abelianization of the {J}ohnson kernel}, J. Eur. 
Math. Soc \textbf{16} (2014) no. 4, 805--822. 
\MR{3191977}

\bibitem{DP-pisa} A.~Dimca, S.~Papadima, 
\href{https://dx.doi.org/10.2422/2036-2145.2011.2.01}%
{\em Finite {G}alois covers, cohomology jump loci, formality 
properties, and multinets},  Ann. Sc. Norm. Super. Pisa Cl. Sci. 
\textbf{10} (2011), no.~2, 253--268. 
\MR{2856148} 

\bibitem{DP-ann} A.~Dimca, S.~Papadima, 
\href{https://doi.org/10.4007/annals.2013.177.2.1}%
{\em Arithmetic group symmetry and finiteness properties of 
{T}orelli groups}, Ann. Math. \textbf{177} (2013), no.~2, 395--423. 
\MR{3010803}

\bibitem{DPS-serre} A.~Dimca, S.~Papadima, A.I.~Suciu,  
{\em Formality, {A}lexander invariants, and a question of {S}erre}, 
unpublished manuscript, \arxiv{math.AT/0512480v3}.
 
\bibitem{DPS-imrn} A.~Dimca, S.~Papadima, A.I.~Suciu,
\href{https://dx.doi.org/10.1093/imrn/rnm119}%
{\em Alexander polynomials: {E}ssential variables and 
multiplicities}, Int. Math. Res. Not. IMRN \textbf{2008}, 
no.~3, Art. ID rnm119, 36 pp. 
\MR{2416998} 

\bibitem{DPS-duke} A.~Dimca, S.~Papadima, A.I.~Suciu,
\href{https://dx.doi.org/10.1215/00127094-2009-030}%
{\em Topology and geometry of cohomology jump loci}, 
Duke Math. Journal \textbf{148} (2009), no.~3, 405--457.
\MR{2527322} 

\bibitem{DF} W.G.~Dwyer, D.~Fried,
\href{https://dx.doi.org/10.1112/blms/19.4.350}%
{\em Homology of free abelian covers. \textup{I}}, Bull.
London Math. Soc. \textbf{19} (1987), no.~4, 350--352.
\MR{0887774} 

\bibitem{Eisenbud} D.~Eisenbud, 
\href{https://dx.doi.org/10.1007/978-1-4612-5350-1}
{\em Commutative algebra}, with a view towards algebraic geometry,
Grad. Texts in Math., vol.~150, Springer-Verlag, New~York, 1995.
\MR{1322960}

\bibitem{Ei} D.~Eisenbud, 
\href{https://dx.doi.org/10.1007/b137572}%
{\em The geometry of syzygies},  
a second course in commutative algebra and algebraic geometry, 
Grad. Texts in Math., vol.~229, Springer-Verlag, New York, 2005.
\MR{2103875}

\bibitem{EN} D.~Eisenbud, W.~Neumann, 
\href{https://press.princeton.edu/titles/2356.html}%
{\em Three-dimensional link theory and invariants of plane curve
singularities}, Annals of Math. Studies, vol.~110, Princeton 
University Press, Princeton, NJ, 1985.
\MR{0817982} 

\bibitem{FR} M. Falk, R. Randell,
\href{https://dx.doi.org/10.1007/BF01394780}%
{\em The lower central series of a fiber-type arrangement},
Invent. Math. \textbf{82} (1985), no.~1, 77--88. 
\MR{808110}

\bibitem{GP} J.~Guaschi, C.M.~Pereiro, 
\href{https://doi.org/10.1016/j.jpaa.2020.106309}%
{\em Lower central and derived series of semi-direct products, 
and applications to surface braid groups}, J. Pure Appl. Algebra 
\textbf{224} (2020), no.~7, 106309, 39 pp. 
\MR{4058243}

\bibitem{Hain} R.~Hain, 
\href{https://doi.org/10.2969/aspm/05210309}%
{\em Relative weight filtrations on completions of mapping 
class groups}, in: Groups of diffeomorphisms, 309--368, 
Adv. Stud. Pure Math., vol.~52, Math. Soc. Japan, Tokyo, 2008. 
\MR{2509715}

\bibitem{Hall} P.~Hall, 
\href{https://doi.org/10.1112/plms/s2-36.1.29}%
{\em A contribution to the theory of groups of prime-power order}, 
Proc. Lond. Math. Soc., II. Ser. \textbf{36} (1933), 29--95.
\Zbl{0007.29102}

\bibitem{Ha05} S.L.~Harvey, 
\href{https://doi.org/10.1016/j.top.2005.03.001}%
{\em Higher-order polynomial invariants of $3$-manifolds giving lower 
bounds for the Thurston norm}, Topology \textbf{44} (2005), no.~5, 895--945. 
\MR{2153977}

\bibitem{HHR} W.~Herfort, K.H.~Hofmann, F.G.~Russo, 
\href{https://doi.org/10.1515/9783110599190}%
{\em Periodic locally compact groups}, De Gruyter Stud. Math., vol. 71. 
De Gruyter, Berlin, 2019.
\MR{3932104}

\bibitem{HS} P.J.~Hilton, U.~Stammbach,
\href{https://dx.doi.org/10.1007/978-1-4419-8566-8}%
{\emph{A course in homological algebra}}, Second ed., Graduate 
Texts in Mathematics, vol.~4, Springer-Verlag, New York, 1997. 
\MR{1438546}  

\bibitem{Hi97} E.~Hironaka,
\href{https://doi.org/10.5802/aif.1573}%
{\em Alexander stratifications of character varieties}, Ann. Inst. 
Fourier  (Grenoble) \textbf{47} (1997), no.~2, 555--583.
\MR{1450425}

\bibitem{Karpilovsky} G.~Karpilovsky, 
{\em The Schur multiplier}, London Math. Soc. Monogr. (N.S.), 
vol.~2, Clarendon Press, Oxford Univ. Press, New York, 1987
\MR{1200015}

\bibitem{Lk09}  M.~Lackenby, 
\href{https://doi.org/10.1112/plms/pdn032}%
{\em New lower bounds on subgroup growth and homology growth}, 
Proc. Lond. Math. Soc. (3) \textbf{98} (2009), no.~2, 271--297. 
\MR{2481949}

\bibitem{Lk10}  M.~Lackenby, 
\href{https://dx.doi.org/10.1016/j.jalgebra.2010.07.012}%
{\em Detecting large groups}, 
J. Algebra \textbf{324} (2010), no.~10, 2636--2657. 
\MR{2725193}

\bibitem{Lazard} M.~Lazard, 
\href{https://doi.org/10.24033/asens.1021}%
{\em Sur les groupes nilpotents et les anneaux de {L}ie}, 
Ann. Sci. \'{E}cole Norm. Sup. (3) \textbf{71} (1954), 101--190.
\MR{0088496}

\bibitem{Leedham-McKay} C.R.~Leedham-Green, S.~McKay, 
{\em The structure of groups of prime power order}, 
London Math. Soc. Monogr. (N.S.), vol.~27. 
Oxford University Press, Oxford, 2002. 
\MR{1918951}

\bibitem{LR} J.C.~Lennox, D.J.S.~Robinson, 
\href{https://doi.org/10.1093/acprof:oso/9780198507284.001.0001}%
{\em The theory of infinite soluble groups}, 
The Clarendon Press, Oxford University Press, Oxford, 2004.
\MR{2093872}

\bibitem{Li} A.~Libgober,
\href{https://dx.doi.org/10.1016/0166-8641(92)90137-O}%
{\em On the homology of finite abelian coverings},
Topology. Appl. \textbf{43} (1992), 157--166.
\MR{1152316}

\bibitem{MKS} W.~Magnus, A.~Karrass, D.~Solitar,
{\em Combinatorial group theory}, Second revised edition, 
Dover, New~York, 1976.
\MR{0422434}

\bibitem{Markl-Papadima} M.~Markl, S.~Papadima,
\href{https://doi.org/10.5802/aif.1315}%
{\emph{Homotopy {L}ie algebras and fundamental groups via 
deformation theory}}, Ann. Inst.  Fourier (Grenoble) \textbf{42} 
(1992), no.~4, 905--935. 
\MR{1196099}

\bibitem{Ms-80}  W.~S.~Massey,
\href{https://dx.doi.org/10.1215/S0012-7094-80-04724-9}%
{\em Completion of link modules}, Duke Math. J.
\textbf{47} (1980), no.~2, 399--420.
\MR{0575904}

\bibitem{Mass} G.~Massuyeau, 
\href{https://dx.doi.org/10.1112/jlms/jdm034}%
{\em Finite-type invariants of $3$-manifolds and the dimension subgroup 
problem}, J. Lond. Math. Soc. (2) \textbf{75} (2007), no.~3, 791–811. 
\MR{2352736}

\bibitem{MS-top} D.~Matei, A.I.~Suciu,
\href{http://dx.doi.org/10.1016/S0040-9383(98)00058-5}%
{\em Homotopy types of complements of $2$-arrangements in $\R^4$}, 
Topology \textbf{39} (2000), no.~1, 61--88. 
\MR{1710992} 

\bibitem{MS00} D.~Matei, A.I.~Suciu,
\href{https://dx.doi.org/10.2969/ASPM/02710185}%
{\em Cohomology rings and nilpotent quotients of real 
and complex arrangements}, Arrangements---{T}okyo 1998, 
Adv. Stud. Pure Math., vol.~27, Kinokuniya, Tokyo, 2000, pp.~185--215.
\MR{1796900}

\bibitem{MS-imrn}  D.~Matei, A.I.~Suciu,
\href{https://dx.doi.org/10.1155/S107379280210907X}%
{\em Hall invariants, homology of subgroups, and characteristic 
varieties}, Internat. Math. Res. Notices \textbf{2002} (2002), 
no.~9, 465--503.
\MR{1884468}

\bibitem{PS-imrn04} S.~Papadima, A.I.~Suciu, 
\href{https://dx.doi.org/10.1155/S1073792804132017}%
{\em Chen {L}ie algebras}, Int. Math. Res. Not. 
\textbf{2004} (2004), no.~21, 1057--1086. 
\MR{2037049}

\bibitem{PS-mathann06} S.~Papadima, A.I.~Suciu,
\href{https://dx.doi.org/10.1007/s00208-005-0704-9}%
{\emph{Algebraic invariants for right-angled {A}rtin groups}}, 
Math. Ann. \textbf{334} (2006), no.~3, 533--555. 
\MR{2207874}

\bibitem{PS-jlms07} S.~Papadima, A.I.~Suciu,
\href{https://doi.org/10.1112/jlms/jdm045}%
{\em Algebraic invariants for Bestvina-Brady groups},   
J. London Math. Soc. \textbf{76} (2007), no.~2, 273--292. 
\MR{2363416}

\bibitem{PS-adv09} S.~Papadima, A.I.~Suciu,
\href{https://dx.doi.org/10.1016/j.aim.2008.09.008}%
{\em Toric complexes and Artin kernels}, 
Adv. Math. \textbf{220} (2009), no.~2, 441--477. 
\MR{2466422}

\bibitem{PS-formal} S.~Papadima, A.~Suciu,
\href{https://www.jstor.org/stable/43679144}%
{\em Geometric and algebraic aspects of $1$-formality}, 
Bull. Math. Soc. Sci. Math. Roumanie \textbf{52} (2009), 
no.~3, 355--375. 
\MR{2554494}

\bibitem{PS-plms10} S.~Papadima, A.I.~Suciu,
\href{https://dx.doi.org/10.1112/plms/pdp045}%
{\em Bieri--{N}eumann--{S}trebel--{R}enz invariants and 
homology jumping loci}, Proc. London Math.~Soc. 
\textbf{100} (2010), no.~3, 795--834.  
\MR{2640291} 

\bibitem{PS-jtop} S.~Papadima, A.I.~Suciu,
\href{https://dx.doi.org/10.1112/jtopol/jts023}%
{\em Homological finiteness in the {J}ohnson filtration of 
the automorphism group of a free group}, 
J. Topol. \textbf{5} (2012), no.~4,  909--944.
\MR{3001315}
 
\bibitem{PS-mrl} S.~Papadima, A.I.~Suciu,
\href{https://dx.doi.org/10.4310/MRL.2014.v21.n4.a13}%
{\em Jump loci in the equivariant spectral sequence}, 
Math. Res. Lett. \textbf{21} (2014), no.~4, 863--883.
\MR{3275650}

\bibitem{PS-crelle} S.~Papadima, A.I.~Suciu,
\href{https://dx.doi.org/10.1515/crelle-2013-0073}%
{\emph{Vanishing resonance and representations of {L}ie algebras}},
J. Reine Angew. Math. \textbf{706} (2015), 83--101.
\MR{3393364}

\bibitem{Paris} L.~Paris, 
\href{https://doi.org/10.1090/S0002-9947-08-04573-X}%
{\em Residual $p$ properties of mapping class groups and surface groups}, 
Trans. Amer. Math. Soc. \textbf{361} (2009), no.~5, 2487--2507. 
\MR{2471926}

\bibitem{Passman}  D.S.~Passman, 
{\em The algebraic structure of group rings}, 
Pure and Applied Mathematics, Wiley-Interscience, 
New York--London--Sydney, 1977.
\MR{0470211}

\bibitem{Quillen68} D.G.~Quillen, 
\href{https://dx.doi.org/10.1016/0021-8693(68)90069-0}%
{\emph{On the associated graded ring of a group ring}}, 
J. Algebra \textbf{10} (1968), 411--418. 
\MR{0231919}

\bibitem{Qu} D.~Quillen, 
\href{https://dx.doi.org/10.2307/1970725}%
{\em Rational homotopy theory}, Ann. of Math. (2) \textbf{90}
(1969), 205--295. 
\MR{0258031}

\bibitem{SW92} P.B.~Shalen, P.~Wagreich,
\href{https://dx.doi.org/10.2307/2154149}%
{\em Growth rates, $\Z_p$-homology, and volumes of hyperbolic 
$3$-manifolds}, Trans. Amer. Math. Soc. \textbf{331} (1992), 
no.~2, 895--917. 
\MR{1156298}

\bibitem{St} J.~Stallings,
 \href{https://dx.doi.org/10.1016/0021-8693(65)90017-7}%
{\emph{Homology and central series of groups}},
J. Algebra \textbf{2} (1965), no.~2, 170--181.
\MR{0175956}

\bibitem{St83} J.R.~Stallings,
 \href{https://dx.doi.org/10.1007/BF01162003}%
{\em Surfaces in three-manifolds and nonsingular equations 
in groups}, Math. Z. \textbf{184} (1983), no.~1, 1--17. 
\MR{0711725}

\bibitem{Su-conm} A.I.~Suciu,
\href{https://dx.doi.org/10.1090/conm/276/04510}%
{\emph{Fundamental groups of line arrangements: enumerative aspects}},
in: {\em Advances in algebraic geometry motivated by physics} 
({L}owell, {MA}, 2000), 43--79, Contemp. Math., vol. 276, 
Amer. Math. Soc., Providence, RI, 2001.
\MR{1837109}

\bibitem{Su-imrn} A.I.~Suciu,
\href{https://dx.doi.org/10.1093/imrn/rns246}%
{\em Characteristic varieties and Betti numbers of free
abelian covers}, Int. Math. Res. Notices \textbf{2014} 
(2014), no. 4, 1063--1124.
\MR{3168402}

\bibitem{Su-toul} A.I.~Suciu,
\href{https://dx.doi.org/10.5802/afst.1412}%
{\em Hyperplane arrangements and {M}ilnor fibrations}, 
Ann. Fac. Sci. Toulouse Math. \textbf{23} (2014), no.~2, 417--481.  
\MR{3205599}

\bibitem{Su-revroum} A.I.~Suciu,
\href{http://imar.ro/journals/Revue_Mathematique/pdfs/2017/1/10.pdf}%
{\em On the topology of Milnor fibrations of hyperplane arrangements},  
Rev. Roumaine Math. Pures Appl. \textbf{62} (2017), no.~1, 191--215.
\MR{3626439}

\bibitem{Su-edinb} A.I.~Suciu,
\href{https://doi.org/10.1017/prm.2019.55}%
{\em Poincar\'{e} duality and resonance varieties}, 
Proc. Roy. Soc. Edinburgh Sect. A
\textbf{150} (2020), nr.~6, 3001--3027. 
\MR{4190099}

\bibitem{Su-lcs} A.I.~Suciu,
{\em Lower central series and split extensions}, 
\arxiv{2105.14129v2}

\bibitem{Su-mfmono} A.I.~Suciu,
{\em Milnor fibrations of arrangements with trivial algebraic monodromy}, 
draft (2021).

\bibitem{SW-mz}  A.I.~Suciu, H.~Wang, 
\href{https://dx.doi.org/10.1007/s00209-016-1811-x}%
{\em Pure virtual braids, resonance, and formality}, 
Math. Z.  \textbf{286} (2017), no.~3--4, 1495--1524.
\MR{3671586}

\bibitem{SW-jpaa}  A.I.~Suciu, H.~Wang, 
\href{https://doi.org/10.1016/j.jpaa.2018.11.006}%
{\em Cup products, lower central series, and holonomy Lie algebras}, 
J. Pure Appl. Algebra \textbf{223} (2019), no.~8, 3359--3385.
\MR{3926216}

\bibitem{SW-forum} A.I.~Suciu, H.~Wang, 
\href{https://doi.org/10.1515/forum-2018-0098}%
{\em Formality properties of finitely generated groups and Lie 
algebras}, Forum Math. \textbf{31} (2019), no.~4, 867--905. 
\MR{3975666}

\bibitem{SW-aam} A.I.~Suciu, H.~Wang, 
\href{https://doi.org/10.1016/j.aam.2019.07.004}%
{\em Chen ranks and resonance varieties of the upper McCool groups}, 
Adv. in Appl. Math. \textbf{110} (2019), 197--234. 
\MR{3983125}

\bibitem{SW-ejm} A.I.~Suciu, H.~Wang, 
\href{https://doi.org/10.1007/s40879-019-00389-6}%
{\em Taylor expansions of groups and filtered-formality}, 
Eur. J. Math.  \textbf{6} (2020), nr.~3, 1073--1096.
\MR{4151729}

\bibitem{SYZ-pisa} A.I.~Suciu, Y.~Yang, G.~Zhao, 
\href{https://dx.doi.org/10.2422/2036-2145.201205_008}%
{\em  Homological finiteness of Abelian covers}, 
Ann. Sc. Norm. Super. Pisa Cl. Sci. \textbf{14} (2015), 
no.~1, 101--153.
\MR{3379489}

\end{thebibliography}
\end{document}